\documentclass[11pt]{article}
\usepackage{amsmath,amsthm,amssymb}
\usepackage[T1]{fontenc}
\usepackage[utf8]{inputenc}
\usepackage[dvipdf]{graphicx}
\usepackage{color}
\usepackage{epstopdf}
\usepackage{dsfont}
\usepackage{mathtools}
\usepackage{enumerate}
\usepackage{mathrsfs}
\usepackage{float}
\usepackage[normalem]{ulem}
\usepackage[colorlinks=true,citecolor=red,linkcolor=db,urlcolor=blue,pdfstartview=FitH]{hyperref}
\usepackage[numbers]{natbib}
\usepackage{booktabs}
\usepackage{subcaption}
\mathtoolsset{showonlyrefs=true}
\numberwithin{equation}{section}

\definecolor{db}{RGB}{0, 0, 130}
\definecolor{rp}{rgb}{0.25, 0, 0.75}
\definecolor{dg}{rgb}{0, 0.6, 0}

\newtheorem{theorem}{Theorem}[section]

\newtheorem{definition}{Definition}[section]
\newtheorem{proposition}{Proposition}[section]
\newtheorem{corollary}{Corollary}[section]

\newtheorem{assumption}{Assumption}[section]
\newtheorem{lemma}{Lemma}[section]

\newtheorem{remark}{Remark}[section]

\def\Cc{\mathcal C}
\def\Fc{\mathcal F}
\def\Gc{\mathcal G}

\def\Ic{\mathcal I}
\def\Lc{\mathcal L}
\def\Pc{\mathcal P}

\def\Wc{\mathcal W}

\def\E{\mathbb{E}}
\def\F{\mathbb{F}}

\def\I{\mathbb I}

\def\P{\mathbb P}
\def\R{\mathbb R}

\def\Sbf{\mathbf S}

\def\x{\times}
\def\Om{\Omega}

\def\eps{\varepsilon}

\title{Mean-field optimal stopping with endogenous quantile cutoffs%
\footnote{\emph{Keywords:} mean-field optimal stopping, McKean--Vlasov dynamics, endogenous quantile cutoff, nonconvex survival constraint, strong--weak equivalence, dynamic programming, propagation of chaos.}%
\footnote{\emph{2020 Mathematics Subject Classification:} Primary 60G40, 49N80; secondary 49L20, 60K35, 93E20.}}
\author{Erhan Bayraktar\footnote{Department of Mathematics, University of Michigan. Email: erhan@umich.edu Erhan Bayraktar gratefully acknowledges support from the National Science Foundation (NSF) under grant DMS-2507940.}
        \and Ibrahim Ekren\footnote{ Department of Mathematics, University of Michigan. Email: iekren@umich.edu. I. Ekren gratefully acknowledges support from the National Science Foundation (NSF) under grant DMS-2406240
} 
        \and Xihao He\footnote{Department of Mathematics, University of Southern California. Email: \mbox{xihaohe@usc.edu}.}}
\date{}

\begin{document}
\maketitle

\begin{abstract}
    We study a mean-field optimal stopping problem with an endogenous population-level shutdown. All remaining agents stop when the survival mass falls below a prescribed threshold.
    We recast the discontinuous objective as the singular, nonconvex constraint that the survival mass lie in $\{0\}\cup[\alpha,1]$. We prove the equivalence of strong and weak values via an approximation and the existence of an optimal rule via compactness and penalization. We also prove a dynamic programming principle. The value is continuous away from the critical boundary but may be discontinuous at the boundary itself. Under strict initial feasibility, finite-population values converge to the mean-field value. In the same regime, the laws of near-optimal empirical measures are tight and every mean-field optimizer admits a recovery sequence. At the threshold, however, finite-population convergence may fail.
\end{abstract}

\tableofcontents
\section{Introduction}
\label{sec:introduction}

Classical optimal stopping asks when to terminate a stochastic activity so as to maximize its expected reward; see, for example, \cite{PeskirShiryaev2006}. In the present work we consider an optimal stopping problem where the stopping happens either at the time chosen by the representative agent or at an endogenous deterministic cutoff determined by the distribution of that stopping time. The horizon therefore depends on the law of the policy being optimized rather than being fixed in advance.

More precisely, let $\tau$ be a representative stopping time and denote
\[
    F_\tau(s):=\mathbb P(\tau\le s).
\]
Fix a minimum survival proportion $\alpha\in(0,1]$. We denote the upper generalized quantile
\begin{equation}
\label{eq:intro_quantile_cutoff}
    q_\alpha(\tau)
    :=\inf\big\{s\in[0,T]:F_\tau(s)>1-\alpha\big\}
     =\inf\big\{s\in[0,T]:\mathbb P(\tau>s)<\alpha\big\},
\end{equation}
with the convention $\inf\varnothing=T+1$. Thus $q_\alpha(\tau)$ is the first time at which the surviving mass falls strictly below $\alpha$; the strict inequality matters when the distribution of $\tau$ has atoms. For a gain process $G$, in its simplest form, the problem we study is
\begin{equation}
\label{eq:intro_quantile_problem}
    \sup_\tau\,\mathbb E\big[G_{\tau\wedge q_\alpha(\tau)}\big].
\end{equation}
Once the law of $\tau$ is fixed, the cutoff in \eqref{eq:intro_quantile_problem} is deterministic. It nevertheless changes with the stopping rule, unlike the exogenous terminal date in a standard stopping problem.

We formulate the problem \eqref{eq:intro_quantile_problem} in a mean-field setting under a stopped McKean--Vlasov diffusion. Let $I_s\in\{0,1\}$ be the survival indicator, $X_s$ be the state, and write
    $\mu_s:=\mathcal L(X_s,I_s).$
Both the state dynamics and the running reward may depend on the population law $\mu_s$. For an admissible stopping rule $\mathbb P$, the objective to optimize is
\begin{equation}
\label{eq:intro_mean_field_objective}
    J_t(\mathbb P)
    :=\mathbb E^{\mathbb P}\left[
       \int_t^T
       f(s,X_s,\mu_s)I_s
       \mathbf 1_{[\alpha,1]}(\mathbb E[I_s])\,ds
    \right].
\end{equation}
In our framework, the reward remains active while at least an $\alpha$-fraction of the population survives and vanishes once the survival mass drops below that level. Thus, the system may operate at or above the viability threshold or shut down completely, but it cannot continue with a strictly positive subthreshold mass.

\paragraph{Background and motivation.}
Consider a population of $N$ agents. Let $\tau^1,\ldots,\tau^N$ be their stopping times, and define the fraction still active at time $s$ by
\[
    q_s^N:=\frac1N\sum_{i=1}^N\mathbf 1_{\{\tau^i>s\}}.
\]
The system-wide shutdown time is
\begin{equation}
\label{eq:intro_finite_cutoff}
    \bar\tau^N
    :=\inf\big\{s\in[0,T]:q_s^N<\alpha\big\}.
\end{equation}
Agent $i$ remains active until $\tau^i\wedge\bar\tau^N$. In a symmetric large population, $q_s^N$ is approximated by $\mathbb P(\tau>s)$, and $\bar\tau^N$ formally converges to the quantile cutoff in \eqref{eq:intro_quantile_cutoff}. The representative-agent problem thus describes the mean-field limit of a system that permits local exits only while the remaining population is collectively viable.

Our primary motivation is distributed quickest change detection. In this problem, sensor $i$ observes a local data stream and raises an irreversible alarm at time $\tau^i$. Under a fusion rule announced in advance and known to every sensor, the center declares a system-wide event at a prescribed empirical quantile of the local alarm times. Equivalently, it stops when the fraction $q_s^N$ of sensors that have not reported falls below a threshold. Equation \eqref{eq:intro_finite_cutoff} is the finite-sensor rule, and \eqref{eq:intro_quantile_cutoff} is its mean-field counterpart. We treat the choice of local alarm policies as a centralized design problem, not as a game among strategic sensors.

Several strands of multisource quickest detection are closely related. Bayraktar and Poor study a two-source problem in which one alarm detects the minimum of two independent latent Poisson disorder times in \cite{BayraktarPoor2007}. When the two source models and priors coincide, their Remark~4.1 also considers the minimum of two independent copies of the optimal one-source alarm. Bayraktar and Lai consider a distributed sensor network in which multiple sensors report to a control center and some may be Byzantine in \cite{BayraktarLai2015}. Fellouris, Bayraktar, and Lai analyze $L$-th-alarm, voting, and Low-Sum-CUSUM procedures. Their $L$-th-alarm rule stops at an order statistic of the local alarm times and is a direct finite-sensor analogue of our quantile mechanism in \cite{FellourisBayraktarLai2018}. Huang, Huang, and Lin formulate a multi-hypothesis Byzantine detection problem and prove first-order asymptotic optimality of the simultaneous $d$-th-alarm rule under sufficient communication bandwidth. They also analyze a more communication-efficient multi-shot rule, which is asymptotically optimal in certain special cases in \cite{HuangHuangLin2021}. Mei develops scalable multistream procedures by aggregating local CUSUM statistics in \cite{Mei2010}. Those fusion-based models prescribe the local detection statistics and aggregation rules. Here the local stopping law is optimized and itself determines the mean-field quantile cutoff.

A second illustration comes from epidemic control. Consider a federation of cities or states, and let $\tau^i$ be the time at which jurisdiction $i$ is classified as high risk according to a local indicator such as incidence, test positivity, or hospital pressure. When only a few jurisdictions meet the high-risk criteria, an authority might restrict travel to and from those jurisdictions while allowing travel elsewhere. Once the proportion of restricted jurisdictions exceeds $1-\alpha$, maintaining a sufficiently connected network may no longer be feasible, and the authority instead suspends nonessential inter-jurisdictional travel throughout the system at $\bar\tau^N$.

Reliability and critical-infrastructure systems offer another instance of the same mechanism. Individual machines or network components may be removed after local failures, while the full system shuts down or enters a safe mode once the proportion of operational units falls below $\alpha$. This is a dynamic analogue of a $k$-out-of-$N$ reliability system \cite{BolandProschan1983}. In both the sensor and reliability examples, a population statistic directly triggers the aggregate action.

\paragraph{Recent developments.}
Two related but conceptually different lines of work have emerged in mean-field stopping. The first concerns mean-field games, in which agents choose their stopping rules strategically and the population distribution enters the Nash equilibrium condition. This line includes the timing-game and bank-run models of Carmona, Delarue, and Lacker \cite{CarmonaDelarueLacker2017} and Nutz \cite{Nutz2018}, Bertucci's obstacle-problem approach \cite{Bertucci2018}, and the relaxed occupation-measure formulation of Bouveret, Dumitrescu, and Tankov \cite{BouveretDumitrescuTankov2020}. More recently, Possama{\"\i} and Talbi have developed weak equilibria and a master-equation perspective in \cite{PossamaiTalbi2025}. A complementary probabilistic approach characterizes randomized-strategy equilibria through coupled reflected forward--backward McKean--Vlasov stochastic differential equations in \cite{CossoDAndolfiDumitrescu2026}. In a different forward--backward formulation, motivated by recursive mean-field Dynkin games, Bayraktar and Rzymowski prove short-time $L^p$ well-posedness and, under additional monotonicity conditions, global $L^2$ well-posedness for fully coupled mean-field doubly reflected forward--backward stochastic differential equations with two optional barriers in \cite{BayraktarRzymowski2026}. Nutz, San Martin, and Tan show that some equilibria in a mean-field timing game cannot arise as limits of finite-player equilibria in \cite{NutzSanMartinTan2020}. He, Tan, and Zou use a mean-field Bank--El Karoui representation to study timing games with common noise in \cite{HeTanZou2025}. The bank-run model of Carmona, Delarue, and Lacker also features default triggered by aggregate withdrawals. All of these models, however, concern noncooperative equilibria. We study instead a centralized problem: a planner maximizes an aggregate criterion under a fixed nonconvex viability constraint.

The second line treats centralized McKean--Vlasov stopping. Belomestny and Schoenmakers develop a particle-regression method for numerical approximation in \cite{BelomestnySchoenmakers2020}. In discrete time and finite state space, Magnino, Zhu, and Lauri{\`e}re establish a mean-field approximation for cooperative many-agent stopping and develop two deep-learning algorithms, one based on dynamic programming and backward induction in \cite{MagninoZhuLauriere2025}. Agram and {\O}ksendal derive verification conditions for conditional McKean--Vlasov optimal stopping with common noise and jumps in \cite{AgramOksendal2024}. Talbi, Touzi, and Zhang formulate the problem weakly through the joint law of the stopped state and survival process, prove a dynamic programming principle, and derive an obstacle equation on Wasserstein space in \cite{TalbiTouziZhang2023}; a companion paper develops the viscosity theory in \cite{TalbiTouziZhangViscosity2023}. For an uncontrolled state, Cosso and Perelli obtain a dynamic programming principle for an extended value function and prove, under suitable assumptions, that it is a viscosity solution of the associated second-order variational inequality in \cite{CossoPerelli2025}. Talbi, Touzi, and Zhang subsequently prove convergence of finite-population stopping problems to their mean-field limit in \cite{TalbiTouziZhang2024}. Under an $H$-hypothesis-type condition, He extends the limit theory and strong--weak comparison to non-Markovian dynamics with common noise in \cite{He2025}, while Talbi and Touzi use It{\^o}--Wentzell formulas to derive dynamic programming equations for mean-field stopping with common noise in \cite{TalbiTouzi2026}. Cardaliaguet, Jackson, and Souganidis study the closely related problem of mean-field control with particle removal. They characterize the limiting value by an infinite-dimensional quasi-variational inequality and prove convergence from finite populations in \cite{CardaliaguetJacksonSouganidis2026}. In a related absorption model, the same authors formulate the mean-field problem on subprobability measures and prove convergence of the finite-particle values in \cite{CardaliaguetJacksonSouganidisAbsorption2026}.

Interacting default models consider McKean--Vlasov equations with hitting-time feedback, including questions of blow-up and global solvability \cite{BayraktarGuoTangZhang2024}, relate the surviving proportion in a finite default system to the limiting McKean--Vlasov default probability \cite{BayraktarGuoTangZhang2025}, and analyze the passage from smooth to singular hitting-time feedback under common noise \cite{HamblyEtAl2025}. In these models, state hitting generates the stopping; here the stopping rules are optimized, and their endogenous quantile triggers a common shutdown.

The constrained-stopping literature either prescribes the distribution of the stopping time \cite{BayraktarMiller2019,BeiglbockEtAl2018} or imposes equality and inequality constraints on expected costs in a weak formulation \cite{BayraktarYao2024Stopping,BayraktarYao2024ControlStopping}. K{\"a}llblad develops a dynamic programming approach to general distribution constraints using measure-valued martingales \cite{Kallblad2022}. By contrast, the stopping rule in \eqref{eq:intro_quantile_problem} generates the quantile endogenously.

Our analysis builds most directly on Talbi, Touzi, and Zhang \cite{TalbiTouziZhang2023,TalbiTouziZhangViscosity2023,TalbiTouziZhang2024} and He \cite{He2025}. The particle-removal model of Cardaliaguet, Jackson, and Souganidis \cite{CardaliaguetJacksonSouganidis2026} is also close, although it includes an additional continuous control and a stopping penalty. Together, these papers provide the foundations for weak formulations, dynamic programming, viscosity theory, strong--weak comparison, and finite-population limits in mean-field stopping and control. Our new feature is the endogenous viability cutoff. It confines the survival mass to the nonconvex set $\{0\}\cup[\alpha,1]$ and creates a singular boundary at $p=\alpha$.

\paragraph{Main results and proof ideas.}
We begin by reducing the discontinuous, thresholded objective to a constrained mean-field stopping problem. Given an admissible rule, we stop all surviving particles at the first deterministic time when the survival mass would fall below $\alpha$. The modified system agrees with the original system before that time; afterward, the threshold factor in \eqref{eq:intro_mean_field_objective} is already zero. This operation leaves the value unchanged and gives a survival flow satisfying
\[
    \mathbb E[I_s]\in\{0\}\cup[\alpha,1].
\]
The singularity is thereby isolated in a deterministic population constraint.

Our main probabilistic result is the equivalence of the strong and weak formulations. For every weak stopping rule, the law of the state, Brownian increments, and survival process from time $t$ onward can be approximated in Wasserstein distance by the corresponding laws induced by strong stopping rules; consequently, the two values coincide. The construction has three steps. We first discretize a weak stopping rule on a fine time grid without violating the population constraint. We then reserve a short Brownian interval immediately after the initial time for randomization. Normalized increments on this interval produce independent uniform variables, which recursively reproduce the conditional Bernoulli laws of the discrete stopping decisions. In this way, the weak rule is approximated by a rule adapted to the strong Brownian information. Lipschitz estimates for the McKean--Vlasov equation give convergence of the projected laws, and lower semicontinuity of the reward functional turns the density result into equality of the values.

To prove that the constrained weak stopping problem admits an optimizer, we work with laws on the path space from time $t$ onward and express admissibility through a martingale problem. Moment estimates, compactness of the survival-path space, and stability of the martingale identities make the set of projected laws induced by weak stopping rules compact in the Wasserstein topology. We measure violations of the constraint by
\[
    H_t(Q):=\int_t^T(\alpha-\mathbb E^Q[I_s])^+\mathbb E^Q[I_s]\,ds\geq 0
\]
and consider the penalized objective $\mathcal F_t(Q)-\varepsilon^{-1}H_t(Q)$. Compactness and continuity give maximizers of the penalized problems, convergence of their values, and an optimal constrained weak rule. Additionally, concatenation, combined with a monotonicity argument for instantaneous stopping at the concatenation time, yields the dynamic programming principle.

The regularity theory reflects the singular nature of the threshold. When the running reward is continuous, the value is jointly continuous in time and the initial law whenever the initial survival mass differs from $\alpha$; it may genuinely be discontinuous at the threshold. The proof couples nearby initial laws and changes only a vanishing fraction of their survival paths to restore feasibility. An ordinary state coupling is not enough, because it can leave a positive survival mass below $\alpha$. We also compare open and closed threshold conventions and describe the one-sided dependence of the value on $\alpha$.

We then obtain the limit theory for the mean field problem. 
On the one hand, we prove tightness of the random empirical path laws, identify every subsequential limit through a martingale problem, and use uniform integrability to pass to the limit in the reward. 
On the other hand, we approximate a strong mean-field rule by finite-valued nonanticipative rules, apply these rules independently to the $N$-particle system, and alter the stopping decisions of a small set of particles so that the empirical viability constraint holds. The values of the $N$-particle system converge and optimizers are recovered when the limiting initial survival mass is strictly above $\alpha$. Below the threshold the limiting value is zero, whereas at the threshold convergence may fail without an additional one-sided feasibility condition.

Our contributions are as follows.
\begin{enumerate}
\item We formulate a quantile-truncated mean-field optimal stopping problem in which an endogenous population cutoff turns individual stopping into a system-wide shutdown rule.

\item We identify the equivalent singular constraint $p_s\in\{0\}\cup[\alpha,1]$ and prove strong--weak equivalence by constraint-preserving discretization and Brownian purification.

\item Using compactness and penalization, we establish existence of an optimal weak rule and prove a dynamic programming principle. We also show that the only possible interior discontinuity of the value occurs at the viability boundary $p=\alpha$.

\item We prove a finite-population limit theorem. It includes tightness and identification of near-optimal empirical laws, a recovery construction for mean-field optimizers, the subthreshold regime, and a counterexample that explains the obstruction at the exact threshold.

\item We compare the open and closed threshold conventions and study how the value depends on the threshold.
\end{enumerate}

\paragraph{Organization of the paper.}
Section~2 sets up the canonical spaces, the strong and weak stopping rules, the thresholded and constrained formulations, and their connection with the quantile-truncated problem. Section~3 proves strong--weak equivalence. Section~4 gives the compactness and penalization arguments for existence, followed by the dynamic programming principle. Section~5 examines regularity of the value and the singular boundary at survival mass $\alpha$. Section~6 treats the finite-population limit and the recovery of optimal mean-field laws. Section~7 compares open and closed threshold conventions and studies how the value depends on the threshold.

\section{Preliminaries and problem formulation}
\subsection{Notation}
\noindent$\mathrm{(i)}$
Fix a finite horizon $T\in(0,+\infty)$ and dimensions $n,d\in\mathbb N$. For a measurable space
$(E,\mathcal E)$, let $\Pc(E)$ be the set of probability measures on
$(E,\mathcal E)$. If $\P\in\Pc(E)$ and $Y$ is a random variable taking values in a measurable space $F$, write
\[
    \Lc^{\P}(Y):=\P\circ Y^{-1}
\]
for the law of $Y$ under $\P$.

\vspace{0.5em}
\noindent$\mathrm{(ii)}$
For a Polish space $(E,d_E)$, let $C([a,b],E)$ denote the space of
continuous $E$-valued functions on $[a,b]$, equipped with the uniform metric,
and let $C_b(E)$ denote the bounded continuous real-valued functions on $E$.
Write $\Pc_2(E)$ for the set of $\rho\in\Pc(E)$ such that
\[
    \int_E d_E(e,e_0)^2\,\rho(de)<+\infty
\]
for one, and hence every, $e_0\in E$. For $\rho_1,\rho_2\in\Pc_2(E)$, the
Wasserstein $2$-distance is
\[
    \Wc_2(\rho_1,\rho_2)
    :=
    \left(
        \inf_{\gamma\in\Gamma(\rho_1,\rho_2)}
        \int_{E\times E}d_E(e_1,e_2)^2\,\gamma(de_1,de_2)
    \right)^{1/2},
\]
where $\Gamma(\rho_1,\rho_2)$ is the set of couplings of $\rho_1$ and
$\rho_2$.

\vspace{0.5em}
\noindent$\mathrm{(iii)}$
For $s\in[0,T]$, let $\I([s,T])$ be the set of $\{0,1\}$-valued decreasing càdlàg functions
$i$ on $\{s-\}\cup[s,T]$. 
In particular,
either $i_{s-}=0$ and
$i_r=0$ for every $r\in[s,T]$, or $i_{s-}=1$ and there exists a unique
$\vartheta_s(i)\in[s,T]\cup\{T+1\}$ such that
\[
    i_r=\mathbf 1_{\{r<\vartheta_s(i)\}},
    \qquad r\in[s,T].
\]
The convention $\vartheta_s(i)=T+1$ means that the particle survives through
$T$, while $\vartheta_s(i)=T$ means that it stops exactly at $T$. If
$i_{s-}=0$, set $\vartheta_s(i):=s$. We equip $\I([s,T])$ with the metric
\begin{equation}\label{eq:metric_I}
    d_{\I,s}(i,j)
    :=|i_{s-}-j_{s-}|+|\vartheta_s(i)-\vartheta_s(j)|.
\end{equation}
The map
   $ i\longmapsto (i_{s-},\vartheta_s(i))$
is a bijection from $\I([s,T])$ onto the compact set
    $\{(0,s)\}\cup\big(\{1\}\times[s,T+1]\big).$
Consequently, $\I([s,T])$ is a compact Polish space under
$d_{\I,s}$. This metric distinguishes both the value at $s-$ and the cases
``stop at $T$'' and ``survive through $T$''.

\vspace{0.5em}
\noindent$\mathrm{(iv)}$
Let $\Sbf:=\R^n\times\{0,1\}$ and equip it with the metric
\[
    d_{\Sbf}\big((x,i),(x',i')\big)^2
    :=|x-x'|^2+|i-i'|^2.
\]
For $m\in\Pc_2(\Sbf)$, define
\begin{equation}\label{eq:M2_definition}
    M_2(m)
    :=\int_{\Sbf}\big(|x|^2+i^2\big)\,m(dx,di).
\end{equation}
For $k\in\mathbb N$, write $\Cc^k:=C([0,T],\R^k)$. For
$t\in[0,T]$, let
\[
    C_0([t,T],\R^d)
    :=\{w\in C([t,T],\R^d):w_t=0\}
\]
and define the path space from time $t$ onward
\begin{equation}\label{eq:path_space_from_t}
    \Om_t
    :=C([t,T],\R^n)
      \times C_0([t,T],\R^d)
      \times\I([t,T]).
\end{equation}
Equip $\Om_t$ with
\begin{equation}\label{eq:path_metric_from_t}
\begin{split}
    d_t\big((x,w,i),(x',w',i')\big)^2
    :=\;&\|x-x'\|_{\infty,[t,T]}^2
       +\|w-w'\|_{\infty,[t,T]}^2
       +d_{\I,t}(i,i')^2.
\end{split}
\end{equation}
The Wasserstein distance on $\Pc_2(\Om_t)$ induced by $d_t$ is denoted by
$\Wc_{2,t}$.

\begin{remark}
    The metric $d_{\I,s}$ controls the time-integrated distance between
    survival paths. Indeed, for $i,j\in\I([s,T])$,
    \begin{equation}\label{eq:I_L1_bound}
        \int_s^T |i_r-j_r|\,dr
        \le |\vartheta_s(i)-\vartheta_s(j)|
        \le d_{\I,s}(i,j).
    \end{equation}
\end{remark}

\subsection{Coefficients and standing assumptions}

Fix $\alpha\in(0,1]$. The coefficients of the stopped McKean--Vlasov diffusion are
\[
    (b,\sigma):[0,T]\times\R^n\times\Pc_2(\Sbf)
    \longrightarrow
    \R^n\times\R^{n\times d},
\]
and the running reward is
\[
    f:[0,T]\times\R^n\times\Pc_2(\Sbf)\longrightarrow\R.
\]

\begin{assumption}\label{ass:coefficient_reward}
The maps $b$, $\sigma$, and $f$ are Borel measurable. Moreover, there are
constants $L,C>0$ such that the following hold.

\medskip
\noindent\textnormal{(i)} For every $s\in[0,T]$,
$x,x'\in\R^n$, and $m,m'\in\Pc_2(\Sbf)$,
\begin{equation}\label{eq:coeff_lipschitz}
\begin{split}
    |b(s,x,m)-b(s,x',m')|
      +|\sigma(s,x,m)-\sigma(s,x',m')|
      \le L\big(|x-x'|+\Wc_2(m,m')\big),
\end{split}
\end{equation}
and
\begin{equation}\label{eq:coeff_growth}
    |b(s,x,m)|+|\sigma(s,x,m)|
    \le C\big(1+|x|+M_2(m)^{1/2}\big).
\end{equation}

\medskip
\noindent\textnormal{(ii)} For every $s\in[0,T]$, the map
$(x,m)\mapsto f(s,x,m)$ is lower semicontinuous on
$\R^n\times\Pc_2(\Sbf)$, and
\begin{equation}\label{eq:f_growth}
    |f(s,x,m)|
    \le C\big(1+|x|^2+M_2(m)\big).
\end{equation}
\end{assumption}

\begin{remark}\label{rem:running_reward_only}
We consider only a running reward. The strong--weak value argument uses a
deterministic cutoff that stops all remaining particles once the survival mass
falls below $\alpha$, and this operation would generally change a terminal
reward $g(X_T,\mu_T)$. Extending these cutoff-based results to an arbitrary
terminal reward would therefore require an additional invariance or
monotonicity assumption on $g$.
\end{remark}

\subsection{Problem formulation}

Let
    $\Om:=\Cc^n\times\Cc^d\times\I([0,T])$
with its Borel $\sigma$-field $\Fc_T$, and let $(X,W,I)$ be the canonical
process. Thus, for $\omega=(x,w,i)\in\Om$,
$    X_s(\omega)=x_s,
    W_s(\omega)=w_s,
    I_s(\omega)=i_s.$
Let $\F=(\Fc_s)_{s\in[0,T]}$ be the natural filtration
$    \Fc_s
    :=\sigma\big(I_{0-},X_r,W_r,I_r:0\le r\le s\big).$
All statements about Brownian motion or adaptedness under a probability $\P$
refer to the $\P$-completion of the relevant filtration.

For $t\in[0,T]$, define the projection onto $\Om_t$
\begin{equation}\label{eq:projection_onto_Omega_t}
\begin{split}
    \Pi_t:\Om&\longrightarrow\Om_t,\\
    \Pi_t(x,w,i)
    &:=\Big(
       x_{|[t,T]},
       (w_s-w_t)_{s\in[t,T]},
       i_{|\{t-\}\cup[t,T]}
       \Big).
\end{split}
\end{equation}
For a probability $\P$ on $\Om$, write
    $\P^t:=\P\circ\Pi_t^{-1}.$
Thus $\P^t$ is the image law on $\Om_t$ of the trajectory from time $t$
onward, with the Brownian coordinate recentered at zero.

After time $t$, the strong information is generated by the current state, the
survival status immediately before $t$, and the Brownian path. Thus, for
$s\in[t,T]$, set
\begin{equation}\label{eq:strong_filtration}
    \Gc_s^t
    :=\sigma\big(X_t,I_{t-},W_r:0\le r\le s\big),
    \qquad
    \mathbb G^t:=(\Gc_s^t)_{s\in[t,T]}.
\end{equation}
We impose no strong-adaptedness condition on the control path before $t$,
since the formulation allows arbitrary pre-$t$ histories.

\begin{definition}\label{def:stopping_rule}
Let $t\in[0,T]$ and $\nu\in\Pc_2(\Sbf)$. A probability $\P$ on
$(\Om,\Fc_T)$ is called a \emph{weak stopping rule on $[t,T]$ associated with
$\nu$} if the following conditions hold:

\begin{enumerate}
\item[(i)] $W$ is a Brownian motion with respect to the $\P$-completed
filtration $\F$, and
$\Lc^{\P}(X_t,I_{t-})=\nu$;

\item[(ii)] for every $s\in[t,T]$, $\P$-a.s.,
\begin{equation}\label{eq:dynamics}
    X_s
    =X_t
     +\int_t^s b(r,X_r,\mu_r)I_r\,dr
     +\int_t^s \sigma(r,X_r,\mu_r)I_r\,dW_r,
\end{equation}
where
\begin{equation}\label{eq:mu_definition}
    \mu_r:=\Lc^{\P}(X_r,I_r),
    \qquad r\in[t,T].
\end{equation}
\end{enumerate}
The collection of weak stopping rules is denoted by $\Pc_W(t,\nu)$.

A weak stopping rule $\P\in\Pc_W(t,\nu)$ is called a \emph{strong stopping
rule} if, for every $s\in[t,T]$, $I_s$ is measurable with respect to the
$\P$-completion of $\Gc_s^t$. The collection of strong stopping rules is
denoted by $\Pc_S(t,\nu)$.
\end{definition}

\begin{lemma}\label{lem:moment_estimate}
Under Assumption \ref{ass:coefficient_reward}, there is a constant $C$,
depending only on $T$ and the constants in the assumption, such that for every
$\P\in\Pc_W(t,\nu)$,
\begin{equation}\label{eq:moment_estimate}
    \E^{\P}\left[\sup_{s\in[t,T]}|X_s|^2\right]
    \le C\left(1+\int_{\Sbf}|x|^2\,\nu(dx,di)\right).
\end{equation}
In particular, $\P^t\in\Pc_2(\Om_t)$.
\end{lemma}

\begin{remark}\label{rem:unrestricted_history}
Only the law of $(X_t,I_{t-})$ is prescribed, so admissible laws may have
different state histories on $[0,t)$. Consequently, the full canonical laws
need not form a compact family, and the laws induced by strong stopping rules
need not be dense among those induced by weak stopping rules in a Wasserstein
topology on the full path space $\Om$. We therefore state the approximation
for their images $\P^t$ on $\Om_t$, which contain exactly the coordinates
entering the dynamics and reward on $[t,T]$.
\end{remark}

For $\P\in\Pc_W(t,\nu)$, define
\begin{equation}\label{eq:objective}
    J_t(\P)
    :=
    \E^{\P}\left[
        \int_t^T
        f(s,X_s,\mu_s)I_s
        \mathbf 1_{[\alpha,1]}\big(\E^{\P}[I_s]\big)
        \,ds
    \right].
\end{equation}
The moment estimate in Lemma \ref{lem:moment_estimate} and the growth bound on $f$ ensure that
$J_t(\P)$ is finite for every $\P\in\Pc_W(t,\nu)$.
The weak and strong value functions are
\begin{equation}\label{eq:value_function}
    V_W(t,\nu)
    :=\sup_{\P\in\Pc_W(t,\nu)}J_t(\P),
    \qquad
    V_S(t,\nu)
    :=\sup_{\P\in\Pc_S(t,\nu)}J_t(\P).
\end{equation}
The set of weak optimizers is denoted by
\[
    \Pc_W^*(t,\nu)
    :=\{\P\in\Pc_W(t,\nu):J_t(\P)=V_W(t,\nu)\};
\]
the notation itself does not assert that an optimizer exists.

Define the constrained classes
\begin{equation}\label{eq:constrained_classes}
\begin{split}
    \widetilde{\Pc}_W(t,\nu)
    &:=\left\{
      \P\in\Pc_W(t,\nu):
      \int_t^T
      \big(\alpha-\E^{\P}[I_s]\big)^+
      \E^{\P}[I_s] \,ds=0
      \right\},\\
    \widetilde{\Pc}_S(t,\nu)
    &:=\widetilde{\Pc}_W(t,\nu)\cap\Pc_S(t,\nu).
\end{split}
\end{equation}

\subsection{Connection with the quantile-truncated stopping problem}

We now relate this formulation to the motivating problem
\begin{equation}\label{eq:original_quantile_problem}
    \sup_{\tau}
    \E\left[
        G_{\tau\wedge F_{\tau}^{-1}(\beta)}
    \right],
\end{equation}
where $F_{\tau}(s):=\P(\tau\le s)$ is the distribution function of
$\tau$. Here $\beta$ denotes the quantile level in
\eqref{eq:original_quantile_problem} to distinguish it from the
survival threshold $\alpha$ used throughout this paper.

Given $\P\in\Pc_W(t,\nu)$, set
\[
    \tau:=\vartheta_t(I).
\]
Then
\[
    \E^\P[I_s]=\P(\tau>s)=1-F_\tau(s).
\]
Define the deterministic survival cutoff
\begin{equation}\label{eq:survival_quantile_cutoff}
    q_\alpha(\tau)
    :=
    \inf\big\{s\in[t,T]:\E^\P[I_s]<\alpha\big\}
    =
    \inf\big\{s\in[t,T]:F_\tau(s)>1-\alpha\big\},
\end{equation}
with the convention that the infimum of the empty set is $T+1$.
Since $s\mapsto \E^\P[I_s]$ is nonincreasing, we have
\[
    \mathbf 1_{[\alpha,1]}(\E^\P[I_s])
    =
    \mathbf 1_{\{s<q_\alpha(\tau)\}},
\]
except possibly at the single cutoff time.
Consequently,
\begin{equation}\label{eq:quantile_indicator_identity}
    I_s\mathbf 1_{[\alpha,1]}(\E^\P[I_s])
    =
    \mathbf 1_{\{s<\tau\wedge q_\alpha(\tau)\}},
    \qquad ds\otimes d\P\text{-a.e.}
\end{equation}

For the accumulated running reward process
\[
    G_u
    :=
    \int_t^{u\wedge T} f(s,X_s,\mu_s)\,ds,
\]
identity \eqref{eq:quantile_indicator_identity} yields
\begin{equation}\label{eq:objective_as_quantile_stopping}
\begin{split}
    J_t(\P)
    =
    \E^\P\left[
        \int_t^T
        f(s,X_s,\mu_s)
        I_s\mathbf 1_{[\alpha,1]}(\E^\P[I_s])\,ds
    \right]
    =
    \E^\P\left[
        G_{\tau\wedge q_\alpha(\tau)}
    \right].
\end{split}
\end{equation}
Thus the cutoff in the objective is precisely the quantile determined by the
law of the stopping time itself.

Specifically, $q_\alpha(\tau)$ is the upper generalized
$(1-\alpha)$-quantile
\[
    F_{\tau,+}^{-1}(1-\alpha)
    :=
    \inf\{s:F_\tau(s)>1-\alpha\}.
\]
When this quantile is unique (for example, when $F_\tau$ is strictly
increasing near the level $1-\alpha$), it agrees with the usual
generalized inverse $F_\tau^{-1}(1-\alpha)$. Hence
\eqref{eq:original_quantile_problem} corresponds to the present
notation through the identification
\[
    \beta=1-\alpha.
\]
Our problem extends \eqref{eq:original_quantile_problem} to a mean-field
setting: both the state dynamics and the running reward depend on the joint
population law $\mu_s=\Lc^\P(X_s,I_s)$. The strong class consists of stopping
rules generated by the Brownian information after time $t$, whereas the weak
class is a relaxation that allows more general admissible laws.

\section{Equivalence of the strong and weak formulations}

We need the following lemmas to prove the equivalence between strong and weak formulations. 
\begin{lemma}
\label{lem:constraint_reduction}
For $A\in\{S,W\}$, we have that
\begin{equation}\label{eq:constraint_reduction_values}
    \sup_{\P\in\Pc_A(t,\nu)}J_t(\P)
    =
    \sup_{\P\in\widetilde{\Pc}_A(t,\nu)}J_t(\P).
\end{equation}
Moreover, if $\P\in\widetilde{\Pc}_A(t,\nu)$, then
\begin{equation}\label{eq:objective_on_constrained_set}
    J_t(\P)
    =
    \E^{\P}\left[
        \int_t^T f(s,X_s,\mu_s)I_s\,ds
    \right].
\end{equation}
\end{lemma}
\begin{proof}
Fix $\P\in\Pc_A(t,\nu)$ and set
$    p_s:=\E^{\P}[I_s],
    s\in[t,T].$
Since $I$ is bounded, right-continuous, and non-increasing, the function $p$
is right-continuous and non-increasing. Define the deterministic cutoff
\begin{equation}\label{eq:deterministic_cutoff}
    \theta
    :=\inf\{s\in[t,T]:p_s<\alpha\},
    \qquad
    \inf\varnothing:=T+1.
\end{equation}
Let
\[
    \bar I_{t-}:=I_{t-},
    \qquad
    \bar I_s:=I_s\mathbf 1_{\{s<\theta\}},
    \quad s\in[t,T],
\]
and set $\bar I_s:=I_s$ for $s<t$. On the same probability space, let
$\bar X$ be the solution on $[t,T]$ initialized at $X_t$ and driven by $W$:
\begin{equation}\label{eq:cutoff_dynamics}
    \bar X_s
    =X_t
     +\int_t^s b(r,\bar X_r,\bar\mu_r)\bar I_r\,dr
     +\int_t^s \sigma(r,\bar X_r,\bar\mu_r)\bar I_r\,dW_r,
    \qquad
    \bar\mu_r:=\Lc^{\P}(\bar X_r,\bar I_r).
\end{equation}
Set $\bar X_s:=X_s$ for $s<t$. Existence and pathwise uniqueness follow from
the Lipschitz assumption and the standard Picard argument for a
McKean--Vlasov SDE with the bounded adapted multiplier $\bar I$.

Since $\bar I=I$ on $[t,\theta)$, pathwise uniqueness on compact
subintervals of $[t,\theta)$ gives
\begin{equation}\label{eq:cutoff_same_before_theta}
    \bar X_s=X_s,
    \qquad
    \bar\mu_s=\mu_s,
    \qquad t\le s<\theta.
\end{equation}
Let
    $\bar\P:=\Lc^{\P}(\bar X,W,\bar I).$
The natural filtration of $(\bar X,W,\bar I)$ is a subfiltration of the
original weak filtration, so $W$ remains Brownian. If $A=S$, then $\bar I$ is
still adapted to $\mathbb G^t$ because $\theta$ is deterministic. Hence
$\bar\P\in\Pc_A(t,\nu)$.

For $s<\theta$, $p_s\ge\alpha$. For $s>\theta$, monotonicity and the
definition of $\theta$ give $p_s<\alpha$, while $\bar I_s=0$. The value at
the single time $\theta$ does not affect integration with respect to $ds$.
Together with \eqref{eq:cutoff_same_before_theta}, these observations give
\[
    J_t(\bar\P)=J_t(\P).
\]
Moreover,
\[
    \E^{\bar\P}[I_s]
    =\begin{cases}
       p_s, &s<\theta,\\
       0,   &s\ge\theta,
     \end{cases},
\]
so
$\bar\P\in\widetilde{\Pc}_A(t,\nu)$. This proves
\eqref{eq:constraint_reduction_values}.

The reverse inequality follows from
$\widetilde{\Pc}_A(t,\nu)\subseteq\Pc_A(t,\nu)$.

For $\P\in\widetilde{\Pc}_A(t,\nu)$, the integrand in
\eqref{eq:constrained_classes} is nonnegative,
\[
    \big(\alpha-p_s\big)^+p_s=0
\]
for Lebesgue-a.e. $s$. Hence, for almost every $s$, either $p_s=0$ or
$p_s\ge\alpha$. If $p_s=0$, then $I_s=0$, $\P$-a.s.; if $p_s\ge\alpha$, the
threshold indicator in \eqref{eq:objective} equals one. Thus
\[
    I_s\mathbf 1_{[\alpha,1]}(p_s)=I_s
\]
for $ds\otimes d\P$-a.e. $(s,\omega)$, which gives
\eqref{eq:objective_on_constrained_set}.
\end{proof}

To approximate a weak rule by strong ones, we discretize its stopping
decision and reserve the initial Brownian interval $[t,t_1^m]$ for independent
randomization.

Fix a sequence of partitions
\begin{equation}\label{eq:partitions}
    \pi_m:
    \quad
    t=t_0^m<t_1^m<\cdots<t_m^m=T,
    \qquad
    |\pi_m|:=\max_{1\le i\le m}(t_i^m-t_{i-1}^m)
    \longrightarrow0.
\end{equation}
Write
    $\delta_m:=t_1^m-t$.
For a survival path on $[t,T]$, write $\tau:=\vartheta_t(I)$.

\begin{lemma}[Adapted discrete approximation]
\label{lem:approximate_by_discrete_stopping}
Let $t<T$ and $\P\in\Pc_W(t,\nu)$. There exists an $\F$-stopping time
$\tau^m$ taking values in
    $\{t,t_1^m,\ldots,t_m^m,T+1\}$
and an associated survival process
\begin{equation}\label{eq:Im_definition}
    I_{t-}^m:=I_{t-},
    \qquad
    I_s^m:=I_{t-}\mathbf 1_{\{s<\tau^m\}},
    \quad s\in[t,T],
\end{equation}
such that
\begin{equation}\label{eq:Im_path_error}
    |\tau^m-\tau|\le |\pi_m|,
    \qquad
    \int_t^T|I_s^m-I_s|\,ds\le|\pi_m|,
    \qquad \P\text{-a.s.}
\end{equation}
If $\P\in\widetilde{\Pc}_W(t,\nu)$, $I^m$ may be chosen so that
\begin{equation}\label{eq:Im_constraint}
    \int_t^T
      \big(\alpha-\E^{\P}[I_s^m]\big)^+
      \E^{\P}[I_s^m] \,ds=0.
\end{equation}
Set
\begin{equation}\label{eq:Wm_definition}
    W_s^m
    :=\begin{cases}
        W_s, &0\le s\le t,\\
        W_t, &t<s\le t_1^m,\\
        W_t+W_s-W_{t_1^m}, &t_1^m<s\le T,
      \end{cases},
\end{equation}
and define $X^m$ by
\begin{equation}\label{eq:Xm_before_t1}
    X_s^m:=X_s,\quad s<t,
    \qquad
    X_s^m:=X_t,\quad s\in[t,t_1^m],
\end{equation}
and, for $s\in[t_1^m,T]$, by
\begin{equation}\label{eq:dynamics_truncated_W}
    X_s^m
    =X_t
     +\int_{t_1^m}^s b(r,X_r^m,\mu_r^m)I_r^m\,dr
     +\int_{t_1^m}^s \sigma(r,X_r^m,\mu_r^m)I_r^m\,dW_r,
    \qquad
    \mu_r^m:=\Lc^{\P}(X_r^m,I_r^m).
\end{equation}
Then
\begin{equation}\label{eq:Xm_convergence}
    \E^{\P}\left[
        \sup_{s\in[t,T]}|X_s^m-X_s|^2
    \right]\longrightarrow0,
\end{equation}
and
\begin{equation}\label{eq:truncated_projected_law_convergence}
\begin{split}
    \Wc_{2,t}\Big(
      &\Lc^{\P}\big(
         X^m_{|[t,T]},
         (W_s^m-W_t)_{s\in[t,T]},
         I^m_{|\{t-\}\cup[t,T]}
      \big),
      \P^t
    \Big)
    \longrightarrow0.
\end{split}
\end{equation}
\end{lemma}

\begin{proof}
Begin with the unconstrained case and define the $\F$-stopping time
\begin{equation}\label{eq:tau_rounding}
    \tau^m
    :=\begin{cases}
       t, &I_{t-}=0,\\
       \min\{t_i^m:1\le i\le m,\ t_i^m\ge\tau\},
          &I_{t-}=1,\ \tau\le T,\\
       T+1, &I_{t-}=1,\ \tau=T+1.
    \end{cases}
\end{equation}
The estimates in \eqref{eq:Im_path_error} follow directly.

Now suppose that $\P$ is constrained, and write
    $p_s:=\E^{\P}[I_s]$.
The function
\[
    s\longmapsto (\alpha-\E^{\P}[I_s])^+\E^{\P}[I_s]
\]
is nonnegative and right-continuous. Since its integral on $[t,T]$ is zero,
it must vanish at every $s\in[t,T)$; otherwise, right-continuity would make it
strictly positive on an interval to the right. Thus
\begin{equation}\label{eq:constraint_pointwise}
    \E^{\P}[I_s]\in\{0\}\cup[\alpha,1],
    \qquad s\in[t,T).
\end{equation}
If $0\le \E^{\P}[I_{t-}]<\alpha$, then
$\E^{\P}[I_t]\le \E^{\P}[I_{t-}]<\alpha$, and
\eqref{eq:constraint_pointwise} implies $\E^{\P}[I_t]=0$. Since $I_t\ge0$,
we have $I_t=0$, $\P$-a.s. Monotonicity then gives $I_s=0$ for all
$s\ge t$, $\P$-a.s. In this case, take $\tau^m:=t$, so that $I^m=0$ on
$[t,T]$.

If $\E^{\P}[I_{t-}]\ge\alpha$, then on $[t,t_1^m)$,
$I_s^m=I_{t-}$ and hence
    $\E^{\P}[I_s^m]=\E^{\P}[I_{t-}]\ge\alpha$.
For $i=1,\ldots,m-1$ and $s\in[t_i^m,t_{i+1}^m)$,
    $I_s^m=I_{t_i^m}$,
so
\[
    \E^{\P}[I_s^m]=p_{t_i^m}
    \in\{0\}\cup[\alpha,1]
\]
by \eqref{eq:constraint_pointwise}. The value at the single time $T$ does not
affect the time integral. This proves \eqref{eq:Im_constraint}.

For the state estimate, \eqref{eq:Wm_definition} gives
\begin{equation}\label{eq:Wm_error}
\begin{split}
    \E^{\P}\left[
       \sup_{s\in[t,T]}
       \big|(W_s^m-W_t)-(W_s-W_t)\big|^2
    \right]
    &\le
    C\E^{\P}\left[
       \sup_{s\in[t,t_1^m]}|W_s-W_t|^2
    \right]
    \le C\delta_m.
\end{split}
\end{equation}
The standard moment estimate, applied also to \eqref{eq:dynamics_truncated_W},
gives
\begin{equation}\label{eq:Xm_uniform_moment}
    \E^{\P}\left[
       \sup_{s\in[t,T]}|X_s|^2
       +\sup_{s\in[t,T]}|X_s^m|^2
    \right]
    \le C\big(1+\E^{\P}|X_t|^2\big).
\end{equation}
On the reserved interval,
\begin{equation}\label{eq:reserved_state_error}
    \E^{\P}\left[
       \sup_{s\in[t,t_1^m]}|X_s-X_t|^2
    \right]
    \le C\delta_m.
\end{equation}
For $r\ge t_1^m$, use the coupling
$((X_r^m,I_r^m),(X_r,I_r))$ to obtain
\begin{equation}\label{eq:mu_coupling_bound}
    \Wc_2^2(\mu_r^m,\mu_r)
    \le
    \E^{\P}\big[|X_r^m-X_r|^2+|I_r^m-I_r|^2\big].
\end{equation}
Consider the identity
\[
\begin{split}
    &b(r,X_r^m,\mu_r^m)I_r^m-b(r,X_r,\mu_r)I_r
    =
    \big(b(r,X_r^m,\mu_r^m)-b(r,X_r,\mu_r)\big)I_r^m
    +b(r,X_r,\mu_r)(I_r^m-I_r),
\end{split}
\]
and its analogue for $\sigma$. Applying the Lipschitz and growth assumptions,
\eqref{eq:mu_coupling_bound}, the Burkholder--Davis--Gundy inequality, and
\eqref{eq:Im_path_error} to these identities yields, for $u\in[t_1^m,T]$,
\begin{equation}\label{eq:Xm_gronwall_pre}
\begin{split}
    \E^{\P}\left[
       \sup_{s\in[t,u]}|X_s^m-X_s|^2
    \right]
    \le\;&C\delta_m
       +C\int_{t_1^m}^u
          \E^{\P}\left[
             \sup_{v\in[t,r]}|X_v^m-X_v|^2
          \right]dr\\
       &+C\E^{\P}\left[
          \left(1+\sup_{r\in[t,T]}|X_r|^2\right)
          \int_t^T|I_r^m-I_r|\,dr
       \right]\\
    \le\;&C(\delta_m+|\pi_m|)
       +C\int_{t_1^m}^u
          \E^{\P}\left[
             \sup_{v\in[t,r]}|X_v^m-X_v|^2
          \right]dr.
\end{split}
\end{equation}
Gronwall's lemma proves \eqref{eq:Xm_convergence}. To finish, use the original
probability space to couple
\[
    \big(X^m_{|[t,T]},(W_s^m-W_t)_{s\in[t,T]},I^m\big),
    \big(X_{|[t,T]},(W_s-W_t)_{s\in[t,T]},I\big).
\]
The first component converges in $L^2$ by \eqref{eq:Xm_convergence}, the
second by \eqref{eq:Wm_error}, and the stopping component by
\eqref{eq:metric_I} and \eqref{eq:Im_path_error}. This proves
\eqref{eq:truncated_projected_law_convergence}.
\end{proof}

The Brownian purification of the grid-valued decisions uses the following law invariance property of the McKean--Vlasov equation.

\begin{lemma}[Law invariance]\label{lem:input_law_invariance}
Fix $a<T$. For $k\in\{1,2\}$, let
$(\Omega^k,\mathcal F^k,\mathbb F^k,\mathbb P^k)$ support a triple
$(\xi^k,B^k,A^k)$ such that:
\begin{enumerate}
\item[(i)] $\xi^k$ is an $\R^n$-valued, square-integrable
$\mathcal F_a^k$-measurable random variable;
\item[(ii)] $B^k$ is a $d$-dimensional $\mathbb F^k$-Brownian motion on
$[a,T]$;
\item[(iii)] $A^k$ is an $\mathbb F^k$-adapted survival process with paths
in $\I([a,T])$;
\item[(iv)]
\[
    \Lc(\xi^1,B^1,A^1)=\Lc(\xi^2,B^2,A^2).
\]
\end{enumerate}
For each $k\in\{1,2\}$, let $Y^k$ be the unique solution of
\begin{equation}\label{eq:input_law_mkv}
    Y_s^k
    =\xi^k
     +\int_a^s b(r,Y_r^k,m_r^k)A_r^k\,dr
     +\int_a^s \sigma(r,Y_r^k,m_r^k)A_r^k\,dB_r^k,
    \qquad
    m_r^k:=\Lc(Y_r^k,A_r^k).
\end{equation}
Then
\begin{equation}\label{eq:input_law_conclusion}
    \Lc(Y^1,\xi^1,B^1,A^1)
    =\Lc(Y^2,\xi^2,B^2,A^2).
\end{equation}
\end{lemma}
\begin{remark}
    The conclusion remains valid if the input includes an
additional time-$a$-measurable random variable that does not
enter the equation, provided that $B^k$ is a Brownian motion in the filtration enlarged
by this variable. In that case, the additional variable is included in the
joint-law conclusion \eqref{eq:input_law_conclusion}.
\end{remark}

\begin{proof}
Let $R$ denote the common law of $(\xi^k,B^k,A^k)$ and realize the canonical
input $(\xi,B,A)$ with law $R$. The canonical process $B$ is a Brownian motion in the canonical
filtration generated by $(\xi,B,A)$. For each measurable deterministic
measure flow $m=(m_r)_{r\in[a,T]}$ satisfying
$\int_a^T M_2(m_r)\,dr<\infty$, the
classical SDE
\begin{equation}\label{eq:fixed_flow_sde}
    Z_s
    =\xi
     +\int_a^s b(r,Z_r,m_r)A_r\,dr
     +\int_a^s \sigma(r,Z_r,m_r)A_r\,dB_r
\end{equation}
has a pathwise unique strong solution under Assumption
\ref{ass:coefficient_reward}. The Picard iterates are nonanticipative
measurable functionals of $(\xi,B,A)$ and converge in $L^2$. Consequently,
the joint law of the solution and $(\xi,B,A)$ is determined by $R$ and $m$.

Let $m^k_r:=\Lc(Y_r^k,A_r^k)$ and, on the common input space, let $Z^k$ solve
\eqref{eq:fixed_flow_sde} with the deterministic flow $m^k$. The standard
moment estimate ensures that these flows satisfy the required integrability.
By the preceding observation,
\begin{equation}\label{eq:Z_Y_same_law}
    \Lc(Z^k,\xi,B,A)
    =\Lc(Y^k,\xi^k,B^k,A^k),
    \qquad k=1,2.
\end{equation}
Hence
    $\Lc(Z_r^k,A_r)=m_r^k.$
Using the common-space coupling gives
\begin{equation}\label{eq:input_law_W2}
    \Wc_2^2(m_r^1,m_r^2)
    \le \E|Z_r^1-Z_r^2|^2.
\end{equation}
The Burkholder--Davis--Gundy inequality, the boundedness of $A$, the
Lipschitz condition, and \eqref{eq:input_law_W2} imply, for $u\in[a,T]$,
\[
    \E\left[\sup_{a\le s\le u}|Z_s^1-Z_s^2|^2\right]
    \le
    C\int_a^u
      \E\left[\sup_{a\le v\le r}|Z_v^1-Z_v^2|^2\right]dr.
\]
Gronwall's lemma yields $Z^1=Z^2$ almost surely. Combining this with
\eqref{eq:Z_Y_same_law} proves \eqref{eq:input_law_conclusion}. The same
construction under the common augmented input law proves the stated conclusion
for an additional time-$a$ random variable.
\end{proof}

For simplification, we write
\begin{equation}\label{eq:Bm_definition}
    B^m:=(W_s^m-W_t)_{s\in[t,T]}.
\end{equation}
Thus $B_s^m=0$ on $[t,t_1^m]$ and
$B_s^m=W_s-W_{t_1^m}$ for $s\ge t_1^m$.

\begin{lemma}
\label{lem:equivalence_discrete_stopping}
Let $t<T$, $\P\in\Pc_W(t,\nu)$, and let $I^m$ and $X^m$ be as in Lemma
\ref{lem:approximate_by_discrete_stopping}. There exists an
$\mathbb G^t$-adapted survival process $\widetilde I^m$, piecewise constant on
the grid $\pi_m$, such that
\begin{equation}\label{eq:purification_input_law}
\begin{split}
    \Lc^{\P}\Big(
       X_t,I_{t-},W_{\cdot\wedge t},B^m,
       I^m_{|\{t-\}\cup[t,T]}
    \Big)
    =
    \Lc^{\P}\Big(
       X_t,I_{t-},W_{\cdot\wedge t},B^m,
       \widetilde I^m_{|\{t-\}\cup[t,T]}
    \Big).
\end{split}
\end{equation}
Define $\widetilde X_s^m:=X_t$ for $s\in[t,t_1^m]$ and, for
$s\in[t_1^m,T]$, let
\begin{equation}\label{eq:tilde_X_truncated}
    \widetilde X_s^m
    =X_t
     +\int_{t_1^m}^s
       b(r,\widetilde X_r^m,\widetilde\mu_r^m)
       \widetilde I_r^m\,dr
     +\int_{t_1^m}^s
       \sigma(r,\widetilde X_r^m,\widetilde\mu_r^m)
       \widetilde I_r^m\,dW_r,
       \qquad
       \widetilde\mu_r^m
    :=\Lc^{\P}(\widetilde X_r^m,\widetilde I_r^m).
\end{equation}
Then
\begin{equation}\label{eq:purification_solution_law}
\begin{split}
    \Lc^{\P}\Big(
       \widetilde X^m_{|[t,T]},B^m,
       \widetilde I^m_{|\{t-\}\cup[t,T]}
    \Big)
    =
    \Lc^{\P}\Big(
       X^m_{|[t,T]},B^m,
       I^m_{|\{t-\}\cup[t,T]}
    \Big).
\end{split}
\end{equation}
If $I^m$ satisfies the constraint \eqref{eq:Im_constraint}, then
$\widetilde I^m$ satisfies the constraint \eqref{eq:Im_constraint}.
\end{lemma}

\begin{proof}
All the spaces involved are Polish, so regular conditional probabilities
exist.
For $i=1,\ldots,m$, set
\[
    \mathcal H_i^m
    :=\sigma\big(
       X_t,I_{t-},W_{\cdot\wedge t},B^m_{\cdot\wedge t_i^m},I_{t_1^m}^m,\ldots,I_{t_{i-1}^m}^m
    \big).
\]
Since the conditioning variables take values in standard Borel spaces, there
exists a Borel function $q_i^m$ with values in $[0,1]$ such that
\begin{equation}\label{eq:q_kernel}
\begin{split}
    q_i^m\big(
       X_t,I_{t-},W_{\cdot\wedge t},B^m_{\cdot\wedge t_i^m},I_{t_1^m}^m,\ldots,I_{t_{i-1}^m}^m
    \big)
      =\E^{\P}[I_{t_i^m}^m\mid\mathcal H_i^m],
      \qquad \P\text{-a.s.}
\end{split}
\end{equation}
Because $I_{t_i^m}^m$ is $\{0,1\}$-valued, the right-hand side determines its
conditional law. Given $\mathcal H_i^m$, the variable $I_{t_i^m}^m$ is
Bernoulli with the success probability specified in \eqref{eq:q_kernel}.

Split the reserved interval $[t,t_1^m]$ into $m$ equal pieces by setting
    $u_j^m:=t+\frac{j}{m}\delta_m,
    \qquad j=0,\ldots,m$.
Let $\Phi$ be the standard normal cumulative distribution function, and use
the first coordinate of $W$ to define
\begin{equation}\label{eq:uniform_seeds}
    U_j^m
    :=\Phi\left(
       \frac{
          W_{u_j^m}^{(1)}-W_{u_{j-1}^m}^{(1)}
       }{
          \sqrt{u_j^m-u_{j-1}^m}
       }
    \right),
    \qquad j=1,\ldots,m.
\end{equation}
Because $W$ is Brownian in the weak filtration, the variables
$U_1^m,\ldots,U_m^m$ are independent and uniformly distributed on $[0,1]$.
The vector $(U_1^m,\ldots,U_m^m)$ is also independent of the tuple
$(X_t,I_{t-},W_{\cdot\wedge t},B^m)$.

Define recursively
\begin{equation}\label{eq:tilde_A_recursion}
    \widetilde I_i
    :=\mathbf 1_{\left\{
       U_i^m\le
       q_i^m\big(
          X_t,I_{t-},W_{\cdot\wedge t},
          B^m_{\cdot\wedge t_i^m},
          \widetilde I_1,\ldots,\widetilde I_{i-1}
       \big)
    \right\}},
    \qquad i=1,\ldots,m.
\end{equation}
The resulting variables satisfy
\begin{equation}\label{eq:discrete_vector_law}
    \Lc^{\P}(X_t,I_{t-},W_{\cdot\wedge t},B^m,I_{t_1^m}^m,\ldots,I_{t_m^m}^m)
    =
    \Lc^{\P}(X_t,I_{t-},W_{\cdot\wedge t},B^m,\widetilde I_1,\ldots,\widetilde I_m).
\end{equation}
To see this, proceed by induction over the grid. At time
$t_1^m$, the conditional law of $I_{t_1^m}^m$ given
$(X_t,I_{t-},W_{\cdot\wedge t},B^m_{\cdot\wedge t_1^m})$ is Bernoulli with
parameter $q_1^m$, and independence of $U_1^m$ gives the same law for
$\widetilde I_1$. Suppose the joint laws agree through time
$t_i^m$. Since $I_{t_i^m}^m$ is
$\Fc_{t_i^m}$-measurable and $W$ is a $\F$-Brownian motion, the increment
\[
    (W_s-W_{t_i^m})_{s\in[t_i^m,t_{i+1}^m]}
\]
is independent of $\Fc_{t_i^m}$. Hence, conditionally on the matched variables
through $t_i^m$, the next segment of $B^m$ has the same Wiener transition law
in the original and constructed systems. At $t_{i+1}^m$, the conditional law of
$I_{t_{i+1}^m}^m$ given the enlarged past is Bernoulli with parameter $q_{i+1}^m$,
and \eqref{eq:tilde_A_recursion} produces exactly this conditional law using
the independent seed $U_{i+1}^m$. The induction gives
\eqref{eq:discrete_vector_law}.

By construction in Lemma \ref{lem:approximate_by_discrete_stopping}, the
segment $I^m_{|[t,t_1^m)}$ is a Borel function of $I_{t-}$. Use the same
function for $\widetilde I^m$ on $[t,t_1^m)$, set
$\widetilde I_{t-}^m:=I_{t-}$, and define
\[
    \widetilde I_s^m:=\widetilde I_i,
    \qquad
    s\in[t_i^m,t_{i+1}^m),
    \quad i=1,\ldots,m,
\]
with the convention $t_{m+1}^m:=T+1$. The process $\widetilde I^m$ is then a
survival process adapted to $\mathbb G^t$, and
\eqref{eq:discrete_vector_law} therefore yields
\eqref{eq:purification_input_law}.

On $[t_1^m,T]$, the process
    $(W_s-W_{t_1^m})_{s\in[t_1^m,T]}$
is the nonconstant part of $B^m$. Hence
\eqref{eq:purification_input_law} gives equality of the joint input laws
needed in Lemma \ref{lem:input_law_invariance}. Applying that lemma from time
$t_1^m$, with initial state $X_t$, yields
\eqref{eq:purification_solution_law}. Finally,
\eqref{eq:purification_input_law} gives
\[
    \Lc^{\P}(\widetilde I_s^m)=\Lc^{\P}(I_s^m)
    \qquad\text{for every }s,
\]
so the deterministic constraint \eqref{eq:Im_constraint} is preserved.
\end{proof}

\begin{lemma}
\label{lem:restore_dynamics}
Let $\widetilde I^m$ be given by Lemma
\ref{lem:equivalence_discrete_stopping}. Set
$\widehat X_s^m:=X_s$ for $s<t$, and let $\widehat X^m$ solve on $[t,T]$:
\begin{equation}\label{eq:restored_dynamics}
    \widehat X_s^m
    =X_t
     +\int_t^s
       b(r,\widehat X_r^m,\widehat\mu_r^m)
       \widetilde I_r^m\,dr
     +\int_t^s
       \sigma(r,\widehat X_r^m,\widehat\mu_r^m)
       \widetilde I_r^m\,dW_r,
\end{equation}
where
\[
    \widehat\mu_r^m
    :=\Lc^{\P}(\widehat X_r^m,\widetilde I_r^m).
\]
Extend $\widetilde I^m$ to times before $t$ by using the original path
$I$, and set
\begin{equation}\label{eq:Pm_hat}
    \widehat\P_m
    :=\Lc^{\P}(\widehat X^m,W,\widetilde I^m).
\end{equation}
Then
    $\widehat\P_m\in\Pc_S(t,\nu).$
    
If $\P\in\widetilde{\Pc}_W(t,\nu)$ and $I^m$ was chosen as in the constrained
part of Lemma \ref{lem:approximate_by_discrete_stopping}, then
    $\widehat\P_m\in\widetilde{\Pc}_S(t,\nu)$.
Moreover,
\begin{equation}\label{eq:restore_state_error}
    \E^{\P}\left[
       \sup_{s\in[t,T]}
       |\widehat X_s^m-\widetilde X_s^m|^2
    \right]
    \le C\delta_m,
\end{equation}
and
\begin{equation}\label{eq:strong_projected_law_convergence}
    \Wc_{2,t}(\widehat\P_m^t,\P^t)\longrightarrow0.
\end{equation}
\end{lemma}

\begin{proof}
The process $\widetilde I^m$ is $\mathbb G^t$-adapted. 
Let $\widehat\F^m$ be the natural filtration of
$(\widehat X^m,W,\widetilde I^m)$, including the copied paths
before $t$. Each of these coordinates is adapted to the original weak
filtration $\F$, so
\[
    \widehat\Fc_s^m\subseteq\Fc_s,
    \qquad s\in[0,T].
\]
Since $W$ is a $\F$-Brownian motion, it remains a $\widehat\F^m$-Brownian. 
The initial law is unchanged, and
$\widetilde I_s^m$ is measurable with respect to the completion of
$\Gc_s^t$ for each $s\ge t$. Hence the pushforward law
$\widehat\P_m$ is a strong stopping rule. If $\P$ is constrained,
Lemma \ref{lem:equivalence_discrete_stopping} gives equality of the marginal
survival laws of $I^m$ and $\widetilde I^m$, so the constraint is preserved.

On $[t,t_1^m]$, $\widetilde X_s^m=X_t$. The growth condition and the standard
moment estimate for \eqref{eq:restored_dynamics} give
\begin{equation}\label{eq:restore_initial_error}
    \E^{\P}\left[
       \sup_{s\in[t,t_1^m]}
       |\widehat X_s^m-X_t|^2
    \right]
    \le C\delta_m.
\end{equation}
For $s\ge t_1^m$, the equations for $\widehat X^m$ and
$\widetilde X^m$ are driven by the same Brownian increments and the same
survival process. The coupling
   $ \big((\widehat X_r^m,\widetilde I_r^m),
         (\widetilde X_r^m,\widetilde I_r^m)\big)$
gives
\[
    \Wc_2^2(\widehat\mu_r^m,\widetilde\mu_r^m)
    \le
    \E^{\P}|\widehat X_r^m-\widetilde X_r^m|^2.
\]
Thus the Burkholder--Davis--Gundy inequality and the Lipschitz condition imply,
for $u\in[t_1^m,T]$,
\[
\begin{split}
    \E^{\P}\left[
       \sup_{s\in[t,u]}
       |\widehat X_s^m-\widetilde X_s^m|^2
    \right]
    \le\; C\delta_m
    +C\int_{t_1^m}^u
       \E^{\P}\left[
          \sup_{v\in[t,r]}
          |\widehat X_v^m-\widetilde X_v^m|^2
       \right]dr.
\end{split}
\]
Gronwall's lemma proves \eqref{eq:restore_state_error}.

On the same probability space, couple
\[
    \big(
       \widehat X^m_{|[t,T]},
       (W_s-W_t)_{s\in[t,T]},
       \widetilde I^m
    \big)
\]
with
\[
    \big(
       \widetilde X^m_{|[t,T]},
       B^m,
       \widetilde I^m
    \big).
\]
By \eqref{eq:restore_state_error} and \eqref{eq:Wm_error}, their
$\Wc_{2,t}$-distance tends to zero. By
\eqref{eq:purification_solution_law}, the second triple has exactly the same
law as
\[
    \big(X^m_{|[t,T]},B^m,I^m\big),
\]
and this law converges to $\P^t$ by
\eqref{eq:truncated_projected_law_convergence}. This proves
\eqref{eq:strong_projected_law_convergence}.
\end{proof}

To prove the convergence of the value function, we need the following lower semicontinuity property of
the running-reward functional on laws over $\Om_t$.

\begin{lemma}[Lower semicontinuity of the running reward]
\label{lem:reward_lsc}
For $Q\in\Pc_2(\Om_t)$, define
\begin{equation}\label{eq:F_functional}
    \mathcal F_t(Q)
    :=\int_t^T\int_{\Om_t}
       f\big(s,x_s,\Lc^Q(x_s,i_s)\big)i_s
       \,Q(dx,dw,di)\,ds.
\end{equation}
Under Assumption \ref{ass:coefficient_reward}, the functional
$\mathcal F_t$ is lower semicontinuous on
$(\Pc_2(\Om_t),\Wc_{2,t})$.
\end{lemma}

\begin{proof}
Let $Q_n\to Q$ in $\Wc_{2,t}$, and choose a subsequence such that
\[
    \lim_{n\to\infty}\mathcal F_t(Q_n)
    =\liminf_{k\to\infty}\mathcal F_t(Q_k).
\]
For each $n$, choose a coupling of $Q_n$ and $Q$ whose expected squared
$d_t$-distance tends to zero. Then there exist these couplings $(X^n,I^n)$, $(X,I)$ on one
probability space with a common $Q$-distributed variable, so that
\begin{equation}\label{eq:common_coupling_convergence}
    \E\left[
       \|X^n-X\|_{\infty,[t,T]}^2
       +d_{\I,t}(I^n,I)^2
    \right]\longrightarrow0.
\end{equation}
Passing to a further subsequence, we may assume that the expectations in
\eqref{eq:common_coupling_convergence} are summable. Hence
\begin{equation}\label{eq:almost_sure_path_convergence}
    \|X^n-X\|_{\infty,[t,T]}
    +d_{\I,t}(I^n,I)
    \longrightarrow0
    \qquad\text{a.s.}
\end{equation}
By \eqref{eq:I_L1_bound},
    $I_s^n\longrightarrow I_s$
for $ds\otimes d\P$-a.e. $(s,\omega)$.
Set
    $\mu_s^n:=\Lc^{Q_n}(x_s,i_s),
    \mu_s:=\Lc^Q(x_s,i_s)$.
The common-space coupling gives
\begin{equation}
\begin{split}
    \int_t^T\Wc_2^2(\mu_s^n,\mu_s)\,ds
    \le
    \E\int_t^T
       \big(|X_s^n-X_s|^2+|I_s^n-I_s|^2\big)\,ds
       \longrightarrow0.
\end{split}
\end{equation}
Along another subsequence,
   $ \Wc_2(\mu_s^n,\mu_s)\longrightarrow0$
for Lebesgue-a.e. $s$. Combining this with
\eqref{eq:almost_sure_path_convergence}, the lower semicontinuity of $f$, and
the fact that $I_s^n,I_s\in\{0,1\}$, we obtain
\begin{equation}\label{eq:pointwise_reward_liminf}
    \liminf_{n\to\infty}
       f(s,X_s^n,\mu_s^n)I_s^n
    \ge
       f(s,X_s,\mu_s)I_s
\end{equation}
for $ds\otimes d\P$-a.e. $(s,\omega)$.

Let $h(s,x,m):=1+|x|^2+M_2(m)$. Then \eqref{eq:f_growth} gives
    $f(s,x,m)i+C h(s,x,m)\ge0$
for $i\in\{0,1\}$.
Moreover, $\Wc_2(\mu_s^n,\mu_s)\to0$ implies
$M_2(\mu_s^n)\to M_2(\mu_s)$ for a.e. $s$. Fatou's lemma applied to the
shifted nonnegative integrands therefore gives
\begin{equation}\label{eq:fatou_shifted}
    \mathcal F_t(Q)+C\mathcal H_t(Q)
    \le
    \liminf_{n\to\infty}
       \big(\mathcal F_t(Q_n)+C\mathcal H_t(Q_n)\big),
\end{equation}
where
\[
    \mathcal H_t(Q)
    :=\int_t^T\int_{\Om_t}
       h\big(s,x_s,\Lc^Q(x_s,i_s)\big)
       \,Q(dx,dw,di)\,ds.
\]
Since
\[
    M_2(\mu_s)
    =\E^Q\big[|X_s|^2+I_s\big],
\]
we have
\begin{equation}\label{eq:H_explicit}
    \mathcal H_t(Q)
    =(T-t)
     +2\int_t^T\E^Q|X_s|^2\,ds
     +\int_t^T\E^Q[I_s] \,ds.
\end{equation}
The coupling \eqref{eq:common_coupling_convergence} and Cauchy--Schwarz imply
\[
    \int_t^T\E|X_s^n|^2\,ds
    \longrightarrow
    \int_t^T\E|X_s|^2\,ds.
\]
Indeed,
\[
\begin{split}
    \left|
      \E\int_t^T|X_s^n|^2\,ds
      -\E\int_t^T|X_s|^2\,ds
    \right|
    \le
    T\big(\E\|X^n-X\|_\infty^2\big)^{1/2}
     \big(\E(\|X^n\|_\infty+\|X\|_\infty)^2\big)^{1/2}
    \longrightarrow0,
\end{split}
\]
where the second factor is bounded because $Q_n\to Q$ in $\Wc_{2,t}$. Also,
by \eqref{eq:I_L1_bound},
\[
    \E\int_t^T|I_s^n-I_s|\,ds
    \le
    \E|\vartheta_t(I^n)-\vartheta_t(I)|
    \longrightarrow0.
\]
It follows that $\mathcal H_t(Q_n)\to\mathcal H_t(Q)$. Subtracting the
convergent $\mathcal H_t$ terms from \eqref{eq:fatou_shifted} yields
\[
    \mathcal F_t(Q)
    \le
    \liminf_{n\to\infty}\mathcal F_t(Q_n),
\]
which is the desired lower semicontinuity.
\end{proof}

We can now combine the approximation with lower semicontinuity to obtain
strong--weak equivalence.

\begin{theorem}\label{thm:equivalence_value_function}
Suppose Assumption \ref{ass:coefficient_reward} holds and let
$\nu\in\Pc_2(\Sbf)$.

\begin{enumerate}
\item[(i)] If $t<T$, then for every $\P\in\Pc_W(t,\nu)$ there exists a
sequence $\P_m\in\Pc_S(t,\nu)$ such that
\begin{equation}\label{eq:density_projected_laws}
    \Wc_{2,t}(\P_m^t,\P^t)\longrightarrow0.
\end{equation}
Equivalently, the projected laws
    $\{\P^t:\P\in\Pc_S(t,\nu)\}$ induced by strong stopping rules
are dense in the projected laws
   $ \{\P^t:\P\in\Pc_W(t,\nu)\}$ induced by weak stopping rules
under $\Wc_{2,t}$.

\item[(ii)] If $t<T$ and
$\P\in\widetilde{\Pc}_W(t,\nu)$, the sequence in part \textnormal{(i)} can be
chosen in $\widetilde{\Pc}_S(t,\nu)$.

\item[(iii)] For every $t\in[0,T]$,
    $V_S(t,\nu)=V_W(t,\nu)$.
\end{enumerate}
\end{theorem}

\begin{proof}
Suppose first that $t<T$. Parts \textnormal{(i)} and \textnormal{(ii)} follow
by combining Lemma \ref{lem:approximate_by_discrete_stopping}, Lemma
\ref{lem:equivalence_discrete_stopping}, and Lemma
\ref{lem:restore_dynamics}. For a given weak rule $\P$, the
law $\widehat\P_m$ defined in \eqref{eq:Pm_hat} belongs to
$\Pc_S(t,\nu)$ and satisfies \eqref{eq:strong_projected_law_convergence}; in the
constrained case, it belongs to $\widetilde{\Pc}_S(t,\nu)$.

For part \textnormal{(iii)}, the inclusion
    $\Pc_S(t,\nu)\subseteq\Pc_W(t,\nu)$
immediately gives
\begin{equation}\label{eq:easy_value_inequality}
    V_S(t,\nu)\le V_W(t,\nu).
\end{equation}
For the reverse inequality, Lemma \ref{lem:constraint_reduction} gives
\begin{equation}\label{eq:values_reduced}
\begin{split}
    V_W(t,\nu)
    =\sup_{\P\in\widetilde{\Pc}_W(t,\nu)}J_t(\P),
    \qquad
    V_S(t,\nu)
    =\sup_{\P\in\widetilde{\Pc}_S(t,\nu)}J_t(\P).
\end{split}
\end{equation}
Fix $\P\in\widetilde{\Pc}_W(t,\nu)$ and let
$\P_m\in\widetilde{\Pc}_S(t,\nu)$ be the approximating sequence from part
\textnormal{(ii)}. By \eqref{eq:objective_on_constrained_set},
\begin{equation}\label{eq:J_as_F}
    J_t(\P)=\mathcal F_t(\P^t),
    \qquad
    J_t(\P_m)=\mathcal F_t(\P_m^t).
\end{equation}
Since $\P_m^t\to\P^t$ in $\Wc_{2,t}$, Lemma \ref{lem:reward_lsc} yields
\[
    J_t(\P)
    \le
    \liminf_{m\to\infty}J_t(\P_m)
    \le
    V_S(t,\nu).
\]
Taking the supremum over
$\P\in\widetilde{\Pc}_W(t,\nu)$ and using
\eqref{eq:values_reduced} gives
\[
    V_W(t,\nu)\le V_S(t,\nu).
\]
Together with \eqref{eq:easy_value_inequality}, this proves
part \textnormal{(iii)} for $t<T$.

At $t=T$, both admissible classes are nonempty. On a product probability
space, let $(\xi,j)\sim\nu$ and let $W$ be a Brownian motion
independent of $(\xi,j)$. Set $X_s:=\xi$ for all $s\in[0,T]$, and set
$I_{0-}:=j$ and $I_s:=j$ for all $s\in[0,T]$. The induced canonical law is a
strong stopping rule at time $T$. Since the integral defining $J_T$ is zero
for every admissible rule,
\[
    V_S(T,\nu)=V_W(T,\nu)=0.
\]
Thus part \textnormal{(iii)} also holds at the terminal time.
\end{proof}

\section{Existence and the Dynamic Programming Principle}

For the existence results in this section, we strengthen Assumption
\ref{ass:coefficient_reward}\textnormal{(ii)} by requiring that, for every
$r\in[0,T]$, the map
   $(x,m)\longmapsto f(r,x,m)$
is continuous on $\R^n\times\Pc_2(\Sbf)$.
This additional continuity is not needed for the weak dynamic programming
principle.

For $Q\in\Pc_2(\Om_t)$, define
\[
    p_r^Q:=\int_{\Om_t}i_r\,Q(dx,dw,di),
    \qquad
    H_t(Q):=\int_t^T(\alpha-p_r^Q)^+p_r^Q\,dr,
    \qquad
    r \in [t,T],
\]
and recall the functional $\mathcal F_t$ from
\eqref{eq:F_functional}.  Define the set of projected weak-rule laws by
\begin{equation}\label{eq:projected_admissible_set}
    \mathfrak K(t,\nu)
    :=\{\P^t:\P\in\Pc_W(t,\nu)\}
    \subset\Pc_2(\Om_t).
\end{equation}

\begin{lemma}[Compactness and continuity]
\label{lem:projected_law_compactness}
Under Assumption \ref{ass:coefficient_reward}, the set
$\mathfrak K(t,\nu)$ is compact in
$(\Pc_2(\Om_t),\Wc_{2,t})$.  The map $H_t$ is continuous on
$\Pc_2(\Om_t)$.  Under the additional continuity assumption on $f$ made
above, $\mathcal F_t$ is continuous on $\mathfrak K(t,\nu)$.
\end{lemma}
\begin{proof}
If $t=T$, the dynamic interval is empty, so $\mathcal F_T=H_T=0$.  Moreover, $\mathfrak K(T,\nu)$ consists of the laws on $\Om_T$ with prescribed $(X_T,I_{T-})$-marginal $\nu$.  Since the state marginal is fixed and $\I([T,T])$ is compact, this set is compact in $\Wc_{2,T}$.  We therefore assume $t<T$.

Let $\mathbb F^t=(\mathcal F_u^t)_{u\in[t,T]}$ be the raw canonical filtration on $\Om_t$, where $\mathcal F_u^t:=\sigma(X_t,I_{t-},X_r,B_r,I_r:t\le r\le u)$ and $B$ denotes the second coordinate.  We begin with a martingale characterization of admissibility.  For $\varphi\in C_c^2(\R^{n+d})$, set
\begin{equation}\label{eq:projected_martingale_problem}
\mathcal A_r^{m,i}\varphi(x,w):=i b(r,x,m)\cdot D_x\varphi(x,w)+\frac12\begin{pmatrix}i\sigma(r,x,m)\\ I_d\end{pmatrix}\begin{pmatrix}i\sigma(r,x,m)\\ I_d\end{pmatrix}^{\!\top}:D^2\varphi(x,w),
\end{equation}
and, for $Q\in\Pc_2(\Om_t)$ with $\mu_r^Q=\Lc^Q(X_r,I_r)$, define
\[
M_u^{Q,\varphi}:=\varphi(X_u,B_u)-\varphi(X_t,0)-\int_t^u\mathcal A_r^{\mu_r^Q,I_r}\varphi(X_r,B_r)\,dr.
\]
If $Q\in\mathfrak K(t,\nu)$, It\^o's formula gives $\Lc^Q(X_t,I_{t-})=\nu$ and shows that $M^{Q,\varphi}$ is a $Q$-martingale in $\mathbb F^t$ for every $\varphi\in C_c^2(\R^{n+d})$.

Conversely, suppose that these martingale identities and the initial-law condition
hold. Applying the identities to compactly supported test functions that agree
locally with the coordinate functions and their pairwise products shows that
the process
\[
M_u^X:=X_u-X_t-\int_t^u I_r b(r,X_r,\mu_r^Q)\,dr
\]
and $B$ are continuous local martingales, with covariations
\[
\langle B^\ell,B^j\rangle_u=\delta_{\ell j}(u-t),\qquad \langle M^{X,p},B^\ell\rangle_u=\int_t^u I_r\sigma_{p\ell}(r,X_r,\mu_r^Q)\,dr,
\]
and
\[
\langle M^{X,p},M^{X,q}\rangle_u=\int_t^u I_r(\sigma\sigma^\top)_{pq}(r,X_r,\mu_r^Q)\,dr.
\]
The multidimensional L\'evy characterization therefore implies that $B$ is
a Brownian motion with respect to the usual augmentation of $\mathbb F^t$, hence also
with respect to its completion. Subtracting
$\int_t^u I_r\sigma(r,X_r,\mu_r^Q)\,dB_r$ from the state local martingale
leaves a continuous local martingale with zero quadratic variation. It is
therefore identically zero, so the state equation holds. Thus the martingale
problem is equivalent to weak admissibility on $\Om_t$ and yields a Brownian motion in
the full canonical filtration $\mathbb F^t$.

To prove relative compactness, observe that the $B$-marginal is the fixed
Wiener law and the survival coordinate takes values in the compact space
$\I([t,T])$.  It therefore remains to control the state coordinate.  Let
$\chi_R:\R^n\to\R^n$ be bounded,
$1$-Lipschitz, and equal to the identity on the ball of radius $R$. For
$Q\in\mathfrak K(t,\nu)$, let $X^{R,Q}$ solve the same McKean--Vlasov
equation with the same $(B,I)$ but initial state $\chi_R(X_t)$. The standard
stability estimate gives
\begin{equation}\label{eq:truncation_estimate_compactness}
\sup_{Q\in\mathfrak K(t,\nu)}\E^Q\|X-X^{R,Q}\|_\infty^2\le C\int_{\Sbf}|x-\chi_R(x)|^2\,\nu(dx,di)\xrightarrow[R\to\infty]{}0.
\end{equation}
For fixed $R$, the bounded initial condition, the growth assumption, BDG, and Gronwall yield a constant $C_R$ such that, uniformly in $Q$,
\begin{equation}\label{eq:truncated_fourth_moment}
\E^Q\|X^{R,Q}\|_\infty^4\le C_R,\qquad \E^Q|X_v^{R,Q}-X_u^{R,Q}|^4\le C_R|v-u|^2.
\end{equation}
Kolmogorov's criterion therefore implies that the laws of $X^{R,Q}$ are tight
for each fixed $R$. If $\omega(x,\delta)$ denotes the usual modulus of
continuity, then
$\omega(X,\delta)\le2\|X-X^{R,Q}\|_\infty+\omega(X^{R,Q},\delta)$.
Consequently, for every $\eta>0$,
\[
\sup_Q Q(\omega(X,\delta)>\eta)\le \frac{16}{\eta^2}\sup_Q\E^Q\|X-X^{R,Q}\|_\infty^2+\sup_QQ(\omega(X^{R,Q},\delta)>\eta/2).
\]
First choose $R$ using \eqref{eq:truncation_estimate_compactness}, and then let $\delta\downarrow0$ in \eqref{eq:truncated_fourth_moment}.  Since the law of $X_t$ is fixed by $\nu$, the Arzel\`a--Ascoli criterion gives tightness of the state-coordinate laws in $C([t,T],\R^n)$.

To upgrade weak tightness to $\Wc_{2,t}$-relative compactness, we also need uniform integrability of the squared state norm.  By Cauchy--Schwarz, Markov's inequality, \eqref{eq:truncated_fourth_moment}, and the moment estimate, for every fixed $R$ and $K>0$,
\[
\sup_Q\E^Q\!\left[\|X\|_\infty^2\mathbf1_{\{\|X\|_\infty>K\}}\right]
\le C\sup_Q\E^Q\|X-X^{R,Q}\|_\infty^2+\frac{C_R}{K}.
\]
Letting $K\to\infty$ and then $R\to\infty$ proves uniform integrability.  The
$B$-coordinate is uniformly integrable because its law is fixed, and the
survival coordinate is bounded in $d_{\I,t}$.  Hence
$\mathfrak K(t,\nu)$ is relatively compact in $\Wc_{2,t}$.

To prove closedness, let $Q_k\in\mathfrak K(t,\nu)$ and
$Q_k\to Q$ in $\Wc_{2,t}$. By gluing asymptotically optimal couplings,
realize $(X^k,B^k,I^k)\sim Q_k$ and $(X,B,I)\sim Q$ on one probability
space so that
\begin{equation}\label{eq:coupled_W2_convergence}
\E\!\left[\|X^k-X\|_\infty^2+\|B^k-B\|_\infty^2+d_{\I,t}(I^k,I)^2\right]\longrightarrow0.
\end{equation}
The initial-law condition passes to the limit because $(x,w,i)\mapsto(x_t,i_{t-})$ is continuous.  If $\mu_r^k=\Lc^{Q_k}(X_r,I_r)$ and $\mu_r=\Lc^Q(X_r,I_r)$, then \eqref{eq:I_L1_bound} and the coupling imply
\begin{equation}\label{eq:marginal_flow_convergence}
\int_t^T\Wc_2^2(\mu_r^k,\mu_r)\,dr\le(T-t)\E\|X^k-X\|_\infty^2+\E d_{\I,t}(I^k,I)\longrightarrow0.
\end{equation}
The Lipschitz assumption gives convergence of the unstopped coefficients in
$L^2(dt\otimes d\P)$. Since \eqref{eq:I_L1_bound} also yields
$\int_t^T|I_r^k-I_r|\,dr\to0$ in $L^1$, the growth bound and a truncation
argument allow the indicators to be inserted. Thus
\begin{equation}\label{eq:stopped_coefficient_convergence}
\E\int_t^T\!|I_r^k b(r,X_r^k,\mu_r^k)-I_r b(r,X_r,\mu_r)|\,dr+\E\int_t^T\!|I_r^k\sigma(r,X_r^k,\mu_r^k)-I_r\sigma(r,X_r,\mu_r)|^2\,dr\to0.
\end{equation}
The second convergence also implies $L^1$ convergence of the corresponding quadratic and cross-covariation matrices by Cauchy--Schwarz.  Since $D\varphi$ and $D^2\varphi$ are bounded and uniformly continuous, \eqref{eq:coupled_W2_convergence} and \eqref{eq:stopped_coefficient_convergence} therefore give
\begin{equation}\label{eq:generator_L1_convergence}
\E\int_t^T\left|\mathcal A_r^{\mu_r^k,I_r^k}\varphi(X_r^k,B_r^k)-\mathcal A_r^{\mu_r,I_r}\varphi(X_r,B_r)\right|dr\longrightarrow0.
\end{equation}
In particular, the corresponding martingale increments converge in $L^1$ for
every fixed pair of times, providing the uniform integrability needed to pass
to the limit in the martingale identities.

The only remaining issue is that $i\mapsto i_r$ is not continuous when the stopping time has an atom at $r$.  The set
\[
\{r\in(t,T):Q(I_{t-}=1,\vartheta_t(I)=r)>0\}
\]
is at most countable, so choose a countable dense $D\subset(t,T)$ avoiding it.  For every $r\in D$, evaluation of the survival coordinate is $Q$-a.s. continuous; hence $I_r^k\to I_r$ in probability under the coupling.

Fix $u<v$ in $D$.  Let $G_k$ be any bounded continuous cylinder function of $(X_t^k,I_{t-}^k)$ and finitely many values $(X_r^k,B_r^k,I_r^k)$ at times $r\in D\cap(t,u]$, and let $G$ denote the same cylinder evaluated at the limit coordinates.  Then $G_k\to G$ in probability and the sequence is uniformly bounded.  Since $Q_k$ solves the martingale problem,
\[
\E[(M_v^{Q_k,\varphi}-M_u^{Q_k,\varphi})G_k]=0.
\]
The $L^1$ convergence from \eqref{eq:generator_L1_convergence} and boundedness
of $G_k$ allow passage to the limit:
\[
    \E^Q[(M_v^{Q,\varphi}-M_u^{Q,\varphi})G]=0.
\]
For fixed $u\in D$, these cylinders generate $\mathcal F_u^t$: the $X$- and
$B$-coordinates are continuous, while the survival paths are right-continuous,
so every earlier evaluation is recovered from times in $D$ decreasing to it.
A monotone-class argument therefore extends the identity to every bounded
$\mathcal F_u^t$-measurable $G$.

Finally, for arbitrary $t\le u<v\le T$, choose $u_k,v_k\in D$ such that
$u_k\downarrow u$, $v_k\to v$, and $u_k<v_k$; when $v=T$, take
$v_k\uparrow T$. The process $M^{Q,\varphi}$ is $L^1$-continuous because
$\varphi$ is bounded and the generator is integrable. Passing to the limit
proves the martingale identity for all $u<v$, first in the raw filtration and
then in its $Q$-completion. By the martingale characterization at the
beginning of the proof, $B$ is Brownian in the full canonical filtration
$\mathbb F^t$ and the state equation holds under $Q$.

It remains to realize $Q$ as the image under $\Pi_t$ of a law on the original canonical space.  Let $U$ be a $d$-dimensional Brownian motion on $[0,t]$, independent of $(X,B,I)\sim Q$, and set $\overline W_r=U_r$ for $r\le t$ and $\overline W_r=U_t+B_r$ for $r\ge t$.  Before $t$, set $\overline X_r=X_t$ and $\overline I_r=I_{t-}$, and from $t$ onward keep the coordinates distributed according to $Q$.  Independence of $U$ and the Brownian property of $B$ in $\mathbb F^t$ imply, for every $0\le a<b\le T$ and $\xi\in\R^d$,
\[
\E\!\left[e^{i\xi\cdot(\overline W_b-\overline W_a)}\mid\overline{\mathcal F}_a\right]=e^{-\frac12|\xi|^2(b-a)}.
\]
Indeed, the cases $b\le t$, $a<t<b$, and $t\le a$ follow respectively from
that $U$ is a Brownian motion, independence of $U$ and the system on $\Om_t$, and
that $B$ is a $\mathbb F^t$-Brownian motion. Thus $\overline W$ is a
Brownian motion in the completed canonical filtration. The state equation and
initial law are inherited from $Q$, and the image under $\Pi_t$ of the
resulting weak rule is exactly $Q$. Hence $Q\in\mathfrak K(t,\nu)$, proving
closedness and therefore compactness.

We finish by checking continuity of the two functionals.  If $Q_k\to Q$ in $\Wc_{2,t}$, use the same coupling as above.  Since $a\mapsto(\alpha-a)^+a$ is Lipschitz on $[0,1]$,
\[
\int_t^T|p_r^{Q_k}-p_r^Q|\,dr\le\E\int_t^T|I_r^k-I_r|\,dr\le\E d_{\I,t}(I^k,I)\longrightarrow0,
\]
which proves continuity of $H_t$.  Under the additional continuity assumption on $f$, Lemma \ref{lem:reward_lsc} applies to both $f$ and $-f$, so $\mathcal F_t$ is both lower and upper semicontinuous, hence continuous on $\mathfrak K(t,\nu)$.
\end{proof}

For $\eps>0$, the penalized value can be written in terms of laws on $\Om_t$ as
\begin{equation}\label{eq:penalized_value}
 V_\eps(t,\nu)
 :=\sup_{Q\in\mathfrak K(t,\nu)}
       \left\{\mathcal F_t(Q)-\frac1\eps H_t(Q)\right\}.
\end{equation}
This is also the supremum over $\Pc_W(t,\nu)$, since both terms
depend only on the projected law $\P^t$.

\begin{theorem}[Existence]
\label{thm:existence}
Suppose Assumption \ref{ass:coefficient_reward} holds and, for every
$r\in[0,T]$, the map
    $(x,m)\longmapsto f(r,x,m)$
is continuous on $\R^n\times\Pc_2(\Sbf)$. Then the following assertions hold
for every
$(t,\nu)\in[0,T]\times\Pc_2(\Sbf)$.

\begin{enumerate}
\item[(i)] The penalized problem \eqref{eq:penalized_value}
has a maximizer for every $\eps>0$.

\item[(ii)] The singular weak problem has a maximizer. In particular,
there exists $\P^*\in\widetilde\Pc_W(t,\nu)$ such that
$J_t(\P^*)=V_W(t,\nu)$. Consequently, by Theorem
\ref{thm:equivalence_value_function}, the weak and strong values agree,
although the existence of a strong optimizer is not asserted. 

\item[(iii)] As $\eps\downarrow0$,
       $V_\eps(t,\nu)\downarrow V_W(t,\nu).$
If $\eps_k\downarrow0$ and $\P_k$ is a maximizer of
$V_{\eps_k}(t,\nu)$ for each $k$, then every $\Wc_{2,t}$-limit point $Q_0$ of
$\{\P_k^t\}$ is of the form $Q_0=\P_0^t$ for an optimal
$\P_0\in\widetilde\Pc_W(t,\nu)$.  Moreover,
\[
 \frac1{\eps_k}H_t(\P_k^t)\longrightarrow0.
\]
\end{enumerate}
\end{theorem}
{
\begin{remark}
    The existence of a maximizer for the singular weak problem is due to the compactness of the set $\{Q\in\mathfrak K(t,\nu):H_t(Q)=0\}$. The penalized problem provides an approximation of the singular weak problem.
\end{remark}
}
\begin{proof}
By Lemma \ref{lem:projected_law_compactness}, the set
$\mathfrak K(t,\nu)$ is compact and both $\mathcal F_t$ and $H_t$ are
continuous.  Hence the penalized objective attains its maximum, proving
\textnormal{(i)}.

The set
\[
    \mathfrak K_0(t,\nu)
    :=\{Q\in\mathfrak K(t,\nu):H_t(Q)=0\}
\]
is compact and nonempty; for example, one may stop all surviving particles at
time $t$.  By Lemma
\ref{lem:constraint_reduction},
\begin{equation}\label{eq:singular_value_compact_form}
    V_W(t,\nu)
    =\max_{Q\in\mathfrak K_0(t,\nu)}\mathcal F_t(Q).
\end{equation}
Thus a maximizer $Q^*$ exists.  The closedness result in Lemma
\ref{lem:projected_law_compactness} yields a weak rule $\P^*$ with
$(\P^*)^t=Q^*$, proving \textnormal{(ii)}.

Since every $Q\in\mathfrak K_0(t,\nu)$ has zero penalty,
\begin{equation}\label{eq:penalized_lower_bound}
    V_{\eps_k}(t,\nu)\ge V_W(t,\nu).
\end{equation}
The growth and moment estimates give a finite constant $C_0$, independent
of $Q\in\mathfrak K(t,\nu)$, such that $\mathcal F_t(Q)\le C_0$.
Optimality and \eqref{eq:penalized_lower_bound} imply
\begin{equation}\label{eq:penalty_O_epsilon}
 0\le H_t(\P_k^t)
 \le \eps_k\big(\mathcal F_t(\P_k^t)-V_W(t,\nu)\big)
 \le \eps_k\big(C_0-V_W(t,\nu)\big),
\end{equation}
so $H_t(\P_k^t)\to0$.

By compactness, after passing to a subsequence without relabeling, we may
assume
$\P_k^t\to Q_0$ in $\Wc_{2,t}$.  Continuity of $H_t$ gives
$H_t(Q_0)=0$.  Moreover,
\[
 0\le \frac{H_t(\P_k^t)}{\eps_k}
 \le \mathcal F_t(\P_k^t)-V_W(t,\nu).
\]
Taking the upper limit and using continuity of $\mathcal F_t$ yields
\[
 0\le\varlimsup_{k\to\infty}\frac{H_t(\P_k^t)}{\eps_k}
 \le \mathcal F_t(Q_0)-V_W(t,\nu)\le0,
\]
where the final inequality follows from
$Q_0\in\mathfrak K_0(t,\nu)$ and
\eqref{eq:singular_value_compact_form}.  Hence
$H_t(\P_k^t)/\eps_k\to0$ and
$\mathcal F_t(Q_0)=V_W(t,\nu)$.  Therefore
$Q_0=\P_0^t$ for some optimal constrained weak rule $\P_0$, and
\[
 V_{\eps_k}(t,\nu)
 =\mathcal F_t(\P_k^t)-\frac{H_t(\P_k^t)}{\eps_k}
 \longrightarrow V_W(t,\nu).
\]
The same conclusion holds for every convergent subsequence.  Since
$\eps\mapsto V_\eps(t,\nu)$ is nondecreasing in $\eps$ (equivalently,
nonincreasing as $\eps\downarrow0$), the full limit follows.
\end{proof}

To formulate the dynamic programming principle, define, for
$\P\in\Pc_W(t,\nu)$,
\[
 \mu_r^\P:=\Lc^\P(X_r,I_r),
\qquad
 R_{t,s}(\P)
 :=\E^\P\!\left[
      \int_t^s f(r,X_r,\mu_r^\P)I_r
      \mathbf 1_{[\alpha,1]}(\E^\P[I_r])\,dr
   \right],
\]
and
\[
 m_{s-}^\P:=\Lc^\P(X_s,I_{s-}),
 \qquad
 m_s^\P:=\Lc^\P(X_s,I_s).
\]
Here $X_s=X_{s-}$ because the state process is continuous.

\begin{lemma}[Weak concatenation]
\label{lem:weak_concatenation}
Let $t\le s\le T$, $\P\in\Pc_W(t,\nu)$, and
$\mathsf Q\in\Pc_W(s,m_{s-}^\P)$.  There exists
$\mathsf R=\P\otimes_s\mathsf Q\in\Pc_W(t,\nu)$ such that
\begin{enumerate}
\item[(i)] $\mathsf R$ and $\P$ agree on $\Fc_{s-}$;
\item[(ii)] $\mathsf R^s=\mathsf Q^s$;
\item[(iii)] $J_t(\mathsf R)=R_{t,s}(\P)+J_s(\mathsf Q)$.
\end{enumerate}
\end{lemma}

\begin{proof}
Since $\Om_s$ is Polish, $\mathsf Q^s$ can be disintegrated with respect to
its initial coordinate:
\[
 \mathsf Q^s(dz)=\int_{\Sbf}K(y,dz)m_{s-}^\P(dy),
 \qquad
 K\big(y,\{z:(x_s,i_{s-})=y\}\big)=1
\]
for $m_{s-}^\P$-a.e. $y$.  Glue a $\P$-history up to $s-$ to a draw
from $K(Y_{s-},\cdot)$ that is conditionally independent of the history given
$Y_{s-}$.  The state paths match at $s$; the Brownian path is glued by
adding its increment from time $s$ to the value $W_s$ of the history; and the
survival paths match because their common pre-$s$ value is $I_{s-}$.

Admissibility follows from the martingale
characterization used in Lemma \ref{lem:projected_law_compactness}.  Conditioning
the martingale identities on $\Om_s$ under $\mathsf Q$ on $Y_{s-}=y$ shows that,
for $m_{s-}^\P$-a.e. $y$, the kernel $K(y,\cdot)$ solves the
martingale problem with the deterministic marginal flow of $\mathsf Q$.
Conditional product gluing therefore preserves the Brownian martingale
identities across $s$.  Before $s$, the marginal flow is that of $\P$;
after $s$, it is that of $\mathsf Q$, because the image of the concatenated law
under $\Pi_s$ is exactly $\mathsf Q^s$.  Thus the glued law belongs to
$\Pc_W(t,\nu)$.  The same marginal identities show that the running
reward before $s$ is that of $\P$ and after $s$ is that of $\mathsf Q$, which
proves \textnormal{(iii)}.
\end{proof}

We also use the instantaneous-stopping order.  For
$m',m\in\Pc_2(\Sbf)$ write $m'\preceq m$ if there exists a probability
measure $\lambda$ on
$\R^n\times\{0,1\}\times\{0,1\}$ such that
\[
 \lambda\circ(x,i^-)^{-1}=m,
 \qquad
 \lambda\circ(x,i)^{-1}=m',
 \qquad
 \lambda(i\le i^-)=1.
\]
Thus $m'$ is obtained from $m$ by stopping a possibly randomized portion of
the surviving mass while leaving the state $x$ unchanged.

\begin{lemma}
\label{lem:instantaneous_stopping_monotonicity}
If $m'\preceq m$, then
\[
       V_W(s,m)\ge V_W(s,m').
\]
Consequently, by strong--weak equivalence, the same inequality holds with
$V_S$ in place of $V_W$.
\end{lemma}

\begin{proof}
Fix $\mathsf Q\in\Pc_W(s,m')$ and a coupling $\lambda$ as above.  Disintegrate
$\lambda$ with respect to $(x,i)$ and, conditionally on
$(X_s,I_{s-})$ under $\mathsf Q$, introduce an ancestor status $i^-$ using the
resulting kernel.  Replace only the pre-$s$ status by $i^-$ and keep the
entire trajectory from time $s$ onward under $\mathsf Q$.  Since $I_s\le I_{s-}\le i^-$, the glued
survival path is admissible and implements the additional stopping at time
$s$.  The added variable is conditionally independent of future
Brownian increments, so the Brownian property is preserved.  The dynamics and the
objective on $[s,T]$ are unchanged, while the new initial law is $m$.
Thus every payoff attainable from $m'$ is attainable from $m$, and the
claim follows by taking the supremum.
\end{proof}

\begin{theorem}[Dynamic Programming Principle]
\label{thm:dpp}
Suppose Assumption \ref{ass:coefficient_reward} holds. Then, for every
$(t,\nu)\in[0,T]\times\Pc_2(\Sbf)$ and $s\in[t,T]$,
\begin{align}
 V_W(t,\nu)
 &=\sup_{\P\in\Pc_W(t,\nu)}
      \left\{R_{t,s}(\P)+V_W(s,m_{s-}^\P)\right\}
 \label{eq:dpp_weak_pre}
 \\
 &=\sup_{\P\in\Pc_W(t,\nu)}
      \left\{R_{t,s}(\P)+V_W(s,m_s^\P)\right\}.
 \label{eq:dpp_weak_post}
\end{align}
Consequently,
\begin{align}
 V_S(t,\nu)
 &=\sup_{\P\in\Pc_S(t,\nu)}
      \left\{R_{t,s}(\P)+V_S(s,m_{s-}^\P)\right\}
 \label{eq:dpp_strong_pre}
 \\
 &=\sup_{\P\in\Pc_S(t,\nu)}
      \left\{R_{t,s}(\P)+V_S(s,m_s^\P)\right\}.
 \label{eq:dpp_strong_post}
\end{align}
\end{theorem}

\begin{proof}
We first prove \eqref{eq:dpp_weak_pre}.  For
$\P\in\Pc_W(t,\nu)$, the same full law belongs to
$\Pc_W(s,m_{s-}^\P)$ when regarded as a rule from time $s$ onward.  Therefore
\[
 J_t(\P)
 =R_{t,s}(\P)+J_s(\P)
 \le R_{t,s}(\P)+V_W(s,m_{s-}^\P).
\]
Taking the supremum gives the ``$\le$'' inequality in
\eqref{eq:dpp_weak_pre}.

For the reverse inequality, fix $\P\in\Pc_W(t,\nu)$ and $\eta>0$.
Choose an $\eta$-optimal rule
$\mathsf Q\in\Pc_W(s,m_{s-}^\P)$, so that
\[
 J_s(\mathsf Q)\ge V_W(s,m_{s-}^\P)-\eta.
\]
By Lemma \ref{lem:weak_concatenation},
\[
 R_{t,s}(\P)+V_W(s,m_{s-}^\P)-\eta
 \le J_t(\P\otimes_s\mathsf Q)
 \le V_W(t,\nu).
\]
Take the supremum over $\P$ and then let $\eta\downarrow0$.

For \eqref{eq:dpp_weak_post}, note that for every
$\P\in\Pc_W(t,\nu)$, the relation
       $m_s^\P\preceq m_{s-}^\P$
is witnessed by the coupling $\Lc^\P(X_s,I_{s-},I_s)$.  Lemma
\ref{lem:instantaneous_stopping_monotonicity} therefore implies that the
right-hand side of \eqref{eq:dpp_weak_pre} is at least as large as the
right-hand side of \eqref{eq:dpp_weak_post}.

For the opposite inequality, fix $\P\in\Pc_W(t,\nu)$.  On the original
probability space, set
\[
 \bar I_r:=I_r,\quad r<s,
 \qquad
 \bar I_r:=I_{s-},\quad r\in[s,T].
\]
Keep $\bar X=X$ before $s$, and on $[s,T]$ let $\bar X$ be the unique
McKean--Vlasov solution driven by the same Brownian motion and controlled
by $\bar I$.  Then
$\bar\P:=\Lc^\P(\bar X,W,\bar I)$ belongs to
$\Pc_W(t,\nu)$,
\[
 R_{t,s}(\bar\P)=R_{t,s}(\P),
 \qquad
 m_s^{\bar\P}=m_{s-}^\P.
\]
Any difference at the single time $s$ does not affect the Lebesgue integral.
Hence every candidate value in the supremum in \eqref{eq:dpp_weak_pre} also
appears in the supremum in
\eqref{eq:dpp_weak_post}, proving equality.

It remains to prove the strong formulation.  For
$\P\in\Pc_S(t,\nu)$, its continuation is, in particular, a weak rule starting
from
$m_{s-}^\P$, and therefore
\[
 J_t(\P)
 \le R_{t,s}(\P)+V_W(s,m_{s-}^\P)
 =R_{t,s}(\P)+V_S(s,m_{s-}^\P).
\]
Taking the supremum gives
$V_S(t,\nu)$ no larger than the right-hand side of
\eqref{eq:dpp_strong_pre}.  The reverse inequality follows from
$\Pc_S(t,\nu)\subset\Pc_W(t,\nu)$, the weak DPP, and
$V_S=V_W$.  This proves \eqref{eq:dpp_strong_pre}.  The no-jump
construction above preserves strong admissibility because $I_{s-}$ is
$\Gc_{s-}^t$-measurable.  Moreover, the monotonicity of the common value
$V_S=V_W$ follows from Lemma
\ref{lem:instantaneous_stopping_monotonicity}.  Repeating the argument for the
post-jump law gives \eqref{eq:dpp_strong_post}.
\end{proof}

\section{Regularity of the value function}

By Theorem \ref{thm:equivalence_value_function}, we may write
\[
    V(t,m):=V_W(t,m)=V_S(t,m),
    \qquad (t,m)\in[0,T]\times\Pc_2(\Sbf).
\]
Set
   $p(m):=m(\R^n\times\{1\})$.
The map $p$ is continuous on $\Pc_2(\Sbf)$, and in fact
    $|p(m)-p(m')|\le \Wc_2(m,m'),$
because $(x,i)\mapsto i$ is $1$-Lipschitz on $\Sbf$.

The value can nevertheless fail to be continuous at the threshold.  To see
this, take
$b=\sigma=0$ and $f\equiv1$, fix $x_0\in\R^n$, and let
\[
    m=\alpha\delta_{(x_0,1)}+(1-\alpha)\delta_{(x_0,0)},
    \qquad
    m_\eps=(\alpha-\eps)\delta_{(x_0,1)}
          +(1-\alpha+\eps)\delta_{(x_0,0)}.
\]
Then
\[
    \Wc_2(m,m_\eps)^2=\eps,
    \qquad
    V(t,m)=\alpha(T-t),
    \qquad
    V(t,m_\eps)=0,
\]
for every $t<T$.

\medskip

The following Lemma is one of the most important results in this work.
\begin{lemma}[Stability with respect to the initial law]
\label{lem:value_initial_stability}
Suppose Assumption \ref{ass:coefficient_reward} holds.  Assume in addition
that, for every $r\in[0,T]$, the map
    $(x,m)\longmapsto f(r,x,m)$
is continuous on $\R^n\times\Pc_2(\Sbf)$.  Let $t_k\in[0,T]$ and let
$\nu_k,\mu_k,\nu\in\Pc_2(\Sbf)$ satisfy
\[
    \Wc_2(\nu_k,\nu)+\Wc_2(\mu_k,\nu)\longrightarrow0,
    \qquad p(\nu)>\alpha.
\]
Then
\[
    |V(t_k,\nu_k)-V(t_k,\mu_k)|\longrightarrow0.
\]
\end{lemma}
{
\begin{remark}
   The proof of Lemma~\ref{lem:value_initial_stability} requires more than the standard stability argument for controlled McKean--Vlasov dynamics, due to the presence of the constraint. In particular, after coupling two nearby initial laws, an admissible control associated with one initial distribution cannot in general be transferred directly to the other, since the resulting law may fail to satisfy the constraint, even when the perturbation of the initial condition is arbitrarily small. The main point is therefore to construct, from a nearly optimal admissible law for \(\nu_k\), a nearby admissible law for \(\mu_k\). This is achieved by coupling the initial conditions and introducing a small randomized modification of the survival process which restores the constraint while affecting the original control only with vanishing probability.  
\end{remark}
}

\begin{proof}
It suffices to prove
\[
    \limsup_{k\to\infty}
    \big(V(t_k,\nu_k)-V(t_k,\mu_k)\big)\le0,
\]
since the reverse inequality follows by interchanging $\nu_k$ and $\mu_k$.
By Lemma \ref{lem:constraint_reduction}, choose
$\P_k\in\widetilde\Pc_W(t_k,\nu_k)$ such that
\[
    J_{t_k}(\P_k)\ge V(t_k,\nu_k)-\frac1k.
\]
Lemma \ref{lem:input_law_invariance} allows us to restart the system from
time $t_k$ on a stochastic basis carrying the same input law from $t_k$
onward.  Couple
$(X_{t_k},I_{t_k-})$ with a random variable
$(\widetilde X_{t_k},\widetilde I_{t_k-})$ of law $\mu_k$ so that
\[
\begin{split}
    \E\big[|X_{t_k}-\widetilde X_{t_k}|^2
          +|I_{t_k-}-\widetilde I_{t_k-}|^2\big]
    =\Wc_2(\nu_k,\mu_k)^2,
\end{split}
\]
and adjoin a uniform random variable $U_k$ that is independent of all existing
variables and measurable at time $t_k$.  The conditional randomization used for
the coupling and $U_k$ are independent of the Brownian increments after $t_k$,
so the enlargement preserves the Brownian property.

Set
\[
    \delta_k:=\E|I_{t_k-}-\widetilde I_{t_k-}|,
    \qquad
    \lambda_k:=\frac{\delta_k}{p(\mu_k)-\alpha+\delta_k}.
\]
For all sufficiently large $k$, $p(\mu_k)>\alpha$, and
\[
    \delta_k\le\Wc_2(\nu_k,\mu_k)^2,
    \qquad \lambda_k\longrightarrow0.
\]
Let $q_k(s):=\E[I_s]$.  Since $\P_k$ is constrained, $q_k$ is
nonincreasing and right-continuous, and
\[
    q_k(s)\in\{0\}\cup[\alpha,1],
    \qquad s<T.
\]
Indeed, if $q_k(s_0)\in(0,\alpha)$ for some $s_0<T$, right-continuity would
make $(\alpha-q_k)^+q_k$ strictly positive on an interval of positive
Lebesgue measure, contradicting the constraint.  Define the deterministic
time
\[
    \theta_k:=\inf\{s\in[t_k,T]:q_k(s)=0\},
\]
with the convention $\inf\varnothing=T+1$, and set
\[
    \widetilde I_s
    :=\widetilde I_{t_k-}
      \Big(I_s+(1-I_s)\mathbf 1_{\{U_k\le\lambda_k\}}
                         \mathbf 1_{\{s<\theta_k\}}\Big),
    \qquad s\in[t_k,T].
\]
This is an adapted survival process.  If
$j_k(s):=\E[\widetilde I_{t_k-}I_s]$, then
$|j_k(s)-q_k(s)|\le\delta_k$.  Hence, for $s<\theta_k$ and $s<T$,
\[
\begin{split}
    \E[\widetilde I_s]
      =(1-\lambda_k)j_k(s)+\lambda_k p(\mu_k)
      \ge \alpha-\delta_k
          +\lambda_k(p(\mu_k)-\alpha+\delta_k)
       =\alpha.
\end{split}
\]
For $s\ge\theta_k$, one has $I_s=0$ almost surely and
$\widetilde I_s=0$.  Values at the single terminal time $T$ do not affect the
constraint.  Consequently, $\widetilde I$ satisfies the constraint.
Moreover, if
\[
    e_k:=\P\big(I\ne\widetilde I\text{ on }[t_k,T]\big),
\]
then
\begin{equation}\label{eq:regularity_control_error}
    e_k\le \delta_k+\lambda_kp(\mu_k)\longrightarrow0.
\end{equation}

Let $\widetilde X$ solve
\[
    \widetilde X_s
    =\widetilde X_{t_k}
     +\int_{t_k}^s
        b(r,\widetilde X_r,\widetilde\mu_r)\widetilde I_r\,dr
     +\int_{t_k}^s
        \sigma(r,\widetilde X_r,\widetilde\mu_r)
        \widetilde I_r\,dW_r,
    \qquad
    \widetilde\mu_r:=\Lc(\widetilde X_r,\widetilde I_r).
\]
Adjoin an independent Brownian history before $t_k$ and take the past state
and survival paths to be constant.  Then the canonical law
\[
    \widetilde\P_k:=\Lc(\widetilde X,W,\widetilde I)
\]
belongs to $\widetilde\Pc_W(t_k,\mu_k)$.  The natural canonical filtration
is contained in the enlarged filtration, so Brownianity is preserved.

The coupled states satisfy
\begin{equation}\label{eq:regularity_state_error}
    \E\left[\sup_{s\in[t_k,T]}|X_s-\widetilde X_s|^2\right]
    \longrightarrow0.
\end{equation}
Since $\nu_k$ and $\mu_k$ converge in $\Wc_2$, the families
$\{|X_{t_k}|^2\}_{k\ge1}$ and
$\{|\widetilde X_{t_k}|^2\}_{k\ge1}$ are uniformly integrable.  The same is
true for
$\sup_{s\in[t_k,T]}|X_s|^2$ and
$\sup_{s\in[t_k,T]}|\widetilde X_s|^2$.  To verify this refinement of
Lemma \ref{lem:moment_estimate}, truncate the initial state at radius $R$
and solve the equation with the same Brownian motion and the same survival
process.  The Lipschitz estimate gives an $L^2$ error bounded by the
truncation error of the initial state, uniformly in $k$, while the
Burkholder--Davis--Gundy and Gronwall inequalities yield a uniform
fourth-moment estimate for bounded initial states.  Letting $R\to\infty$
proves uniform integrability.

The usual BDG estimate, together with the Lipschitz property of $b,\sigma$
and the coupling bound
\[
    \Wc_2(\Lc(X_s,I_s),\Lc(\widetilde X_s,\widetilde I_s))^2
    \le \E|X_s-\widetilde X_s|^2
       +\E|I_s-\widetilde I_s|
\]
yields
\[
\begin{split}
 \E\left[\sup_{u\in[t_k,s]}|X_u-\widetilde X_u|^2\right]
 &\le
 C\Wc_2(\nu_k,\mu_k)^2
 +C\int_{t_k}^s
       \E\left[\sup_{u\in[t_k,r]}|X_u-\widetilde X_u|^2\right]dr \\
 &\qquad
 +C\int_{t_k}^s
       \E\Big[(1+|\widetilde X_r|^2)
               |I_r-\widetilde I_r|\Big]dr
 +Ce_k.
\end{split}
\]
The last two terms converge to zero.  For the weighted term, for every
$R>0$ it is bounded by
\[
    C(1+R)Te_k
    +CT\E\left[
       \sup_{s\in[t_k,T]}|\widetilde X_s|^2
       \mathbf 1_{\{\sup_{s\in[t_k,T]}|\widetilde X_s|^2>R\}}
    \right].
\]
First let $k\to\infty$ and then $R\to\infty$. Gronwall's inequality
gives \eqref{eq:regularity_state_error}.  Together with
\eqref{eq:regularity_control_error} and the bound on the survival-path metric,
this also gives
\begin{equation}\label{eq:regularity_projected_law_error}
    \Wc_{2,t_k}(\P_k^{t_k},\widetilde\P_k^{t_k})\longrightarrow0.
\end{equation}

We now compare the rewards.  The family of marginal laws
$\Lc(X_s,I_s)$ and $\Lc(\widetilde X_s,\widetilde I_s)$ appearing above is
relatively compact in $\Pc_2(\Sbf)$, because the corresponding second
moments are uniformly integrable.  Let $\mathcal K$ be its compact closure.
Equations \eqref{eq:regularity_state_error} and
\eqref{eq:regularity_control_error} imply
\[
    \sup_{s\in[t_k,T]}
    \Wc_2(\Lc(X_s,I_s),
          \Lc(\widetilde X_s,\widetilde I_s))\longrightarrow0.
\]
Fix $R>0$ and $\eta>0$.  By Scorza--Dragoni's theorem, there is a compact
$E_\eta\subset[0,T]$ with Lebesgue measure at least $T-\eta$ such that $f$
is jointly continuous, hence uniformly continuous, on
$E_\eta\times\overline B_R\times\mathcal K$.  Uniform continuity on this set,
together with the preceding estimates, gives the required convergence.  On
$[0,T]\setminus E_\eta$, the quadratic
growth bound gives
a contribution bounded by $C\eta$ uniformly in $k$, while the contribution
of $\{|X_s|>R\}\cup\{|\widetilde X_s|>R\}$ tends to zero uniformly as
$R\to\infty$ by uniform integrability.  The quadratic growth of $f$ and this
uniform integrability therefore give
\[
\begin{split}
 &\int_{t_k}^T
   \E\Big[
      |f(s,X_s,\Lc(X_s,I_s))
       -f(s,\widetilde X_s,
             \Lc(\widetilde X_s,\widetilde I_s))|
   \Big]ds
   \longrightarrow0.
\end{split}
\]
Also,
\[
\begin{split}
 &\int_{t_k}^T
   \E\Big[
      |f(s,\widetilde X_s,
             \Lc(\widetilde X_s,\widetilde I_s))|
      |I_s-\widetilde I_s|
   \Big]ds
   \longrightarrow0,
\end{split}
\]
by the growth bound, uniform integrability, and
\eqref{eq:regularity_control_error}.  Since both laws satisfy the constraint,
Lemma \ref{lem:constraint_reduction} gives
\[
    |J_{t_k}(\P_k)-J_{t_k}(\widetilde\P_k)|\longrightarrow0.
\]
Therefore,
\[
    V(t_k,\nu_k)-V(t_k,\mu_k)
    \le \frac1k
       +J_{t_k}(\P_k)-J_{t_k}(\widetilde\P_k)
    \longrightarrow0.
\]
This proves the desired inequality, and the reverse inequality follows by
symmetry.
\end{proof}

\begin{theorem}[Regularity]
\label{thm:regularity}
Suppose Assumption \ref{ass:coefficient_reward} holds.  Assume in addition
that, for every $r\in[0,T]$, the map
    $(x,m)\longmapsto f(r,x,m)$
is continuous on $\R^n\times\Pc_2(\Sbf)$.

Then, with respect to the product of the Euclidean topology on $[0,T]$ and
the $\Wc_2$ topology, $V$ is continuous at every $(t,m)$ such that
\[
    m(\R^n \x \{1\})\ne\alpha.
\]
Moreover, $V$ is continuous at every $(T,m)$.
\end{theorem}

\begin{remark}
    The lower semicontinuity of $f$ in Assumption
\ref{ass:coefficient_reward}\textnormal{(ii)} is not sufficient for the
continuity asserted in the theorem.  For example, assume $\alpha<1$, fix
$p\in(\alpha,1]$, take $b=\sigma=0$, and set
\[
    f(s,x,m):=\mathbf 1_{\{x\ne0\}}.
\]
This function is bounded and lower semicontinuous.  If
\[
    m=p\delta_{(0,1)}+(1-p)\delta_{(0,0)},
    \qquad
    m_k=p\delta_{(1/k,1)}+(1-p)\delta_{(0,0)},
\]
then $\Wc_2(m_k,m)=\sqrt p/k$, but
\[
    V(t,m)=0,
    \qquad
    V(t,m_k)=p(T-t),
    \qquad t<T.
\]
Thus the continuity assumption used in the existence result is also needed
for the regularity theorem.
\end{remark}
{
\begin{corollary}[One-sided continuity at the threshold]
\label{cor:threshold_continuity}
Suppose the assumptions of Theorem~\ref{thm:regularity} hold.  Fix
$t\in[0,T]$ and let $m_k,m\in\Pc_2(\Sbf)$ satisfy
$$
    \Wc_2(m_k,m)\longrightarrow0,
    \qquad
    p(m_k)\ge\alpha,
    \qquad
    p(m)=\alpha.
$$
Then
$
    V(t,m_k)\longrightarrow V(t,m).
$
\end{corollary}

\begin{proof}
The case $t=T$ follows from Theorem~\ref{thm:regularity}, so we only consider the case $t<T$. By continuity of $p$, one has $p(m_k)\to\alpha$.

We first prove the upper bound.  Choose
$\P_k\in\widetilde\Pc_W(t,m_k)$ such that
$$
    J_t(\P_k)\ge V(t,m_k)-\frac1k,
$$
and set
$
    \theta_k:=\inf\{s\in[t,T]:\E^{\P_k}[I_s]=0\},
$
with the convention $\inf\varnothing=T+1$.  As in the proof of
Lemma~\ref{lem:value_initial_stability},
$$
    \E^{\P_k}[I_s]\in\{0\}\cup[\alpha,1],
    \qquad s<T.
$$
Since $I$ is nonincreasing, it follows that, up to the value at the single
time $T$,
$$
    \P_k\Big(
       I\ne I_{t-}\mathbf 1_{[t,\theta_k)}
       \text{ on }[t,T)
    \Big)
    \le p(m_k)-\alpha.
$$
Couple $(X_t,I_{t-})$ with
$(\widetilde X_t,\widetilde I_{t-})$ of law $m$ so that the corresponding
$L^2$ distance is $\Wc_2(m_k,m)^2$, and define
$$
    \widetilde I_s
       :=\widetilde I_{t-}\mathbf 1_{\{s<\theta_k\}},
    \qquad s\in[t,T].
$$
Since $p(m)=\alpha$, the process $\widetilde I$ satisfies the constraint.
Moreover,
$$
    \P\big(I\ne\widetilde I\text{ on }[t,T)\big)
    \le
    \E|I_{t-}-\widetilde I_{t-}|+p(m_k)-\alpha
    \longrightarrow0.
$$
Let $\widetilde X$ solve the corresponding state equation and let
$\widetilde\P_k$ denote the resulting canonical law.  The estimates used in
the proof of Lemma~\ref{lem:value_initial_stability}, from
\eqref{eq:regularity_state_error} onward, yield
$$
    |J_t(\P_k)-J_t(\widetilde\P_k)|\longrightarrow0.
$$
Since $\widetilde\P_k\in\widetilde\Pc_W(t,m)$,
$$
    \limsup_{k\to\infty}V(t,m_k)\le V(t,m).
$$
For the reverse inequality, fix $\eps>0$ and choose
$\P\in\widetilde\Pc_W(t,m)$ such that
$$
    J_t(\P)\ge V(t,m)-\eps.
$$
Since $\E^\P[I_s]\le p(m)=\alpha$ and
$\E^\P[I_s]\in{0}\cup[\alpha,1]$ for $s<T$, one has
$\E^\P[I_s]\in{0,\alpha}$.  Hence, for some deterministic
$\theta\in[t,T+1]$,
$$
    I_s=I_{t-}\mathbf 1_{\{s<\theta\}},
    \qquad s<T,
    \quad \P\text{-a.s.}
$$
Couple the initial law $m$ with $m_k$ and, starting from $m_k$, use the
survival process
$$
    I_s^k:=I_{t-}^k\mathbf 1_{\{s<\theta\}}.
$$
This control is admissible, since its survival mass equals $p(m_k)\ge\alpha$
before $\theta$ and $0$ afterwards.  Applying again the stability estimates
from Lemma~\ref{lem:value_initial_stability} gives, for the corresponding
law $\P_k$,
$$
    J_t(\P_k)\longrightarrow J_t(\P).
$$
Therefore
$$
    \liminf_{k\to\infty}V(t,m_k)
       \ge J_t(\P)
       \ge V(t,m)-\eps.
$$
Letting $\eps\downarrow0$ proves the result.
\end{proof}



}
\begin{proof}
We first treat the case $p(m)<\alpha$.  For every admissible law $\P$,
\[
    \E^\P[I_s]\le\E^\P[I_{t-}]=p(m)<\alpha,
    \qquad s\in[t,T].
\]
Hence the threshold factor in the objective is identically zero and
$V(t,m)=0$.  Since $p$ is continuous, the same conclusion holds throughout a
neighborhood of $m$.  Thus $V$ is jointly continuous at every point with
$p(m)<\alpha$.

Now fix $m$ with $p(m)>\alpha$.  Lemma
\ref{lem:value_initial_stability} gives continuity with respect to the
initial law, uniformly along arbitrary sequences of starting times.  We need
only establish continuity in time at fixed $m$.

Let $t<s$.  For $\P\in\Pc_W(t,m)$ define
   $\bar m_s^\P:=\Lc^\P(X_s,I_{t-})$.
The moment estimate and the state equation imply, uniformly in
$\P\in\Pc_W(t,m)$,
\begin{equation}\label{eq:regularity_short_reward}
    |R_{t,s}(\P)|
    \le C(1+M_2(m))(s-t)
\end{equation}
and
\begin{equation}\label{eq:regularity_short_state}
    \Wc_2(\bar m_s^\P,m)^2
    \le \E^\P|X_s-X_t|^2
    \le C(1+M_2(m))(s-t).
\end{equation}
Moreover,
    $m_{s-}^\P\preceq \bar m_s^\P,$
because $I_{s-}\le I_{t-}$ pathwise.  The monotonicity lemma and the weak
DPP therefore give
\[
\begin{split}
    V(t,m)
      &\le C(1+M_2(m))(s-t)
          +\sup_{\P\in\Pc_W(t,m)}V(s,\bar m_s^\P).
\end{split}
\]
By \eqref{eq:regularity_short_state} and Lemma
\ref{lem:value_initial_stability},
\[
    \sup_{\P\in\Pc_W(t,m)}
       |V(s,\bar m_s^\P)-V(s,m)|\longrightarrow0
    \qquad\text{as }s\downarrow t.
\]
Indeed, otherwise one could choose a contradicting sequence of laws and
apply the lemma.  Consequently,
\begin{equation}\label{eq:regularity_time_upper}
    V(t,m)\le V(s,m)+o(1)
    \qquad\text{as }s\downarrow t.
\end{equation}

For the opposite inequality, take the admissible rule that keeps every
initially surviving particle alive on $[t,s)$ and stops it at $s$, namely
$I_r=I_{t-}\mathbf 1_{\{r<s\}}$.  Denote its law by $\P^0$.  Then
$m_{s-}^{\P^0}=\bar m_s^{\P^0}$, and the DPP, together with
\eqref{eq:regularity_short_reward}--\eqref{eq:regularity_short_state}
and Lemma
\ref{lem:value_initial_stability}, yields
\begin{equation}\label{eq:regularity_time_lower}
\begin{split}
    V(t,m)
      &\ge R_{t,s}(\P^0)+V(s,m_{s-}^{\P^0})
       \ge V(s,m)-o(1).
\end{split}
\end{equation}
Thus $V(s,m)\to V(t,m)$ as $s\downarrow t$.  The same estimates, with the
initial time increasing to $t$, give left-continuity at every $t>0$.

If $(t_k,m_k)\to(t,m)$ and $p(m)>\alpha$, then
\[
\begin{split}
    |V(t_k,m_k)-V(t,m)|
    \le |V(t_k,m_k)-V(t_k,m)|
       +|V(t_k,m)-V(t,m)|,
\end{split}
\]
and the two terms converge to zero by Lemma
\ref{lem:value_initial_stability} and the continuity in time just proved.
This proves joint continuity when $p(m)>\alpha$.

Finally, $V(T,m)=0$ for every $m$.  Because stopping all particles immediately
yields zero, $V\ge0$; the growth estimate and Lemma \ref{lem:moment_estimate}
give
\[
    0\le V(t,m)\le C(T-t)(1+M_2(m)).
\]
Hence $V(t_k,m_k)\to0=V(T,m)$ whenever
$(t_k,m_k)\to(T,m)$, including the threshold case $p(m)=\alpha$.
\end{proof}

\section{Limit theory}

We work throughout at the initial time $0$ and identify each admissible law
with its image on $\Omega_0$.  In addition to Assumption
\ref{ass:coefficient_reward}, suppose that, for every $r\in[0,T]$, the map
$(x,m)\mapsto f(r,x,m)$ is continuous.  We take the initial pairs
$\{(X^i_0,I^i_{0-})\}_{i\ge1}$ to be independent, the Brownian motions
$\{W^i\}_{i\ge1}$ to be independent, and the two families to be mutually
independent.  Write
\[
    \Lc^{\P^0}(X^i_0,I^i_{0-})=\nu_i\in\Pc_2(\Sbf).
\]
For $N\ge1$, let $\mathbb F^N=(\Fc_s^N)_{s\in[0,T]}$ be the augmented
filtration generated by
$\{(X^i_0,I^i_{0-},W^i):1\le i\le N\}$.

An admissible finite-population stopping strategy is a family
$\Ic^N=(I^{1,N},\ldots,I^{N,N})$ such that each $I^{i,N}$ is
$\mathbb F^N$-adapted, has paths in $\I([0,T])$, and satisfies
$I^{i,N}_{0-}=I^i_{0-}$.  Given $\Ic^N$, let
$\{X^{i,N}\}_{1\le i\le N}$ be the unique strong solution of
\begin{equation}\label{eq:finite_population_sde}
\begin{split}
    X^{i,N}_t
    =X^i_0
       +\int_0^t b(r,X^{i,N}_r,\mu^N_r)I^{i,N}_r\,dr
       +\int_0^t \sigma(r,X^{i,N}_r,\mu^N_r)I^{i,N}_r\,dW^i_r,
      \qquad 0\le t\le T,
\end{split}
\end{equation}
where
\begin{equation}\label{eq:finite_population_empirical_state}
    \mu^N_r:=\frac1N\sum_{j=1}^N
                 \delta_{(X^{j,N}_r,I^{j,N}_r)},
    \qquad
    q^N_r:=\frac1N\sum_{j=1}^N I^{j,N}_r.
\end{equation}
The finite-population objective and value are
\begin{align}
    J_N(\Ic^N)
    &:=\frac1N\sum_{i=1}^N\E^{\P^0}\!\left[
       \int_0^T f(r,X^{i,N}_r,\mu^N_r)I^{i,N}_r
       \mathbf 1_{[\alpha,1]}(q^N_r)\,dr\right],
       \label{eq:finite_population_objective}\\
    V^N(\nu_1,\ldots,\nu_N)
    &:=\sup_{\Ic^N}J_N(\Ic^N).
       \label{eq:finite_population_value}
\end{align}
The empirical measure in \eqref{eq:finite_population_sde} records only the
state and survival coordinates; it does not include the Brownian path.

\begin{lemma}\label{lem:finite_population_cutoff}
For every admissible $\Ic^N$, there exists another admissible strategy
$\widetilde{\Ic}^N$ such that
\begin{equation}\label{eq:finite_constraint_pathwise}
    \frac1N\sum_{i=1}^N\widetilde I^{i,N}_r
    \in\{0\}\cup[\alpha,1],
    \qquad 0\le r<T,
\end{equation}
and
\begin{equation}\label{eq:finite_cutoff_payoff_identity}
    J_N(\Ic^N)
    =\frac1N\sum_{i=1}^N\E^{\P^0}\!\left[
       \int_0^T f(r,\widetilde X^{i,N}_r,\widetilde\mu^N_r)
       \widetilde I^{i,N}_r\,dr\right],
\end{equation}
where $\widetilde X^{i,N}$ and $\widetilde\mu^N$ are generated by
$\widetilde{\Ic}^N$.
\end{lemma}

\begin{proof}
Let
\[
    \tau_N:=\inf\{r\in[0,T]:q^N_r<\alpha\},
\]
with $\inf\emptyset=T+1$, and set
\[
    \widetilde I^{i,N}_{0-}:=I^i_{0-},
    \qquad
    \widetilde I^{i,N}_r:=I^{i,N}_r\mathbf 1_{\{r<\tau_N\}},
    \qquad 0\le r\le T.
\]
The strict inequality in the definition of $\tau_N$ is essential: if the
empirical survival mass is exactly $\alpha$, the indicator in
\eqref{eq:finite_population_objective} is still equal to one.
Because $q^N$ is adapted, right-continuous, and non-increasing, $\tau_N$ is an
$\mathbb F^N$-stopping time and $\widetilde{\Ic}^N$ is admissible.  Before
$\tau_N$ the original and modified systems coincide by pathwise uniqueness.
From $\tau_N$ onward, the original empirical survival mass is below $\alpha$
and remains there, while every modified particle is stopped.  The two
payoffs therefore agree, and \eqref{eq:finite_constraint_pathwise} follows.
The value at the single time $T$ is irrelevant for both time integrals.
\end{proof}

For a constrained finite-population strategy, define the random empirical path
law
\begin{equation}\label{eq:random_empirical_path_law}
    \mathbf P^N
    :=\frac1N\sum_{i=1}^N
          \delta_{(X^{i,N},W^i,I^{i,N})}
    \in\Pc_2(\Omega_0).
\end{equation}
Then, pathwise,
\begin{equation}\label{eq:finite_objective_as_random_functional}
    \frac1N\sum_{i=1}^N
       \int_0^T f(r,X^{i,N}_r,\mu^N_r)I^{i,N}_r\,dr
    =\mathcal F_0(\mathbf P^N).
\end{equation}
Here $\mathbf P^N$ is the random empirical measure rather than its
expectation.

\begin{lemma}[Wasserstein law of large numbers]\label{lem:triangular_wasserstein_lln}
Let $(E,d_E)$ be Polish, let $Y^{1,N},\ldots,Y^{N,N}$ be independent
$E$-valued random variables, and assume that, for some $\delta>0$ and
$e_0\in E$,
\[
    \sup_{N\ge1}\frac1N\sum_{i=1}^N
      \E[d_E(Y^{i,N},e_0)^{2+\delta}]<\infty.
\]
If
\[
    \frac1N\sum_{i=1}^N\Lc(Y^{i,N})
    \longrightarrow Q
    \quad\hbox{in }\Wc_2,
\]
then
\[
    \E\!\left[
       \Wc_2^2\!\left(\frac1N\sum_{i=1}^N\delta_{Y^{i,N}},Q\right)
    \right]\longrightarrow0.
\]
\end{lemma}

\begin{proof}
By the triangle inequality, it suffices to compare the empirical measure with
\[
    \frac1N\sum_{i=1}^N\Lc(Y^{i,N}).
\]
Fix $\eta>0$.  The convergence of the mean laws in $\Wc_2$, together with
the uniform $(2+\delta)$-moment bound, gives a compact set $K\subset E$ such
that the second-moment contribution of $K^c$ is at most $\eta$, uniformly in
$N$.
Partition $K$ into finitely many Borel sets of diameter at most $\eta$, select
one point in each cell, and map each cell to its selected point; map $K^c$ to
$e_0$. The transport cost created by this map is bounded by $C\eta$
uniformly in $N$. On the
resulting finite state space, the expected total variation distance between the
empirical frequencies and their expectations is $O(N^{-1/2})$, because the
variables are independent.  The squared Wasserstein distance on that finite
state space is bounded by a constant times this total variation distance.
Letting first $N\to\infty$ and then $\eta\downarrow0$ proves the claim.
\end{proof}

A $(2+\delta)$-moment bound remains valid after decreasing $\delta$, so we
take $0<\delta\le2$ without loss of generality.

\begin{lemma}[Compactness and identification of finite-population limits]
\label{lem:finite_population_compactness_identification}
Assume that, for some $\delta>0$,
\begin{equation}\label{eq:uniform_initial_higher_moment}
    \sup_{N\ge1}\frac1N\sum_{i=1}^N
       \int_{\Sbf}|x|^{2+\delta}\,\nu_i(dx,di)<\infty,
\end{equation}
and
\begin{equation}\label{eq:average_initial_law_convergence}
    \Wc_2\!\left(\frac1N\sum_{i=1}^N\nu_i,\nu\right)
    \longrightarrow0.
\end{equation}
Let $\mathbf P^N$ be generated by arbitrary finite-population strategies
satisfying \eqref{eq:finite_constraint_pathwise}.  Then the laws of
$\mathbf P^N$, viewed as probability measures on
$(\Pc_2(\Omega_0),\Wc_{2,0})$, are tight.  If along a subsequence
\[
    \Lc^{\P^0}(\mathbf P^N)\Longrightarrow\Gamma,
\]
then
\begin{equation}\label{eq:limit_law_supported_on_weak_rules}
    \Gamma\big(\widetilde\Pc_W(0,\nu)\big)=1.
\end{equation}
Moreover,
\begin{equation}\label{eq:objective_convergence_random_empirical}
    \E^{\P^0}[\mathcal F_0(\mathbf P^N)]
    \longrightarrow
    \int_{\Pc_2(\Omega_0)}\mathcal F_0(Q)\,\Gamma(dQ)
\end{equation}
whenever the laws of $\mathbf P^N$ converge to $\Gamma$.
\end{lemma}

\begin{proof}
The standard SDE estimates, applied after averaging over the particle index,
give
\begin{equation}\label{eq:finite_uniform_path_moments}
    \sup_{N\ge1}\frac1N\sum_{i=1}^N
    \E^{\P^0}\!\left[
       \|X^{i,N}\|_\infty^{2+\delta}
       +\|W^i\|_\infty^{2+\delta}
    \right]<\infty.
\end{equation}
They also give, for $0\le s\le t\le T$,
\begin{equation}\label{eq:finite_uniform_modulus}
    \frac1N\sum_{i=1}^N
       \E^{\P^0}[|X^{i,N}_t-X^{i,N}_s|^{2+\delta}]
    \le C|t-s|^{1+\delta/2}.
\end{equation}
The same estimate holds for the Brownian coordinates.  Since
$\I([0,T])$ is compact, \eqref{eq:finite_uniform_path_moments},
\eqref{eq:finite_uniform_modulus}, Kolmogorov's tightness criterion, and the
uniform integrability supplied by the $(2+\delta)$-moment bound imply relative
compactness in $\Wc_{2,0}$ of the deterministic mean measures
$\E^{\P^0}[\mathbf P^N]$.

These estimates also give tightness of the laws of the random measures in the
Wasserstein topology.  Fix a reference point $\omega_0\in\Omega_0$.
The preceding relative compactness allows us, for every $\eta>0$, to choose
increasing compact sets $K_\ell\subset\Omega_0$ such that
\[
\begin{split}
    \sup_N\E^{\P^0}[\mathbf P^N(K_\ell^c)]
    +\sup_N\E^{\P^0}\!\left[
       \int_{K_\ell^c}d_0(\omega,\omega_0)^2
          \,\mathbf P^N(d\omega)
    \right]
    \le \eta 2^{-3\ell-2}.
\end{split}
\]
By Markov's inequality and a union bound, with probability at least
$1-\eta$ the random measure $\mathbf P^N$ satisfies, simultaneously for every
$\ell$,
\[
    \mathbf P^N(K_\ell^c)\le2^{-\ell},
    \qquad
    \int_{K_\ell^c}d_0(\omega,\omega_0)^2
       \,\mathbf P^N(d\omega)\le2^{-\ell}.
\]
The family of measures satisfying these bounds is tight and has uniformly
integrable second moments; its closure is therefore compact in
$\Wc_{2,0}$.  This proves tightness of
$\{\Lc^{\P^0}(\mathbf P^N)\}_N$ on $\Pc_2(\Omega_0)$.

Passing to a subsequence, the Skorokhod representation theorem allows us to
assume that $\mathbf P^N\to\mathbf P$ almost surely in $\Wc_{2,0}$ and
$\Lc(\mathbf P)=\Gamma$.  By Lemma
\ref{lem:triangular_wasserstein_lln}, applied to the independent initial pairs,
\[
    \Wc_2\!\left(
       \frac1N\sum_{i=1}^N\delta_{(X^i_0,I^i_{0-})},\nu
    \right)\longrightarrow0
    \quad\hbox{in }L^2(\P^0).
\]
The marginal map $Q\mapsto\Lc^Q(X_0,I_{0-})$ is continuous under the
Wasserstein topologies.  Hence, on the Skorokhod space, the initial marginals
of $\mathbf P^N$ converge almost surely to the initial marginal of
$\mathbf P$, while their distance from $\nu$ converges to zero in
probability by the last display.  The $(X_0,I_{0-})$-marginal of $\mathbf P$
is therefore $\nu$ almost surely.  Also, $H_0(\mathbf P^N)=0$ almost surely by
\eqref{eq:finite_constraint_pathwise}.  The continuity of $H_0$, established in
Lemma \ref{lem:projected_law_compactness}, yields
$H_0(\mathbf P)=0$ almost surely.

The path metric also controls the marginal flows.  For
$Q,Q'\in\Pc_2(\Omega_0)$, coupling the full paths and then projecting at each
time gives
\begin{equation}\label{eq:integrated_marginal_bound}
    \int_0^T\Wc_2^2
       \big(\Lc^Q(X_r,I_r),\Lc^{Q'}(X_r,I_r)\big)\,dr
    \le C\Wc_{2,0}^2(Q,Q')+C\Wc_{2,0}(Q,Q').
\end{equation}
Here the state contribution is bounded by the uniform path distance, while
\[
    \int_0^T|i_r-i'_r|\,dr\le d_{\I,0}(i,i').
\]
We next identify the dynamics for $\Gamma$-almost every realization of the
random limiting law.  For
$Q\in\Pc_2(\Omega_0)$ and $\phi\in C_c^2(\R^n\times\R^d)$, write
$Q_r:=\Lc^Q(X_r,I_r)$ and set
\begin{align*}
    \mathcal L_r^{Q_r}\phi(x,w,i)
    :={}&D_x\phi(x,w)\cdot b(r,x,Q_r)i
       +\frac12D^2_{xx}\phi(x,w):
          \sigma\sigma^\top(r,x,Q_r)i\\
      &+\frac12\Delta_w\phi(x,w)
       +D^2_{xw}\phi(x,w):\sigma(r,x,Q_r)i,
\end{align*}
and
\begin{equation}\label{eq:quenched_canonical_martingale}
\begin{split}
    M_t^{Q,\phi}
    :=\phi(X_t,W_t)-\phi(X_0,0)
    -\int_0^t
          \mathcal L_r^{Q_r}\phi(X_r,W_r,I_r)\,dr.
\end{split}
\end{equation}
For each particle, It\^o's formula shows that
$M_t^{\mathbf P^N,\phi}(X^{i,N},W^i,I^{i,N})$ is an
$\mathbb F^N$-martingale.  If $0\le s<t\le T$ and $h$ is bounded, continuous,
and $\Fc_s$-measurable on $\Omega_0$, then
$h(X^{i,N},W^i,I^{i,N})$ is $\Fc^N_s$-measurable.  The martingale increments
for different particle indices are orthogonal: the controls may use the whole
population filtration, but the corresponding stochastic integrals are driven
by Brownian motions with zero mutual quadratic covariation.  Consequently,
\begin{equation}\label{eq:empirical_martingale_variance}
\begin{split}
    \E^{\P^0}\!\left[
       \left|\int_{\Omega_0}h(\omega)
          \big(M_t^{\mathbf P^N,\phi}(\omega)
              -M_s^{\mathbf P^N,\phi}(\omega)\big)
          \mathbf P^N(d\omega)\right|^2
    \right]
    \le \frac{C_\phi\|h\|_\infty^2}{N}.
\end{split}
\end{equation}
The constant $C_\phi$ is finite and independent of $N$ by
\eqref{eq:finite_uniform_path_moments}.

The variance estimate \eqref{eq:empirical_martingale_variance} yields
martingale identities for almost every realization of the limiting random
law.  Let
\[
    \lambda_\Gamma(B)
       :=\int_{\Pc_2(\Omega_0)}
           Q\big(\vartheta_0(I)\in B\big)\,\Gamma(dQ),
    \qquad B\in\mathcal B\big([0,T]\cup\{T+1\}\big).
\]
The set of atoms of $\lambda_\Gamma$ is at most countable.  Choose a countable
dense set $D\subset(0,T)$ which contains no such atom.  Then, for every
$s\in D$,
\begin{equation}\label{eq:no_quenched_atom_on_D}
    Q\big(\vartheta_0(I)=s\big)=0
    \quad\text{for $\Gamma$-almost every }Q.
\end{equation}
For each $s\in D$, fix a countable algebra $\mathcal H_s$ of bounded continuous
$\Fc_s$-measurable functions generated by bounded continuous cylinder
functions of the $X$- and $W$-coordinates at rational times not exceeding
$s$, the coordinate $I_{0-}$, and continuous functions of
$\vartheta_0(I)$ which are constant on $[s,T]\cup\{T+1\}$.  If
$Q(\vartheta_0(I)=s)=0$, this algebra generates $\Fc_s$ modulo $Q$: the only
information lost by forcing continuity at $s$ is the distinction between
$\vartheta_0(I)=s$ and $\vartheta_0(I)>s$, and the former event is $Q$-null.
Rational-time cylinder functions generate the continuous-coordinate
sigma-fields, while continuous functions of $\vartheta_0(I)$ that are constant
on $[s,T]\cup\{T+1\}$ generate the survival-coordinate sigma-field modulo the
null boundary event; the monotone-class theorem gives the claim. Also fix a
countable subset of $C_c^2(\R^n\times\R^d)$ that is dense for uniform
convergence on compact sets of the functions and their derivatives through
order two. Localization and the coefficient bounds extend identities from
this core to all of $C_c^2(\R^n\times\R^d)$.

For $s<t$ in $D$, $h\in\mathcal H_s$, and a test function $\phi$ in this
countable core, define
\[
    \Phi_{s,t}^{\phi,h}(Q)
       :=\int_{\Omega_0}h(\omega)
          \big(M_t^{Q,\phi}(\omega)-M_s^{Q,\phi}(\omega)\big)
          \,Q(d\omega).
\]
This functional is continuous under $\Wc_{2,0}$.  The endpoint-evaluation maps
for the continuous coordinates and the function $h$ are continuous, while
\eqref{eq:integrated_marginal_bound} gives convergence of the marginal flows in
$L^2([0,T];\Wc_2)$.  The Lipschitz and growth bounds for $b$ and $\sigma$ then
give convergence in $L^1$ of the integrated generators; for the
$\sigma\sigma^\top$ term, use Cauchy--Schwarz together with convergence of the
second moments supplied by $\Wc_{2,0}$ convergence.

Consequently, on the Skorokhod space,
$\Phi_{s,t}^{\phi,h}(\mathbf P^N)\to
 \Phi_{s,t}^{\phi,h}(\mathbf P)$ almost surely.  On the other hand,
\eqref{eq:empirical_martingale_variance} shows that the left-hand side converges
to zero in $L^2$.  Hence
\[
    \Phi_{s,t}^{\phi,h}(\mathbf P)=0
    \quad\text{almost surely}.
\]
Take the intersection of the full-measure sets obtained for the countable choices of
$s,t,h,\phi$, the atom-free conditions in
\eqref{eq:no_quenched_atom_on_D}, and the initial-law and constraint
conditions.  This yields a Borel set $\mathcal G\subset\Pc_2(\Omega_0)$ with
$\Gamma(\mathcal G)=1$.  Every $Q\in\mathcal G$ satisfies all the martingale
identities, has no stopping-time atom at any $s\in D$, has initial marginal
$\nu$, and satisfies the constraint.

Fix $Q\in\mathcal G$.  For $s<t$ in $D$, the monotone-class theorem extends the
identity from $\mathcal H_s$ to every bounded $\Fc_s$-measurable $h$, and the
$C^2$ approximation extends it to every $\phi\in C_c^2$.  Thus
$M^{Q,\phi}$ has the martingale increment property at times in $D$.  Since
$M^{Q,\phi}$ has continuous paths, this already gives the martingale property
at arbitrary times.  More explicitly, for $0\le s<t\le T$, a bounded
$\Fc_s$-measurable $h$, and sequences
$s_k\in D\cap(s,(s+t)/2)$ with $s_k\downarrow s$ and
$t_k\in D\cap((s+t)/2,t)$ with $t_k\uparrow t$, one has
$h\in\Fc_{s_k}$ and
\[
    \E^Q[h(M_{t_k}^{Q,\phi}-M_{s_k}^{Q,\phi})]=0.
\]
Letting $k\to\infty$ proves the identity at $(s,t)$; the cases $s=0$ or $t=T$
are handled by the same one-sided approximation.  Thus every
$Q\in\mathcal G$ solves the full joint martingale problem with generator
$\mathcal L_r^{Q_r}$; the identity is not merely averaged over $\Gamma$.

We now translate the joint martingale problem into the formulation of a weak
stopping rule.  Apply the identities first to compactly supported cutoffs
of the coordinate functions and their products, and then localize. Test
functions depending only on $w$, on $x$, and on the products
$w_aw_b$, $x_ax_b$, and $x_aw_b$ show that, under $Q$, $W$ is a continuous
local martingale with
$\langle W\rangle_t=tI_d$, and that
\[
    Y_t:=X_t-X_0-\int_0^t b(r,X_r,Q_r)I_r\,dr
\]
is a continuous local martingale satisfying
\[
    d\langle Y\rangle_r
       =\sigma\sigma^\top(r,X_r,Q_r)I_r\,dr,
    \qquad
    d\langle Y,W\rangle_r
       =\sigma(r,X_r,Q_r)I_r\,dr.
\]
L\'evy's characterization makes $W$ a Brownian motion with respect to the
completed canonical filtration.  The continuous local martingale
\[
    Y_t-\int_0^t\sigma(r,X_r,Q_r)I_r\,dW_r
\]
has zero quadratic variation and is therefore identically zero.  Hence
\[
    X_t=X_0+
      \int_0^t b(r,X_r,\Lc^Q(X_r,I_r))I_r\,dr
      +\int_0^t\sigma(r,X_r,\Lc^Q(X_r,I_r))I_r\,dW_r,
\]
$Q$-almost surely.  Together with the initial-law and constraint identities
proved above, this gives $Q\in\widetilde\Pc_W(0,\nu)$ for every
$Q\in\mathcal G$, and proves \eqref{eq:limit_law_supported_on_weak_rules}.

Finally, $\mathcal F_0$ is continuous on
$(\Pc_2(\Omega_0),\Wc_{2,0})$.  Let $Q_k\to Q$ in
$\Wc_{2,0}$ and use asymptotically optimal couplings.  The coupled path
coordinates converge in $L^2$, and \eqref{eq:integrated_marginal_bound} gives
convergence of the $(X_r,I_r)$-marginals in
$L^2([0,T];\Wc_2)$.  Along a subsequence the state, survival coordinate, and
marginal law therefore converge for almost every time.  The continuity of
$f(r,\cdot,\cdot)$ gives convergence of the integrand, while its quadratic
growth and $\Wc_{2,0}$ convergence give uniform integrability.  The subsequence
principle then yields
\[
    \mathcal F_0(Q_k)\longrightarrow\mathcal F_0(Q).
\]
On the Skorokhod representation space,
$\mathcal F_0(\mathbf P^N)\to\mathcal F_0(\mathbf P)$ almost surely.  Moreover,
\[
    |\mathcal F_0(\mathbf P^N)|
       \le C\left(1+
          \int_{\Omega_0}\|x\|_\infty^2\,\mathbf P^N(d\omega)
       \right),
\]
and Jensen's inequality together with
\eqref{eq:finite_uniform_path_moments} gives
\[
    \sup_N\E^{\P^0}
       [|\mathcal F_0(\mathbf P^N)|^{1+\delta/2}]<\infty.
\]
These random variables are therefore uniformly integrable, and
\eqref{eq:objective_convergence_random_empirical} follows.
\end{proof}

\begin{lemma}[Recovery sequence]\label{lem:finite_population_recovery}
Assume \eqref{eq:uniform_initial_higher_moment},
\eqref{eq:average_initial_law_convergence}, and
    $\E^\nu[I_{0-}]>\alpha$.
For every $\P\in\widetilde\Pc_W(0,\nu)$, there are admissible
finite-population strategies satisfying \eqref{eq:finite_constraint_pathwise}
such that
\begin{equation}\label{eq:recovery_empirical_convergence}
    \E^{\P^0}[\Wc_{2,0}^2(\mathbf P^N,\P)]\longrightarrow0
\end{equation}
and
\begin{equation}\label{eq:recovery_objective_convergence}
    J_N(\Ic^N)\longrightarrow\mathcal F_0(\P).
\end{equation}
\end{lemma}

The recovery construction rests on the following approximation result, which is the main technical difficulty in the limit theory compared to the usual argument. The standard argument is a combination of law of large numbers of i.i.d. copies of the stopping rules and corresponding uniformly-integrability estimate. With constraint, we need subtle modification to the stopping rules in $N$-particle system so that the constraint is always satisfied.
\begin{lemma}
\label{lem:finite_valued_triangular_approximation}
Assume \eqref{eq:uniform_initial_higher_moment} and
\eqref{eq:average_initial_law_convergence}, and let
$\P\in\widetilde\Pc_S(0,\nu)$.  There exist non-anticipative Borel maps
\[
    \varphi_m:\R^n\times\{0,1\}\times\Cc^d
       \longrightarrow\I([0,T]),
    \qquad m\ge1,
\]
with the following properties:
\noindent $\mathrm{(i)}$
For each $m$, $\vartheta_0(\varphi_m(x,i,w))$ takes values in a finite subset
of $[0,T]\cup\{T+1\}$, $\varphi_m(x,i,w)_{0-}=i$, and $\varphi_m$ is
continuous outside a set of
$\Lc^\P(X_0,I_{0-},W)$-measure zero.  If
\[
    I^m:=\varphi_m(X_0,I_{0-},W),
\]
$X^m$ is the McKean--Vlasov state driven by $W$, controlled by $I^m$, and using
its own marginal flow, and
\[
    \P^m:=\Lc^\P(X^m,W,I^m),
    \qquad
    p^m_r:=\E^\P[I^m_r],
\]
then
\begin{equation}\label{eq:finite_valued_rule_limit}
    \Wc_{2,0}(\P^m,\P)
      +\int_0^T|p^m_r-\E^\P[I_r]|\,dr
      +|p^m_T-\E^\P[I_T]|
    \longrightarrow0.
\end{equation}

\noindent$\mathrm{(ii)}$
For each fixed $m$, apply the same rule to the independent, not necessarily
identically distributed inputs and set
\[
    J^{i,m}:=\varphi_m(X^i_0,I^i_{0-},W^i).
\]
Let $\widehat X^{i,m}$ solve the SDE with initial state $X^i_0$, Brownian motion
$W^i$, control $J^{i,m}$, and the deterministic marginal flow
$\mu^m_r:=\Lc^\P(X^m_r,I^m_r)$.  Set
\[
    \widehat{\mathbf P}^{N,m}
      :=\frac1N\sum_{i=1}^N
           \delta_{(\widehat X^{i,m},W^i,J^{i,m})},
    \qquad
    \widehat\mu^{N,m}_r
      :=\frac1N\sum_{i=1}^N
           \delta_{(\widehat X^{i,m}_r,J^{i,m}_r)},
\]
and
\[
    \overline p^{N,m}_r
       :=\frac1N\sum_{i=1}^N\E^{\P^0}[J^{i,m}_r].
\]
Then, for every fixed $m$,
\begin{equation}\label{eq:fixed_m_triangular_approximation}
\begin{split}
    \E^{\P^0}\!
       [\Wc_{2,0}^2(\widehat{\mathbf P}^{N,m},\P^m)]
      +\E^{\P^0}\!\left[
         \int_0^T\Wc_2^2(\widehat\mu^{N,m}_r,\mu^m_r)\,dr
       \right]
      +\int_0^T|\overline p^{N,m}_r-p^m_r|\,dr
      +|\overline p^{N,m}_T-p^m_T|
      \longrightarrow0.
\end{split}
\end{equation}
Moreover, there is a nondecreasing sequence of integers $m_N\to\infty$ such
that
\begin{equation}\label{eq:triangular_diagonal_conclusion}
\begin{split}
    &\Wc_{2,0}^2(\P^{m_N},\P)
      +\E^{\P^0}\!
         [\Wc_{2,0}^2(\widehat{\mathbf P}^{N,m_N},\P^{m_N})]\\
    &\quad
      +\E^{\P^0}\!\left[
         \int_0^T
            \Wc_2^2(\widehat\mu^{N,m_N}_r,\mu^{m_N}_r)\,dr
       \right]
      +\int_0^T
          |\overline p^{N,m_N}_r-\E^\P[I_r]|\,dr
      +|\overline p^{N,m_N}_T-\E^\P[I_T]|
      \longrightarrow0.
\end{split}
\end{equation}
\end{lemma}

\begin{proof}
Since $\P$ is a strong stopping rule, $\vartheta_0(I)$ is a stopping time
for the filtration generated by $(X_0,I_{0-},W)$; equivalently, its level
events admit non-anticipative Borel representatives on the input space.

Choose partitions
\[
    0=t^m_0<t^m_1<\cdots<t^m_m=T,
    \qquad |\pi_m|\longrightarrow0.
\]
On $\{I_{0-}=1\}$, consider the disjoint events
\[
    \{\vartheta_0(I)=0\},
    \qquad
    \{t^m_{j-1}<\vartheta_0(I)\le t^m_j\},
    \quad 1\le j\le m.
\]
The $j$th event is measurable with respect to the sigma-field generated by
$X_0$, $I_{0-}$, and the Brownian path up to $t^m_j$.  By the Doob--Dynkin
lemma, choose a Borel subset $B^m_j$ of the corresponding Polish prefix space
whose pullback agrees almost surely with that event.
By regularity of Borel probability measures, $B^m_j$ can be approximated by a
continuity set $A^m_j$ in the same prefix space.  Choose the approximations so
that the pullback of every
$\partial A^m_j$ has zero $\Lc^\P(X_0,I_{0-},W)$-measure and
\[
    \sum_{j=0}^m
       \P\!\left(
          \{(X_0,I_{0-},W_{\cdot\wedge t^m_j})\in A^m_j\}
          \mathbin\triangle
          \{(X_0,I_{0-},W_{\cdot\wedge t^m_j})\in B^m_j\}
       \right)
    \le m^{-2}.
\]
Let the approximating stopping parameter be $0$ when the initial
indicator is zero, and otherwise to be the first value among
$0,t^m_1,\ldots,t^m_m$ whose corresponding prefix belongs to $A^m_j$; use
$T+1$ when no such event occurs.  This is a finite-valued stopping time.  The
associated map $\varphi_m$ is continuous away from the union of the pullbacks
of the boundaries $\partial A^m_j$.

On the event that all approximating events agree with the corresponding
original events, the approximating stopping parameter differs from the
original by at most $|\pi_m|$ on $[0,T]$, while $T+1$ is preserved.  Since the
stopping-parameter space is bounded,
\begin{equation}\label{eq:finite_rule_control_error}
    \E^\P[d_{\I,0}(I^m,I)^2]
       \le |\pi_m|^2+(T+1)^2m^{-2}
       \longrightarrow0.
\end{equation}
Thus the construction does not identify stopping at $T$ with
survival through $T$.

The uniform higher-moment assumption and
\eqref{eq:average_initial_law_convergence} imply that $\nu$ has a finite
$(2+\delta)$-moment.  Couple the equations for $X^m$ and $X$ using the same
initial variable and Brownian motion, and couple their measure arguments by
$(X^m,I^m)$ and $(X,I)$.  The Lipschitz property of $b$ and $\sigma$, the
Burkholder--Davis--Gundy inequality, H\"older's inequality, and Gronwall's
lemma give
\[
    \E^\P\!\left[\sup_{0\le r\le T}|X^m_r-X_r|^2\right]
       \le C
       \big(\E^\P[d_{\I,0}(I^m,I)^2]\big)^{\delta/(2+\delta)}
       \longrightarrow0.
\]
Together with \eqref{eq:finite_rule_control_error}, this proves the Wasserstein
convergence in \eqref{eq:finite_valued_rule_limit}.  The two survival-mass
convergences follow directly from
\[
    \int_0^T|I^m_r-I_r|\,dr\le d_{\I,0}(I^m,I),
    \qquad
    |I^m_T-I_T|\le d_{\I,0}(I^m,I).
\]

Fix $m$.  Since $W$ is Brownian with respect to a filtration containing
$(X_0,I_{0-})$, the law of $(X_0,I_{0-},W)$ under $\P$ is the product of
$\nu$ and Wiener measure.  Let an initial pair with law
$N^{-1}\sum_{i=1}^N\nu_i$ be optimally coupled with one of law $\nu$, and use
the same independent Brownian motion for the two equations.
The coupled initial pairs converge in $L^2$.  Since $\varphi_m$ is continuous
outside a null set for the limiting input law, the extended continuous mapping
theorem gives convergence in probability of the corresponding survival paths
under $d_{\I,0}$.  The distance is bounded, so this convergence also holds in
every finite $L^q$.  Stability of the SDE then gives convergence in
$L^2$ of the state paths.  The law of the solution started from the
mixture $N^{-1}\sum_i\nu_i$ is exactly the average of the laws of the
solutions started from the individual $\nu_i$.  It follows that
\begin{equation}\label{eq:fixed_m_average_law_convergence}
    \Wc_{2,0}\!\left(
       \frac1N\sum_{i=1}^N
          \Lc^{\P^0}(\widehat X^{i,m},W^i,J^{i,m}),
       \P^m
    \right)\longrightarrow0.
\end{equation}
This argument uses only convergence of the average initial law; the individual
laws $\nu_i$ need not be equal.

For fixed $m$, the triples
$(\widehat X^{i,m},W^i,J^{i,m})$, $1\le i\le N$, are independent.  The averaged
$(2+\delta)$-moment estimate for the equations is uniform in $N$.
Lemma \ref{lem:triangular_wasserstein_lln}, together with
\eqref{eq:fixed_m_average_law_convergence}, therefore yields the first term of
\eqref{eq:fixed_m_triangular_approximation}.  The integrated marginal term
follows from \eqref{eq:integrated_marginal_bound}.  The last two terms follow
from \eqref{eq:fixed_m_average_law_convergence}, the inequality
$\int_0^T|i_r-i'_r|dr\le d_{\I,0}(i,i')$, and the continuity of
$i\mapsto i_T$ under $d_{\I,0}$.

A diagonal sequence can be chosen as follows.  Pass to and relabel a
subsequence so that the
left-hand side of
\eqref{eq:finite_valued_rule_limit} is at most $1/m$.  For each $m$, choose
$N_m>N_{m-1}$ such that the left-hand side of
\eqref{eq:fixed_m_triangular_approximation} is at most $1/m$ for every
$N\ge N_m$.  Set $m_N=m$ for $N_m\le N<N_{m+1}$.  Then $m_N\to\infty$, and
the triangle inequality gives \eqref{eq:triangular_diagonal_conclusion}.
\end{proof}

\begin{proof}[Proof of Lemma \ref{lem:finite_population_recovery}]
First suppose that $\P\in\widetilde\Pc_S(0,\nu)$.  Apply Lemma
\ref{lem:finite_valued_triangular_approximation} and use the diagonal $m_N$
from \eqref{eq:triangular_diagonal_conclusion}.  For readability, write
$J^i$, $\widehat X^i$, $\widehat{\mathbf P}^N$,
$\widehat\mu^N$, $\overline p_N$, and $\P^{m_N}$ for the corresponding
objects, and set
\[
    p_r:=\E^\P[I_r].
\]
Then
\begin{equation}\label{eq:diagonal_recovery_errors}
\begin{split}
    &\Wc_{2,0}^2(\P^{m_N},\P)
      +\E^{\P^0}[\Wc_{2,0}^2(\widehat{\mathbf P}^N,\P^{m_N})]\\
    &\quad
      +\E^{\P^0}\!\left[
         \int_0^T\Wc_2^2(\widehat\mu^N_r,
             \Lc^{\P^{m_N}}(X_r,I_r))\,dr
       \right]
      +\int_0^T|\overline p_N(r)-p_r|\,dr
      +|\overline p_N(T)-p_T|
      \longrightarrow0.
\end{split}
\end{equation}

The map $r\mapsto p_r$ is non-increasing and right-continuous.  Since
$\P\in\widetilde\Pc_W(0,\nu)$,
\[
    p_r\in\{0\}\cup[\alpha,1],
    \qquad 0\le r<T.
\]
Let
\[
    \tau:=\inf\{r\in[0,T]:p_r=0\},
\]
with $\inf\emptyset=T+1$.  Then $p_r\ge\alpha$ for $r<\tau$ and $r<T$, while
$p_r=0$ for $r\ge\tau$.

Set $k_N:=\lceil\alpha N\rceil$ and
\[
    G_N:=\left\{\sum_{i=1}^N I^i_{0-}\ge k_N\right\}.
\]
Since
$N^{-1}\sum_{i=1}^N\E^{\P^0}[I^i_{0-}]\to\E^\nu[I_{0-}]>\alpha$ and the
indicators are independent, Chebyshev's inequality yields
$\P^0(G_N^c)\to0$.  On $G_N$, let
\[
    \rho_N:=\inf\left\{r<\tau\wedge T:
                  \sum_{i=1}^N J^i_r<k_N\right\},
\]
with $\inf\emptyset=+\infty$.  Since the $J^i$ are finite-valued
adapted survival paths, $\rho_N$ is an $\mathbb F^N$-stopping time.  On
$\{\rho_N<\tau\wedge T\}$, choose an
$\Fc^N_{\rho_N}$-measurable set $A_N\subset\{1,\ldots,N\}$ satisfying
\[
    \#A_N=k_N,
    \qquad
    \{i:J^i_{\rho_N}=1\}\subset A_N
       \subset\{i:J^i_{\rho_N-}=1\}.
\]
Such a set exists because at least $k_N$ particles are alive immediately
before the first strict crossing, whereas fewer than $k_N$ remain alive at the
crossing; select any additional indices in increasing order.  When
$\rho_N=0$, the left limit refers to the prescribed value at $0-$.

Set $I^{i,N}_{0-}:=I^i_{0-}$.  On $G_N^c$, set $I^{i,N}_r:=0$ for
$0\le r\le T$.  On $G_N$, define, for $0\le r<T$,
\[
    I^{i,N}_r:=
    \begin{cases}
       J^i_r,
       & r<\rho_N\wedge\tau,\\
       \mathbf 1_{\{i\in A_N\}},
       & \rho_N\le r<\tau,
         \quad\rho_N<\tau\wedge T,\\
       0,
       & \tau\le r<T,
    \end{cases}
\]
and set
\[
    I^{i,N}_T:=
    \begin{cases}
       0, & \tau<T,\\
       J^i_T, & \tau\ge T.
    \end{cases}
\]
The second case in the first display does not arise when $\rho_N=+\infty$.  The
value at $T$ is kept equal to $J^i_T$ when $\tau\ge T$ because the integral
constraint does not impose a condition at the single time $T$, while the metric
$d_{\I,0}$ distinguishes stopping at $T$ from survival through $T$.  By the
choice of $A_N$, no path is changed from $0$ to $1$ at $\rho_N$; at time $T$,
$J^i_T\le\mathbf 1_{\{i\in A_N\}}$ whenever the second case was used.
Thus the paths are adapted, right-continuous, and non-increasing.  Moreover,
\[
    \frac1N\sum_{i=1}^N I^{i,N}_r
      \in\{0\}\cup[\alpha,1],
    \qquad 0\le r<T.
\]

Let $M_N(r):=N^{-1}\sum_iJ^i_r$.  On $G_N$, for $r<\tau\wedge T$,
\[
    \frac1N\sum_{i=1}^N|I^{i,N}_r-J^i_r|
       =\left(\frac{k_N}{N}-M_N(r)\right)^+,
\]
and, if $\tau<T$, the left-hand side equals $M_N(r)$ for $r\ge\tau$.  Since the
$J^i$ are independent,
\[
    \E^{\P^0}[|M_N(r)-\overline p_N(r)|]\le \frac1{2\sqrt N}.
\]
Using $k_N/N\le\alpha+1/N$, $p_r\ge\alpha$ before $\tau$, and $p_r=0$ after
$\tau$, we obtain
\begin{equation}\label{eq:control_modification_l1}
\begin{split}
    \E^{\P^0}\!\left[
       \frac1N\sum_{i=1}^N\int_0^T
          |I^{i,N}_r-J^i_r|\,dr
    \right]
    \le{}T\P^0(G_N^c)+\frac TN+\frac{T}{2\sqrt N}
    +2\int_0^T|\overline p_N(r)-p_r|\,dr
      \longrightarrow0.
\end{split}
\end{equation}
The terminal time convention also gives
\begin{equation}\label{eq:control_modification_endpoint}
    \E^{\P^0}\!\left[
       \frac1N\sum_{i=1}^N|I^{i,N}_T-J^i_T|
    \right]
    \le \P^0(G_N^c)+\mathbf 1_{\{\tau<T\}}\overline p_N(T)
    \longrightarrow0.
\end{equation}
For two survival paths with the same value at $0-$,
\[
    d_{\I,0}(i,j)
    \le \int_0^T|i_r-j_r|\,dr+|i_T-j_T|.
\]
Since $d_{\I,0}$ is bounded, \eqref{eq:control_modification_l1} and
\eqref{eq:control_modification_endpoint} imply
\begin{equation}\label{eq:control_modification_metric}
    \E^{\P^0}\!\left[
       \frac1N\sum_{i=1}^N
          d_{\I,0}(I^{i,N},J^i)^2
    \right]\longrightarrow0.
\end{equation}

Let $X^{i,N}$ be the interacting system controlled by $I^{i,N}$.  Couple
$X^{i,N}$ with $\widehat X^i$ using the same initial variable and Brownian
motion.  For $t\in[0,T]$, the Lipschitz property of $b$ and $\sigma$, the
Burkholder--Davis--Gundy inequality, and the indexwise coupling of the empirical
measures give
\begin{align*}
    &\E^{\P^0}\!\left[
       \frac1N\sum_{i=1}^N
          \sup_{0\le r\le t}|X^{i,N}_r-\widehat X^i_r|^2
    \right]\\
    &\quad\le
       C\int_0^t\E^{\P^0}\!\left[
          \frac1N\sum_{i=1}^N
             \sup_{0\le u\le r}|X^{i,N}_u-\widehat X^i_u|^2
       \right]dr
       +C\E^{\P^0}\!\left[
          \int_0^T
          \Wc_2^2\!\left(
             \widehat\mu^N_r,
             \Lc^{\P^{m_N}}(X_r,I_r)
          \right)dr
       \right]\\
    &\qquad
       +C\E^{\P^0}\!\left[
          \frac1N\sum_{i=1}^N\int_0^T
             |I^{i,N}_r-J^i_r|\,dr
       \right]
       +C\left(
          \E^{\P^0}\!\left[
             \frac1N\sum_{i=1}^N\int_0^T
                |I^{i,N}_r-J^i_r|\,dr
          \right]
       \right)^{\delta/(2+\delta)}.
\end{align*}
All moment constants in this estimate are uniform in $m_N$ and $N$, because
the controls are bounded and the averaged initial $(2+\delta)$-moments are
uniformly bounded.  The second term tends to zero by
\eqref{eq:diagonal_recovery_errors}; the last two terms tend to zero by
\eqref{eq:control_modification_l1}.  Gronwall's lemma
therefore yields
\begin{equation}\label{eq:interacting_independent_state_difference}
    \E^{\P^0}\!\left[
       \frac1N\sum_{i=1}^N
          \sup_{0\le r\le T}|X^{i,N}_r-\widehat X^i_r|^2
    \right]\longrightarrow0.
\end{equation}
Combining \eqref{eq:diagonal_recovery_errors},
\eqref{eq:control_modification_metric}, and
\eqref{eq:interacting_independent_state_difference} with the indexwise
coupling proves \eqref{eq:recovery_empirical_convergence}.  Since the empirical
survival mass belongs to $\{0\}\cup[\alpha,1]$ for every $r<T$, the threshold
indicator in the finite-population objective may be removed.  The continuity
and uniform-integrability argument used at the end of Lemma
\ref{lem:finite_population_compactness_identification} then gives
\eqref{eq:recovery_objective_convergence}.

For a general $\P\in\widetilde\Pc_W(0,\nu)$, Theorem
\ref{thm:equivalence_value_function}\textnormal{(ii)} allows us to choose
$\P^\ell\in\widetilde\Pc_S(0,\nu)$ such that
$\Wc_{2,0}(\P^\ell,\P)\to0$.  The continuity of $\mathcal F_0$ gives
$\mathcal F_0(\P^\ell)\to\mathcal F_0(\P)$.  For each $\ell$, the strong-case
construction gives strategies $\Ic^{N,\ell}$, with random empirical path
laws denoted by $\mathbf P^{N,\ell}$, for which
\[
    \E^{\P^0}[\Wc_{2,0}^2(\mathbf P^{N,\ell},\P^\ell)]
      +|J_N(\Ic^{N,\ell})-\mathcal F_0(\P^\ell)|
      \longrightarrow0
      \quad\hbox{as }N\to\infty.
\]
Choose $N_\ell>N_{\ell-1}$ so that the last display is at most $1/\ell$ for
all $N\ge N_\ell$, and use $\Ic^{N,\ell}$ when
$N_\ell\le N<N_{\ell+1}$.  The triangle inequality and
$\P^\ell\to\P$ prove \eqref{eq:recovery_empirical_convergence}, while the same
choice and $\mathcal F_0(\P^\ell)\to\mathcal F_0(\P)$ prove
\eqref{eq:recovery_objective_convergence}.
\end{proof}

\begin{proposition}[Subthreshold initial mass]\label{prop:subthreshold_limit}
Assume \eqref{eq:uniform_initial_higher_moment} and
\eqref{eq:average_initial_law_convergence}.  If $p(\nu)<\alpha$, then
\[
    \lim_{N\to\infty}V^N(\nu_1,\ldots,\nu_N)=0=V(0,\nu).
\]
\end{proposition}

\begin{proof}
For every weak stopping rule starting from $\nu$,
$\E[I_r]\le p(\nu)<\alpha$ for all $r\in[0,T]$.  Hence its thresholded
objective is zero, and therefore $V(0,\nu)=0$.

For the finite system, $q^N_r\le q^N_{0-}$ for every $r$.  Thus the finite
objective can be nonzero only on
\[
    G_N:=\left\{\frac1N\sum_{i=1}^NI^i_{0-}\ge\alpha\right\}.
\]
The expectation of the average inside $G_N$ converges to $p(\nu)$, and its
variance is at most $1/(4N)$.  Since $p(\nu)<\alpha$, Chebyshev's inequality
gives $\P^0(G_N)\to0$.  The moment estimate for the interacting system and the
growth bound on $f$ yield, uniformly over all finite-population strategies,
\[
    \E^{\P^0}\!\left[
       \left(1+\frac1N\sum_{i=1}^N\|X^{i,N}\|_\infty^2\right)^{1+\delta/2}
    \right]\le C.
\]
H\"older's inequality therefore gives
\[
    \sup_{\Ic^N}|J_N(\Ic^N)|
    \le C\P^0(G_N)^{\delta/(2+\delta)}\longrightarrow0.
\]
Stopping every particle at time $0$ gives payoff zero, so the values converge
to zero.
\end{proof}

\begin{theorem}[Finite-population limit]\label{thm:limit_theory}
Assume that, for some $\delta>0$,
\begin{equation}
    \sup_{N\ge1}\frac1N\sum_{i=1}^N
       \int_{\Sbf}|x|^{2+\delta}\,\nu_i(dx,di)<\infty,
       \qquad
    \Wc_2\!\left(\frac1N\sum_{i=1}^N\nu_i,\nu\right)
    \longrightarrow0,
\end{equation}
and that $\E^\nu[I_{0-}]>\alpha$.  Then
\begin{equation}\label{eq:value_limit}
    \lim_{N\to\infty}V^N(\nu_1,\ldots,\nu_N)=V(0,\nu).
\end{equation}

Let $\eps_N\downarrow0$ and let $\Ic^N$ be $\eps_N$-optimal for the
finite-population problem.  Replace $\Ic^N$ by the cutoff strategy from Lemma
\ref{lem:finite_population_cutoff}, and let $\mathbf P^N$ be its random
empirical path law.  The laws of $\mathbf P^N$ are tight on
$(\Pc_2(\Omega_0),\Wc_{2,0})$, and every limit law $\Gamma$ satisfies
\begin{equation}\label{eq:optimal_limit_support}
    \Gamma\!\left(
       \left\{Q\in\widetilde\Pc_W(0,\nu):
          \mathcal F_0(Q)=V(0,\nu)\right\}
    \right)=1.
\end{equation}

Conversely, for every
$\P^*\in\widetilde\Pc_W(0,\nu)$ satisfying
$\mathcal F_0(\P^*)=V(0,\nu)$, there is a sequence of admissible
finite-population strategies such that
\begin{equation}\label{eq:optimizer_recovery}
    \E^{\P^0}[\Wc_{2,0}^2(\mathbf P^N,\P^*)]\longrightarrow0,
    \qquad
    V^N(\nu_1,\ldots,\nu_N)-J_N(\Ic^N)\longrightarrow0.
\end{equation}
\end{theorem}

\begin{proof}
Applying Lemma \ref{lem:finite_population_recovery} to each element of
$\widetilde\Pc_W(0,\nu)$ and then taking the supremum gives
\[
    V(0,\nu)\le\liminf_{N\to\infty}V^N(\nu_1,\ldots,\nu_N).
\]
For the upper bound, choose $N^{-1}$-optimal finite-population
strategies and apply the cutoff lemma.  Along a subsequence realizing the
limsup, Lemma \ref{lem:finite_population_compactness_identification} gives a
limit law $\Gamma$ concentrated on $\widetilde\Pc_W(0,\nu)$, and
\[
\begin{split}
    \limsup_{N\to\infty}V^N(\nu_1,\ldots,\nu_N)
    &=\int\mathcal F_0(Q)\,\Gamma(dQ)\\
    &\le V(0,\nu).
\end{split}
\]
This proves \eqref{eq:value_limit}.

For an $\eps_N$-optimal sequence, every subsequential limit $\Gamma$ satisfies
\[
    \int\mathcal F_0(Q)\,\Gamma(dQ)=V(0,\nu),
    \qquad
    \mathcal F_0(Q)\le V(0,\nu)
    \quad\Gamma\hbox{-a.s.}
\]
Hence equality holds $\Gamma$-almost surely, proving
\eqref{eq:optimal_limit_support}.  Applying Lemma
\ref{lem:finite_population_recovery} to $\P^*$ and using
\eqref{eq:value_limit} gives \eqref{eq:optimizer_recovery}.
\end{proof}

\begin{remark}
The probability measure
\[
    \frac1N\sum_{i=1}^N
       \Lc^{\P^0}(X^{i,N},W^i,I^{i,N})
    =\E^{\P^0}[\mathbf P^N]
\]
need not be a weak stopping rule when the finite-population controls may use the
whole filtration $\mathbb F^N$.  Such controls can use one particle's Brownian
motion as a common random signal.  Conditional on that signal, the empirical
measure may converge to one of several admissible weak laws, while its
unconditional average generally does not solve the McKean--Vlasov equation
with its own marginal flow.  This is why Theorem
\ref{thm:limit_theory} is stated in terms of the law of the random
empirical measure and why its limit is a probability measure supported on weak
optimizers.
\end{remark}

\begin{remark}[The threshold case]
The strict inequality $p(\nu)>\alpha$ cannot in general be replaced by
$p(\nu)=\alpha$.  Suppose $\alpha\in(0,1)$, $b=\sigma=0$, $f\equiv1$, and
\[
    \nu=\alpha\delta_{(0,1)}+(1-\alpha)\delta_{(0,0)},
\]
with i.i.d. initial pairs of law $\nu$.  The mean-field value is
$V(0,\nu)=\alpha T$.  In the finite system it is optimal to keep every initially
alive particle alive, and therefore
\[
    V^N
    =T\,\E\!\left[
       \frac1N\sum_{i=1}^NI^i_{0-}
       \mathbf 1_{[\alpha,1]}\!\left(
          \frac1N\sum_{i=1}^NI^i_{0-}
       \right)
    \right]
    \longrightarrow \frac{\alpha T}{2}
\]
by the central limit theorem.  Thus no unconditional value-convergence theorem
is valid at the threshold without an additional one-sided feasibility
assumption on the initial empirical survival mass.
\end{remark}
\section{Dependence of the value function on the threshold}

For $\alpha\in[0,1]$ and $A\in\{W,S\}$, define
\begin{align}
    V_A^\alpha(t,\nu)
    &:={\sup}_{\P\in\Pc_A(t,\nu)}
      \E^\P\!\left[
        \int_t^T f(s,X_s,\mu_s^\P)I_s
        \mathbf 1_{[\alpha,1]}\!\left(\E^\P[I_s]\right)ds
      \right],
      \label{eq:closed_alpha_value}
      \\
    V_{\alpha,A}(t,\nu)
    &:={\sup}_{\P\in\Pc_A(t,\nu)}
      \E^\P\!\left[
        \int_t^T f(s,X_s,\mu_s^\P)I_s
        \mathbf 1_{(\alpha,1]}\!\left(\E^\P[I_s]\right)ds
      \right].
      \label{eq:open_alpha_value}
\end{align}
Throughout this section, write $\mu_s^\P=\Lc^\P(X_s,I_s)$ and
$p_s^\P:=\E^\P[I_s]$.  Applying Theorem
\ref{thm:equivalence_value_function} at each threshold gives
$V_W^\alpha=V_S^\alpha$; we denote their common value by $V^\alpha$.
At $\alpha=0$, the two indicators in
\eqref{eq:closed_alpha_value}--\eqref{eq:open_alpha_value} yield the same
objective because $p_s^\P=0$ implies $I_s=0$, $\P$-a.s.

Lemma \ref{lem:constraint_reduction} has the following open-threshold
analogue.

\begin{lemma}
\label{lem:open_constraint_reduction}
For $A\in\{W,S\}$,
\begin{equation}
\label{eq:open_constraint_reduction}
    V_{\alpha,A}(t,\nu)
    =
    \sup\left\{
       \mathcal F_t(\P^t):
       \P\in\Pc_A(t,\nu),\quad
       p_s^\P\in\{0\}\cup(\alpha,1]
       \text{ for every }s\in[t,T]
    \right\}.
\end{equation}
Likewise, Lemma \ref{lem:constraint_reduction} may be written in the
pointwise form
\begin{equation}
\label{eq:closed_constraint_pointwise}
    V_A^\alpha(t,\nu)
    =
    \sup\left\{
       \mathcal F_t(\P^t):
       \P\in\Pc_A(t,\nu),\quad
       p_s^\P\in\{0\}\cup[\alpha,1]
       \text{ for every }s\in[t,T]
    \right\}.
\end{equation}
\end{lemma}

\begin{proof}
We prove \eqref{eq:open_constraint_reduction}.  The same cutoff argument,
with ``$\le\alpha$'' replaced by ``$<\alpha$'', gives
\eqref{eq:closed_constraint_pointwise}.

Fix $\P\in\Pc_A(t,\nu)$.  Since $s\mapsto p_s^\P$ is nonincreasing and
right-continuous, the time
\[
    \tau:=\inf\{s\in[t,T]:p_s^\P\le\alpha\},
\]
with the convention $\inf\varnothing=T+1$, is deterministic.  Stop every
remaining particle at $\tau$, and let the state solve
\eqref{eq:dynamics} with this modified survival process and the same initial
state and Brownian motion.  The modified and original systems agree before
$\tau$.  From $\tau$ onward the open-threshold indicator in the original
objective is zero, while the modified survival mass is zero.  Hence the
objective is unchanged.  The modified survival mass belongs to
$\{0\}\cup(\alpha,1]$ at every time.  A deterministic cutoff preserves strong
adaptedness, so the argument applies to both $A=W$ and $A=S$.
\end{proof}

Small perturbations of the survival process satisfy the following stability
property.

\begin{lemma}
\label{lem:small_survival_perturbation}
For each $k$, let $(X^k,W^k,I^k)$ and
$(\bar X^k,W^k,\bar I^k)$ be two coupled weak solutions on a common
probability space $(\Omega^k,\mathcal H^k,\mathbb Q_k)$.  Assume that
$(X_t^k,I_{t-}^k)=(\bar X_t^k,\bar I_{t-}^k)$, $\mathbb Q_k$-a.s., and that
this common initial pair has law $\nu$.  Suppose also that both solutions are
driven by the same Brownian motion.  Let $\P_k$ and $\bar\P_k$ be their
canonical laws, and assume
that
\[
    \mathbb Q_k\big(I^k_{|\{t-\}\cup[t,T]}
       \ne \bar I^k_{|\{t-\}\cup[t,T]}\big)\longrightarrow0.
\]
Then
\[
    \Wc_{2,t}(\P_k^t,\bar\P_k^t)\longrightarrow0.
\]
\end{lemma}

\begin{proof}
Let $A_k$ be the event in the statement.  The Lipschitz assumptions on
$b$ and $\sigma$, the BDG inequality, and the coupling bound for the law
argument give
\begin{align*}
    \E^{\mathbb Q_k}\!\left[\sup_{u\in[t,s]}|X_u^k-\bar X_u^k|^2\right]
    \le{}&
    C\int_t^s
      \E^{\mathbb Q_k}\!\left[\sup_{v\in[t,r]}|X_v^k-\bar X_v^k|^2\right]dr
    \\
    &+C\E^{\mathbb Q_k}\!\left[
       \left(1+\|X^k\|_{\infty,[t,T]}^2
               +\|\bar X^k\|_{\infty,[t,T]}^2\right)
       \mathbf 1_{A_k}
    \right].
\end{align*}
Indeed,
\[
    \Wc_2^2\!\left(
       \Lc(X_r^k,I_r^k),\Lc(\bar X_r^k,\bar I_r^k)
    \right)
    \le
    \E^{\mathbb Q_k}\!\left[|X_r^k-\bar X_r^k|^2+|I_r^k-\bar I_r^k|\right],
\]
and $|I_r^k-\bar I_r^k|\le\mathbf 1_{A_k}$.  All the projected laws in this
argument belong to the compact set $\mathfrak K(t,\nu)$; hence their squared
path norms are uniformly integrable.  The last term therefore tends to zero.
Gronwall's lemma yields
\[
    \E^{\mathbb Q_k}\!\left[\sup_{s\in[t,T]}|X_s^k-\bar X_s^k|^2\right]\longrightarrow0.
\]
Moreover, $d_{\I,t}(I^k,\bar I^k)$ is bounded by a deterministic constant on
$A_k$ and vanishes on $A_k^c$.  The coupling of the two systems therefore
gives
$\Wc_{2,t}(\P_k^t,\bar\P_k^t)\to0$.
\end{proof}

\begin{theorem}
\label{thm:open_strong_weak_equivalence}
Under Assumption \ref{ass:coefficient_reward}, for every
$(t,\nu,\alpha)\in[0,T]\times\Pc_2(\Sbf)\times[0,1]$,
\[
    V_{\alpha,S}(t,\nu)=V_{\alpha,W}(t,\nu).
\]
Their common value is denoted by $V_\alpha(t,\nu)$.
\end{theorem}

\begin{proof}
Clearly, $V_{\alpha,S}\le V_{\alpha,W}$.  If
$p(\nu)\le\alpha$, then $p_s^\P\le p(\nu)\le\alpha$ for every admissible law,
so both values are zero; the case $t=T$ is equally clear.  We may therefore
assume that $t<T$ and $p(\nu)>\alpha$.

Fix a weak law from the class on the right-hand side of
\eqref{eq:open_constraint_reduction}, and denote it by $\P$.  Let
\[
    \zeta:=\inf\{s\in[t,T]:p_s^\P=0\},
\]
with $\zeta=T+1$ if the set is empty.  Enlarge the time-$t$ sigma-field by a
random variable $U\sim\mathrm{Unif}(0,1)$ independent of the original
system.  The Brownian increments after $t$ remain independent of the enlarged
initial sigma-field, so Brownianity is preserved in the enlarged filtration.
For $q\in(0,1)$, set $I_{t-}^q=I_{t-}$ and
\[
    I_s^q
    :=I_s\vee
      \left(I_{t-}\mathbf 1_{\{U\le q\}}\mathbf 1_{\{s<\zeta\}}\right),
    \qquad s\in[t,T].
\]
Let $X^q$ solve \eqref{eq:dynamics} with $I^q$, and denote the resulting
canonical law by $\P^q$.  For $s<\zeta$,
\[
    \E[I_s^q]
    =(1-q)p_s^\P+q p(\nu)
    >\alpha+q\big(p(\nu)-\alpha\big),
\]
whereas $\E[I_s^q]=0$ for $s\ge\zeta$.  Hence $\P^q$ is admissible for the
closed constraint at level
$\alpha+q(p(\nu)-\alpha)$.

The paths $I^q$ and $I$ differ with probability at most $q p(\nu)$.
Lemma \ref{lem:small_survival_perturbation} therefore gives
$\Wc_{2,t}((\P^q)^t,\P^t)\to0$ as $q\downarrow0$.  By Lemma
\ref{lem:reward_lsc},
\[
    \mathcal F_t(\P^t)
    \le \liminf_{q\downarrow0}\mathcal F_t((\P^q)^t).
\]
For fixed $q$, Theorem \ref{thm:equivalence_value_function}\textnormal{(ii)},
applied at the threshold $\alpha+q(p(\nu)-\alpha)$, provides strong laws
admissible for that closed constraint whose projections on $\Om_t$ converge to
$(\P^q)^t$.
Those laws also satisfy the open constraint at level $\alpha$.  Applying Lemma
\ref{lem:reward_lsc} once more gives
\[
    V_{\alpha,S}(t,\nu)\ge \mathcal F_t((\P^q)^t).
\]
Letting $q\downarrow0$ and then taking the supremum over $\P$ proves the reverse
inequality.
\end{proof}

\begin{proposition}[Dynamic programming for the open threshold]
\label{prop:open_threshold_dpp}
Under Assumption \ref{ass:coefficient_reward}, for $t\le s\le T$,
\begin{align}
    V_{\alpha,W}(t,\nu)
    ={}&\sup_{\P\in\Pc_W(t,\nu)}
       \left\{
       \E^\P\!\left[
          \int_t^s f(r,X_r,\mu_r^\P)I_r
          \mathbf 1_{(\alpha,1]}(p_r^\P)dr
       \right]
       +V_{\alpha,W}(s,m_{s-}^\P)
       \right\}
       \label{eq:dpp_open_pre}
       \\
    ={}&\sup_{\P\in\Pc_W(t,\nu)}
       \left\{
       \E^\P\!\left[
          \int_t^s f(r,X_r,\mu_r^\P)I_r
          \mathbf 1_{(\alpha,1]}(p_r^\P)dr
       \right]
       +V_{\alpha,W}(s,m_s^\P)
       \right\}.
       \label{eq:dpp_open_post}
\end{align}
The same identities hold with $W$ replaced by $S$.
\end{proposition}

\begin{proof}
Lemma \ref{lem:weak_concatenation} applies unchanged and splits the objective
at time $s$.  Lemma \ref{lem:instantaneous_stopping_monotonicity} then permits
$m_{s-}^\P$ to be replaced by $m_s^\P$.  Neither step depends on whether the
terminal time of the threshold interval is included.  Theorem
\ref{thm:open_strong_weak_equivalence} gives the strong identities by the same
argument as in Theorem \ref{thm:dpp}.
\end{proof}

\begin{theorem}[Dependence on the threshold]
\label{thm:alpha_dependence}
Suppose Assumption \ref{ass:coefficient_reward} holds, and fix
$(t,\nu)\in[0,T]\times\Pc_2(\Sbf)$.

\begin{enumerate}
\item[(i)] The maps $\alpha\mapsto V^\alpha(t,\nu)$ and
$\alpha\mapsto V_\alpha(t,\nu)$ are nonincreasing.  Moreover, whenever
$0\le\beta<\alpha\le1$,
\[
    V_\alpha(t,\nu)
    \le V^\alpha(t,\nu)
    \le V_\beta(t,\nu).
\]

\item[(ii)] For every $\alpha\in[0,1)$,
\[
    V_\alpha(t,\nu)
    =\lim_{\beta\downarrow\alpha}V_\beta(t,\nu)
    =\lim_{\beta\downarrow\alpha}V^\beta(t,\nu).
\]

\item[(iii)] If $\alpha\ne p(\nu)$, then
\[
    V_\alpha(t,\nu)=V^\alpha(t,\nu).
\]

\item[(iv)] Suppose in addition that, for every $r\in[0,T]$, the map
$(x,m)\mapsto f(r,x,m)$ is continuous on
$\R^n\times\Pc_2(\Sbf)$.  Then, for every $\alpha\in(0,1]$,
\[
    V^\alpha(t,\nu)
    =\lim_{\beta\uparrow\alpha}V_\beta(t,\nu)
    =\lim_{\beta\uparrow\alpha}V^\beta(t,\nu).
\]
Consequently, under this additional continuity assumption, both threshold
value functions are continuous at every $\alpha\ne p(\nu)$.  At
$\alpha=p(\nu)$, one always has $V_{p(\nu)}(t,\nu)=0$, and the only
possible discontinuity is the downward jump from $V^{p(\nu)}(t,\nu)$ to
$V_{p(\nu)}(t,\nu)$.
\end{enumerate}
\end{theorem}

\begin{proof}
By Lemma \ref{lem:open_constraint_reduction} and
\eqref{eq:closed_constraint_pointwise}, increasing the threshold only shrinks
the corresponding constrained class.  Moreover, the open class at level
$\alpha$ is contained in the closed class at the same level, which in turn is
contained in the open class at every lower level $\beta<\alpha$.  These
inclusions prove part \textnormal{(i)}.

Fix $\alpha<1$.  For every fixed admissible law $\P$, dominated convergence,
justified by \eqref{eq:f_growth} and Lemma \ref{lem:moment_estimate}, gives
\begin{align*}
    \lim_{\beta\downarrow\alpha}
      \E^\P\!\left[
        \int_t^T f(s,X_s,\mu_s^\P)I_s
        \mathbf 1_{(\beta,1]}(p_s^\P)ds
      \right]
    =
      \E^\P\!\left[
        \int_t^T f(s,X_s,\mu_s^\P)I_s
        \mathbf 1_{(\alpha,1]}(p_s^\P)ds
      \right],
\end{align*}
and the same limit holds when the interval $(\beta,1]$ in the first indicator
is replaced by $[\beta,1]$.  Evaluating both $\beta$-problems at an
$\varepsilon$-optimal law for $V_\alpha(t,\nu)$ gives a lower bound for each
right-hand limit.  Part \textnormal{(i)} gives, for
$\beta>\alpha$,
$V_\beta\le V_\alpha$ and $V^\beta\le V_\alpha$.  Letting
$\varepsilon\downarrow0$ proves part \textnormal{(ii)}.

For part \textnormal{(iii)}, suppose first that $\alpha>p(\nu)$.  Every
admissible survival mass is then at most $p(\nu)$, so both values vanish.  The
same is clear when $\alpha=0$ and $p(\nu)=0$.  We are left with the case
$\alpha<p(\nu)$.

Take a law $\P$ in the closed constrained class in
\eqref{eq:closed_constraint_pointwise}.  Let
$\zeta:=\inf\{s\in[t,T]:p_s^\P=0\}$, with $\zeta=T+1$ if the set is empty.
Using the independent time-$t$ randomization from the proof of Theorem
\ref{thm:open_strong_weak_equivalence}, define $I^q$ in the same way.  For
$s<\zeta$,
\[
    \E[I_s^q]=(1-q)p_s^\P+q p(\nu)>\alpha,
\]
and the modified survival mass is zero from $\zeta$ onward.  Thus the modified
law is admissible for the open threshold $\alpha$.  Lemma
\ref{lem:small_survival_perturbation} gives convergence of its projection on $\Om_t$ to
$\P^t$, and Lemma \ref{lem:reward_lsc} yields
\[
    V_\alpha(t,\nu)
    \ge \liminf_{q\downarrow0}\mathcal F_t((\P^q)^t)
    \ge \mathcal F_t(\P^t).
\]
Taking the supremum over the closed constrained class gives
$V_\alpha\ge V^\alpha$; the reverse inequality follows from part
\textnormal{(i)}.

We turn to part \textnormal{(iv)}.  If $\alpha>p(\nu)$, all values vanish for
thresholds sufficiently close to $\alpha$ from the left.  Assume instead that
$\alpha\le p(\nu)$ and let
$\beta_k\uparrow\alpha$.  Using \eqref{eq:open_constraint_reduction}, choose
$\P_k\in\Pc_W(t,\nu)$ such that
\begin{equation}
\label{eq:beta_near_optimizer}
    p_s^{\P_k}\in\{0\}\cup(\beta_k,1]
    \quad\text{for every }s,
    \qquad
    \mathcal F_t(\P_k^t)\ge V_{\beta_k}(t,\nu)-\frac1k.
\end{equation}
Set
\[
    \zeta_k:=\inf\{s\in[t,T]:p_s^{\P_k}\le\beta_k\},
    \qquad
    \tau_k:=\inf\{s\in[t,T]:p_s^{\P_k}\le\alpha\},
\]
with the value $T+1$ when the corresponding set is empty.  By the cutoff in
Lemma \ref{lem:open_constraint_reduction}, $p_s^{\P_k}=0$ for
$s\ge\zeta_k$.

If $\tau_k=\zeta_k$ or $\tau_k=T+1$, set $\bar\P_k=\P_k$.  Otherwise,
realize $\P_k$ on a stochastic basis with
coordinates $(X^k,W^k,I^k)$ and probability $\mathbb Q_k$, and enlarge the
time-$t$ sigma-field by a random variable $U_k\sim\mathrm{Unif}(0,1)$
independent of the original system.  Brownian increments after $t$ remain
independent of the enlarged initial sigma-field.  If
$p_{\tau_k-}^{\P_k}>p_{\tau_k}^{\P_k}$, set
\[
    r_k:=
    \frac{\alpha-p_{\tau_k}^{\P_k}}
         {p_{\tau_k-}^{\P_k}-p_{\tau_k}^{\P_k}},
\]
and set $r_k=0$ when the denominator is zero.  In the first case,
$p_{\tau_k}^{\P_k}\le\alpha\le p_{\tau_k-}^{\P_k}$, so $r_k\in[0,1]$.
The zero-denominator case necessarily has
$p_{\tau_k-}^{\P_k}=p_{\tau_k}^{\P_k}=\alpha$.  Define
\[
    K_k
    :=I^k_{\tau_k}
      +(I^k_{\tau_k-}-I^k_{\tau_k})\mathbf 1_{\{U_k\le r_k\}}.
\]
Then $K_k\in\{0,1\}$ and $K_k\le I^k_{\tau_k-}$.  Independence of $U_k$ gives
\[
    \E^{\mathbb Q_k}[K_k]
    =p_{\tau_k}^{\P_k}
      +r_k\big(p_{\tau_k-}^{\P_k}-p_{\tau_k}^{\P_k}\big)
    =\alpha,
\]
with the same conclusion in the zero-denominator case.  Define
\[
    \bar I_s^k
    :=
    \begin{cases}
       I_s^k, & t\le s<\tau_k,\\
       K_k, & \tau_k\le s<\zeta_k,\\
       0,   & \zeta_k\le s\le T,
    \end{cases}
    \qquad
    \bar I_{t-}^k:=I^k_{t-}.
\]
This is a valid nonanticipative survival path adapted to the filtration
enlarged by $U_k$.  Its survival mass is strictly larger than $\alpha$ before
$\tau_k$, equal to $\alpha$ on $[\tau_k,\zeta_k)$, and zero from $\zeta_k$
onward.  The law $\bar\P_k$ obtained by solving the state equation with
$\bar I^k$ is admissible for the closed constraint at level $\alpha$.

A path is altered only on the event $K_k=1$ when its original stopping time
lies in $[\tau_k,\zeta_k)$.  Consequently,
\[
    \mathbb Q_k(I^k\ne\bar I^k)
    =\E^{\mathbb Q_k}[K_k-I^k_{\zeta_k-}]
    =\alpha-p_{\zeta_k-}^{\P_k}
    \le\alpha-\beta_k,
\]
where, if $\zeta_k=T+1$, $I^k_{\zeta_k-}$ and
$p_{\zeta_k-}^{\P_k}$ are understood as $I^k_T$ and $p_T^{\P_k}$,
respectively; the inequality follows from
$p_{\zeta_k-}^{\P_k}\ge\beta_k$.  Hence the probability that $I^k$ and
$\bar I^k$ differ tends to zero, and Lemma
\ref{lem:small_survival_perturbation} gives
\[
    \Wc_{2,t}(\P_k^t,\bar\P_k^t)\longrightarrow0.
\]
Under the additional continuity assumption, Lemma
\ref{lem:projected_law_compactness} shows that $\mathcal F_t$ is continuous on the
compact set $\mathfrak K(t,\nu)$, hence uniformly continuous there.  Therefore
\[
    \mathcal F_t(\bar\P_k^t)-\mathcal F_t(\P_k^t)\longrightarrow0.
\]
Since $\bar\P_k$ is admissible for the closed constraint at level $\alpha$,
\[
    V^\alpha(t,\nu)
    \ge \mathcal F_t(\bar\P_k^t)
    \ge V_{\beta_k}(t,\nu)-\frac1k-o(1).
\]
Part \textnormal{(i)} gives the reverse inequality in the limit, so
\[
    \lim_{\beta\uparrow\alpha}V_\beta(t,\nu)=V^\alpha(t,\nu).
\]
In the nontrivial regime $\beta<\alpha\le p(\nu)$, part
\textnormal{(iii)} gives $V_\beta=V^\beta$.  The second equality in part
\textnormal{(iv)} follows.
\end{proof}

\begin{remark}
\label{rem:alpha_lsc_counterexample}
Part \textnormal{(iv)} is false under lower semicontinuity alone.  Fix
$a\in(0,1)$ and $p_0\in(a,1]$, take $t<T$, $b=\sigma=0$, and
\[
    \nu=p_0\delta_{(0,1)}+(1-p_0)\delta_{(0,0)},
    \qquad
    f(s,x,m):=\mathbf 1_{\{p(m)<a\}}.
\]
The reward is bounded and lower semicontinuous because $p$ is continuous.  For
every $\gamma<a$, instantaneously reducing the surviving mass to any
$q\in(\gamma,a)$ and then keeping it constant yields the reward $q(T-t)$.  No
strategy can give more than $a(T-t)$, and therefore
\[
    V_\gamma(t,\nu)=V^\gamma(t,\nu)=a(T-t),
    \qquad \gamma<a.
\]
At the threshold $a$, the reward is zero whenever the threshold indicator is
active, so
\[
    V_a(t,\nu)=V^a(t,\nu)=0.
\]
Thus a left discontinuity can occur at a threshold different from $p(\nu)$
when $f$ is merely lower semicontinuous.
\end{remark}

\begin{remark}
\label{rem:open_alpha_no_optimizer}
The nonclosed open constraint can destroy existence even when $f$ is
continuous.  Fix $0<\alpha<p_0\le1$, take $t<T$, $b=\sigma=0$,
\[
    \nu=p_0\delta_{(0,1)}+(1-p_0)\delta_{(0,0)},
    \qquad
    f(s,x,m):=2\alpha-p(m).
\]
For every admissible law, the instantaneous open-threshold reward is
\[
    p_s^\P(2\alpha-p_s^\P)
    \mathbf 1_{(\alpha,1]}(p_s^\P)
    \le \alpha^2,
\]
and the inequality is strict at every time.  On the other hand, keeping a
constant surviving mass $q>\alpha$ gives
$q(2\alpha-q)(T-t)\uparrow\alpha^2(T-t)$ as $q\downarrow\alpha$.
Consequently,
\[
    V_\alpha(t,\nu)=\alpha^2(T-t),
\]
but the supremum is not attained.  The corresponding closed-threshold value
is attained by keeping the surviving mass exactly equal to $\alpha$.
\end{remark}

\bibliography{references}
\bibliographystyle{plain}
\end{document}